%% file: main.tex
\documentclass[12pt]{article}

\input{packages}
\input{theorem}

\input{title}

\input{paddings}

\input{commands}

\newenvironment{sciabstract}{}

\allowdisplaybreaks

\begin{document} 
\maketitle

\begin{sciabstract}
  \textbf{Abstract.}
   In a setting of a complex manifold with a fixed positive line bundle and a submanifold, we consider the optimal Ohsawa-Takegoshi extension operator, sending a holomorphic section of the line bundle on the submanifold to the holomorphic extension of it on the ambient manifold with the minimal $L^2$-norm.
 	We show that for a tower of submanifolds in the semiclassical setting, i.e. when we consider a large tensor power of the line bundle, the extension operators satisfy transitivity property modulo some small defect, which can be expressed through Toeplitz type operators.
 	We calculate the first significant term in the asymptotic expansion of this “transitivity defect".
 	As a byproduct, we deduce the composition rules for Toeplitz type operators, the extension and restriction operators, and calculate the second term in the asymptotic expansion of the optimal constant in the semi-classical version of Ohsawa-Takegoshi extension theorem.
\end{sciabstract}

\pagestyle{fancy}
\lhead{}
\chead{Complex embeddings and Toeplitz operators}
\rhead{\thepage}
\cfoot{}


\newcommand{\Addresses}{{
  \bigskip
  \footnotesize
  \noindent \textsc{Siarhei Finski, CNRS-CMLS, École Polytechnique F-91128 Palaiseau Cedex, France.}\par\nopagebreak
  \noindent  \textit{E-mail }: \texttt{finski.siarhei@gmail.com}.
}} 

\vspace*{-0.4cm}

\tableofcontents

\section{Introduction}\label{sect_intro}
	One of the main goals of this paper is to prove that for a tower of submanifolds, transitivity property is satisfied for Ohsawa-Takegoshi extension operator, sending a holomorphic section on the submanifold to the holomorphic extension of it on the ambient manifold with the minimal $L^2$-norm, modulo some small error, which can be expressed through Toeplitz type operators.
	\par 
	More precisely, we fix two (not necessarily compact) complex manifolds $X, Y$ of dimensions $n$ and $m$ respectively.
	We fix also a complex embedding $\iota : Y \to X$, a positive line bundle $(L, h^L)$ over $X$ and an arbitrary Hermitian vector bundle $(F, h^F)$ over $X$.
	In particular, we assume that for the curvature $R^L$ of the Chern connection on $(L, h^L)$, the closed real $(1, 1)$-differential form
	\begin{equation}\label{eq_omega}
		\omega := \frac{\imun}{2 \pi} R^L
	\end{equation}
	is positive.
	We denote by $g^{TX}$ the Riemannian metric on $X$ induced by $\omega$ as follows
	\begin{equation}\label{eq_gtx_def}
		g^{TX}(\cdot, \cdot) := \omega(\cdot, J \cdot),
	\end{equation}
	where $J : TX \to TX$ is the complex structure on $X$.
	We denote by $g^{TY}$ the induced metric on $Y$. 	
	\par 
	\textit{We assume throughout the whole article that the triple $(X, Y, g^{TX})$, and the Hermitian vector bundles $(L, h^L)$, $(F, h^F)$ are of bounded geometry in the sense of Definitions \ref{defn_bnd_subm}, \ref{defn_vb_bg}.}
	\par 
	This means that we assume uniform lower bounds $r_X, r_Y > 0$ on the injectivity radii of $X$, $Y$, the existence of the geodesic tubular neighborhood of $Y$ of uniform size $r_{\perp} > 0$ in $X$, and some uniform bounds on related curvatures and the second fundamental form of the embedding.
	\par 
	Now, we fix some positive (with respect to the orientation given by the complex structure) volume forms  $dv_X$, $dv_Y$ on $X$ and $Y$.
	For smooth sections $f, f'$ of $L^p \otimes F$ over $X$, we define the $L^2$-scalar product using the pointwise scalar product $\langle \cdot, \cdot \rangle_h$ induced by $h^L$ and $h^F$ as follows
	\begin{equation}\label{eq_l2_prod}
		\scal{f}{f'}_{L^2(X)} := \int_X \scal{f(x)}{f'(x)}_h dv_X(x).
	\end{equation}
	Similarly, using $dv_Y$, we introduce the $L^2$-scalar product for sections of $\iota^*( L^p \otimes F)$ over $Y$.
	We denote by $L^2(X, L^p \otimes F), L^2(Y,\iota^*( L^p \otimes F))$ the spaces of $L^2$-sections of $L^p \otimes F$ over $X$ and $Y$.
	\par 
	Given a continuous smoothing linear operator $K : L^2(X, L^p \otimes F) \to L^2(X, L^p \otimes F)$, the Schwartz kernel theorem  guarantees the existence of the Schwartz kernel, $K(x_1, x_2) \in (L^p \otimes F)_{x_1} \otimes (L^p \otimes F)_{x_2}^*$; $x_1, x_2 \in X$, evaluated with respect to $dv_X$, i.e. 
	\begin{equation}
		(Ks) (x_1) = \int_X K(x_1, x_2) \cdot s(x_2) dv_X(x_2), \qquad s \in L^2(X, L^p \otimes F).
	\end{equation}
	Similarly, we define the Schwartz kernels $K_1(y, x)$, $K_2(x, y)$, $x \in X$, $y \in Y$, for smoothing operators $K_1 : L^2(X, L^p \otimes F) \to L^2(Y,\iota^*( L^p \otimes F))$,  $K_2 : L^2(Y,\iota^*( L^p \otimes F)) \to L^2(X, L^p \otimes F)$ with respect to the volume forms $dv_X$ and $dv_Y$ respectively.
	\par 
	For a Hermitian vector bundle $(E, h^E)$ over $X$, we denote 
	\begin{equation}
		\ccal^{\infty}_{b}(X, E) 
		:= 
		\Big\{
			f \in \ccal^{\infty}(X, E) : \text{ for any } k \in \nat, \text{ there is } C > 0, \text{ such that } |\nabla^k f| \leq C
		\Big\},
	\end{equation}
	where $\nabla$ is the connection induced by the Chern connection on $E$ and the Levi-Civita connection on $TX$, and $| \cdot |$ is the norm induced by the metrics $g^{TX}$, $h^E$.
	\par 
	Assume that for the Riemannian volume forms $dv_{g^{TX}}$, $dv_{g^{TY}}$ of $(X, g^{TX})$, $(Y, g^{TY})$, we have
	\begin{equation}\label{eq_vol_comp_unif}
		\frac{dv_{g^{TX}}}{dv_X}, \frac{dv_X}{dv_{g^{TX}}} \in \ccal^{\infty}_{b}(X),
		\qquad
		\frac{dv_{g^{TY}}}{dv_Y}, \frac{dv_Y}{dv_{g^{TY}}} \in \ccal^{\infty}_{b}(Y).
	\end{equation}
	\par 
	We denote by $H^0_{(2)}(X, L^p \otimes F)$ and $H^0_{(2)}(Y, \iota^*( L^p \otimes F))$ the vector spaces of holomorphic sections of $L^p \otimes F$ over $X$ and $Y$ respectively with bounded $L^2$-norm.
	In \cite[(4.1)]{FinOTAs}, by relying on the bounded geometry assumption, we proved that the restriction to $Y$ of any $L^2$-holomorphic section, defined on $X$, has finite $L^2$-norm. In other words, the operator 
	\begin{equation}\label{eq_defn_res_map}
		\res_p^{Y|X} : H^{0}_{(2)}(X, L^p \otimes F) \to H^{0}_{(2)}(Y, \iota^* ( L^p \otimes F)), \qquad f \mapsto f|_Y,
	\end{equation}
	is well-defined.
	By extending Ohsawa-Takegoshi theorem, in \cite[Theorem 4.1]{FinOTAs}, cf. \cite{OhsTak1}, \cite{Ohsawa}, \cite[\S 13]{DemBookAnMet}, \cite{DemExtRed}, we established that there is $p_1 \in \nat$, such that (\ref{eq_defn_res_map}) is surjective for any $p \geq p_1$.
	The right inverse of this restriction, defined for $p \geq p_1$ by taking the holomorphic extension with the minimal $L^2$-norm, is called the (\textit{Ohsawa-Takegoshi}) \textit{extension} operator, and it is denoted by
	\begin{equation}\label{eq_ext_op}
		\ext_p^{X|Y} :  H^{0}_{(2)}(Y, \iota^* ( L^p \otimes F)) \to H^{0}_{(2)}(X, L^p \otimes F).
	\end{equation}
	\par
	We identify the normal bundle $N^{X|Y}$ of $Y$ in $X$ as an orthogonal complement of $TY$ in $TX$ (with respect to $g^{TX}$), so that we have the following orthogonal decomposition
	\begin{equation}\label{eq_tx_rest}
		TX|_Y \to TY \oplus N^{X|Y}.
	\end{equation}
	We denote by $g^{N^{X|Y}}$ the metric on $N^{X|Y}$ induced by $g^{TX}$, and let $P_N^{X|Y}$ be the induced projection from $TX|_Y$ to $N^{X|Y}$.
	By an abuse of notation, we denote the induced projection from $(TX|_Y)^*$ to $(N^{X|Y})^*$ by the same symbol.
	\par 
	For $y \in Y$, $Z_N \in N_y^{X|Y}$, let $\real \ni t \mapsto \exp_y^{X}(tZ_N) \in X$ be the geodesic in $X$ in the direction $Z_N$.
	Bounded geometry condition means, in particular, that this map induces a diffeomorphism of $r_{\perp}$-neighborhood of the zero section in $N^{X|Y}$ with a tubular neighborhood $U$ of $Y$ in $X$.
	\par 
	Using this diffeomorphism, we define $\kappa_N^{X|Y} : U \to \real_+$ as the only function verifying
	\begin{equation}\label{eq_kappan}
		dv_X = \kappa_N^{X|Y} \cdot dv_Y \wedge dv_{N^{X|Y}},
	\end{equation}
	where $dv_{N^{X|Y}}$ is the relative Riemannian volume form on $(N^{X|Y}, g^{N^{X|Y}})$.
	We have $\kappa_N^{X|Y}|_Y = 1$ if
	\begin{equation}\label{eq_comp_vol_omeg}
		dv_X = dv_{g^{TX}}, \qquad dv_Y = dv_{g^{TY}}.
	\end{equation}
	\par 
	Let us now fix a tower of submanifolds $Y \xhookrightarrow[]{\iota_1} W \xhookrightarrow[]{\iota_2} X$, $\iota := \iota_2 \circ \iota_1$ of dimensions $m, l$ and $n$ respectively.
	In addition to the volume forms $dv_X$, $dv_Y$, we fix a positive volume form $dv_W$ on $W$, verifying assumptions, similar to (\ref{eq_vol_comp_unif}) with respect to the metric $g^{TW}$ induced by $g^{TX}$.
	We assume, moreover, that the triples $(X, W, g^{TX})$, $(W, Y, g^{TW})$ are of bounded geometry in the sense of Definition \ref{defn_bnd_subm}.
	We denote by $\textbf{r}^X$ (resp. $\textbf{r}^W$) the scalar curvature of $X$ (resp. $W$), and let $\Lambda_{\omega} [ R^F ] \in \enmr{F}$ be the contraction of the curvature of the Chern connection of $(F, h^F)$ with the Kähler form $\omega$.
	We denote by $\Lambda_{\iota_2^* \omega} [ R^F ] \in \enmr{\iota_2^* F}$ the analogous contraction defined on $W$.
	\par 
	\begin{sloppypar}
	We denote by $(N^{X|Y})^{(1, 0)}$, $(N^{X|Y})^{(0, 1)}$ the holomorphic, antiholomorphic components of $N^{X|Y} \otimes \comp$, corresponding to $\imun$ and $- \imun$ eigenspaces of the induced complex structure action.
	\end{sloppypar}
	\begin{thm}\label{thm_as_trans}
		There are $p_1 \in \nat^*$, $C > 0$, such that for any $p \geq p_1$, we have
		\begin{equation}\label{eq_as_trans}
			\Big\| 
				\ext_p^{X|Y}
				-
				\ext_p^{X|W} \circ \ext_p^{W|Y}
			\Big\|
			\leq
			C 
			\cdot
			p^{- \frac{n - m + 1}{2}},
		\end{equation}
		where $\| \cdot \|$ denotes the operator norm.
		Moreover, under the assumptions (\ref{eq_comp_vol_omeg}) and $dv_W = dv_{g^{TW}}$,
		\begin{equation}\label{eq_as_trans_pres}
			\Big\| 
				\ext_p^{X|Y}
				-
				\ext_p^{X|W} \circ \ext_p^{W|Y}
			\Big\|
			\sim
			C_0 
			\cdot
			p^{- \frac{n - m + 3}{2}},
		\end{equation}
		where the constant $C_0 \geq 0$ is defined as follows
		\begin{equation}\label{eq_defect_calc_pres}
			C_0 := \frac{1}{\sqrt{\pi}} \sup_{y \in Y} \Big \| \frac{1}{8 \pi} \partial_{N^{W|Y}} \big( \textbf{r}^X_y - \textbf{r}^W_y \big) \cdot {\rm{Id}}_{F_y} - \frac{1}{2 \pi \imun} \nabla^{1, 0}_{N^{W|Y}} \big( \Lambda_{\omega} [R^F_y] - \Lambda_{\iota_2^* \omega} [ R^F_y ] \big) \Big \|,
		\end{equation}		 
		where the operator $\partial_{N^{W|Y}} : \ccal^{\infty}(Y) \to \ccal^{\infty}(Y, (N^{W|Y})^{(1, 0)*})$ is defined by the composition $P_N^{W|Y} \circ \partial$, the operator $\nabla^{1, 0}_{N^{W|Y}}: \ccal^{\infty}(Y, \enmr{\iota^* F}) \to \ccal^{\infty}(Y, (N^{W|Y})^{(1, 0)*} \otimes \enmr{\iota^* F})$ is similarly defined by the composition $P_N^{W|Y} \circ \nabla^{1, 0}$ for the $(1, 0)$-component $\nabla^{1, 0}$ of the Chern connection on $\enmr{F}$, endowed with the induced Hermitian metric, and the norm is considered as a norm of an element from $(N^{W|Y})^{(1, 0)*} \otimes \enmr{\iota^* F}$ with the induced metric.
	\end{thm}
	\begin{rem}
		a) In \cite[Theorem 1.1]{FinOTAs}, we obtained that, as $p \to \infty$,
		\begin{equation}\label{eq_ext_p_as_no_ass}
			\big\| 
				\ext_p^{X|Y}
			\big\|
			\sim
			\sup_{y \in Y} \kappa_N^{X|Y}(y)^{\frac{1}{2}} \cdot  p^{- \frac{n - m}{2}}.
		\end{equation}
		Hence, (\ref{eq_as_trans}) means that the “defect of transitivity" for the extension operator is of lower order of magnitude than the operator itself.
		We call this property \textit{asymptotic transitivity}.
		\par 
		b) Bounded geometry condition implies that $C_0$ is a finite number.
		\par
		c)
		The estimate (\ref{eq_as_trans}) alone can be obtained directly from \cite[Theorem 1.1]{FinOTAs} and some local calculations, following from Section \ref{sect_tower}.
	\end{rem}
	The main goal of this paper is to give a more precise asymptotic description of the sequence of operators $\ext_p^{X|Y} - \ext_p^{X|W} \circ \ext_p^{W|Y} : H^0_{(2)}(Y, \iota^*( L^p \otimes F)) \to H^0_{(2)}(X, L^p \otimes F)$, $p \geq p_1$.
	Remark that for different $p$, those operators act on different spaces, so the phrase “asymptotic description" itself has to be explained.
	For this, we introduce below \textit{Toeplitz type operators}.
	\par 
	Let $H^{0, Y \perp}_{(2)}(X, L^p \otimes F)$ be the vector space of $L^2$-holomorphic functions which are orthogonal (with respect to the $L^2$-scalar product (\ref{eq_l2_prod})) to $L^2$-holomorphic functions vanishing along $Y$.
	Denote by $B_p^{X|Y \perp}$, $B_p^X$ the orthogonal projections from $L^2(X, L^p \otimes F)$ to $H^{0, Y \perp}_{(2)}(X, L^p \otimes F)$ and $H^{0}_{(2)}(X, L^p \otimes F)$ respectively. 
	The operator $B_p^X$ (resp. $B_p^{X|Y\perp}$) will be called the \textit{Bergman projector} (resp. \textit{orthogonal Bergman projector}).
	We extend $\ext_p^{X|Y}$ to $L^2(Y, \iota^* ( L^p \otimes F))$ as $f \mapsto (\ext_p^{X|Y} \circ B_p^Y) f$.
	\par 
	\begin{sloppypar}
	Now, for a section $f \in \ccal^{\infty}_{b}(X, \enmr{F})$, we associate a sequence of linear operators $T_{f, p}^{X} \in \enmr{L^2(X, L^p \otimes F)}$, $p \in \nat$, called \textit{Berezin-Toeplitz operator}, by
	\begin{equation}
		T_{f, p}^{X}(g) := B_p^X (f \cdot B_p^X g).
	\end{equation}
	\end{sloppypar}
	\par 
	We define the sequences of operators $T_{f, p}^{Y|X} : L^2(X, L^p \otimes F) \to L^2(Y, \iota^* ( L^p \otimes F))$, $T_{f, p}^{X|Y} : L^2(Y, \iota^* ( L^p \otimes F)) \to L^2(X, L^p \otimes F)$, $p \in \nat$, by
	\begin{equation}\label{eq_toepl_fund_defned}
		T_{f, p}^{Y|X} := \res_p^{Y|X} \circ T_{f, p}^X \circ (B_p^X - B_p^{X|Y \perp}),  
		\quad
		T_{f, p}^{X|Y} :=  (B_p^X - B_p^{X|Y \perp}) \circ T_{f, p}^X \circ \ext_p^{X|Y}.
	\end{equation}
	\par 
	As we show in Proposition \ref{prop_toepl_type_poly_suff},  the asymptotic study of operators $T_{f, p}^{Y|X}, T_{f, p}^{X|Y}$, fundamental to this paper, reduces to their study for some functions $f$, polynomial-like in the normal directions to $Y$.
	To describe those functions precisely, we fix a smooth function $\rho : \real_{+} \to [0, 1]$, satisfying
	\begin{equation}\label{defn_rho_fun}
		\rho(x) =
		\begin{cases}
			1, \quad \text{for $x < \frac{1}{4}$},\\
			0, \quad \text{for $x > \frac{1}{2}$}.
		\end{cases}
	\end{equation}
	Let $\pi : N^{X|Y} \to Y$ be the natural projection.
	We fix $g \in \ccal^{\infty}_{b}(Y, {\rm{Sym}}^k (N^{X|Y})^* \otimes \enmr{\iota^* F})$, $k \in \nat$, and construct a section $\{g\} \in \ccal^{\infty}(N^{X|Y}, \pi^* \enmr{\iota^* F})$, polynomial in the vertical directions, as follows $\{ g \} (y, Z_N) := g(y) \cdot Z_N^{\otimes k}$, $y \in Y$, $Z_N \in N_y^{X|Y}$.
	\par
	Recall that we introduced a diffeomorphism of $r_{\perp}$-neighborhood of the zero section in $N^{X|Y}$ with a tubular neighborhood $U$ of $Y$ in $X$ after (\ref{eq_tx_rest}).
	By an abuse of notation, we denote by $\pi : U \to Y$ the projection $(y, Z_N) \mapsto y$ induced by $\pi$ and the above diffeomorphism. 
	Over $U$, we identify $L, F$ to $\pi^* (\iota^* L), \pi^* (\iota^* F)$ by the parallel transport with respect to Chern connections along the geodesic $[0, 1] \ni t \mapsto (y,t  Z_N) \in X$, $|Z_N| < r_{\perp}$.
	From now on, we use those identifications implicitly.
	For fixed $p \in \nat^*$, over $U$, we define the section $\llangle g \rrangle \in \ccal^{\infty}_b(X, \enmr{F})$ as
	\begin{equation}\label{eq_brack_defn}
		\llangle g \rrangle (y, Z_N)
		:=
		p^{\frac{k}{2}}
		\cdot
		\rho \Big(\frac{|Z_N|}{r_{\perp}} \Big) \cdot \{ g \} (y, Z_N),
	\end{equation}
	where the norm $|Z_N|$, is taken with respect to $g^{N^{X|Y}}$.
	Away from $U$, we extend $\llangle g \rrangle$ by zero.
	We extend the operator $\llangle \cdot \rrangle$ linearly to $\oplus_{k = 0}^{\infty} \ccal^{\infty}_{b}(Y, {\rm{Sym}}^k (N^{X|Y})^* \otimes \enmr{\iota^* F})$. 
	\begin{defn}\label{defn_ttype}
		A sequence of linear operators $T_p^Y \in \enmr{L^2(Y, \iota^*( L^p \otimes F))}$, $p \in \nat$, (resp. $T_{p}^{Y|X} : L^2(X, L^p \otimes F) \to L^2(Y, \iota^* ( L^p \otimes F))$, $T_{p}^{X|Y} : L^2(Y, \iota^* ( L^p \otimes F)) \to L^2(X, L^p \otimes F)$) verifying $B_p^Y \circ T_p^Y \circ B_p^Y = T_p^Y$ (resp. $B_p^Y \circ T_p^{Y|X} \circ (B_p^X - B_p^{X|Y \perp}) = T_p^{Y|X}$, $(B_p^X - B_p^{X|Y \perp}) \circ T_p^{X|Y} \circ B_p^Y = T_p^{X|Y}$), is called a \textit{Toeplitz operator with exponential decay} (resp. \textit{of type} $Y|X$, $X|Y$) if there is a sequence $f_i \in \ccal^{\infty}_{b}(Y, \enmr{\iota^* F})$ (resp. $g_i^{h} \in \oplus_{k = 0}^{\infty} \ccal^{\infty}_{b}(Y, {\rm{Sym}}^{2k + j}  (N^{X|Y})^{(1, 0)*} \otimes \enmr{\iota^* F})$, $g_i^{a} \in \oplus_{k = 0}^{\infty} \ccal^{\infty}_{b}(Y, {\rm{Sym}}^{2k + j} (N^{X|Y})^{(0, 1)*} \otimes \enmr{\iota^* F})$, where $j \in \{ 1, 2 \}$ is of the same parity as $i$), and  $c > 0$, $p_1 \in \nat^*$, such that for any $k, l \in \nat$, there is $C > 0$, such that for any $p \geq p_1$, the Schwartz kernels, evaluated with respect to $dv_X$, $dv_Y$, for $y_1, y_2 \in Y$, $x \in X$, satisfy
		\begin{equation}\label{eq_toepl_off_diag}
		\begin{aligned}
			&
			\Big|  
				T_p^Y (y_1, y_2) 
				- 
				\sum_{r = 0}^{k}
				p^{-r}
				T_{f_r, p}^Y(y_1, y_2)  
			\Big|_{\ccal^l} 
			\leq 
			C p^{m - k + \frac{l}{2}} 
			\cdot 
			\exp \big(- c \sqrt{p} \cdot \dist_Y(y_1, y_2) \big),
			\\
			&
			\Big|  
				T_p^{X|Y} (x, y_1) 
				- 
				\sum_{r = 0}^{k}
				p^{-\frac{r}{2}}
				T_{\llangle g_r^h \rrangle, p}^{X|Y}(x, y_1)  
			\Big|_{\ccal^l} 
			\leq 
			C p^{m + \frac{l - k}{2}} 
			\cdot 
			\exp \big(- c \sqrt{p} \cdot \dist_X(x, y_1) \big),
			\\
			&
			\Big|  
				T_p^{Y|X} (y_1, x) 
				- 
				\sum_{r = 0}^{k}
				p^{-\frac{r}{2}}
				T_{\llangle g_r^a \rrangle, p}^{Y|X}(y_1, x)  
			\Big|_{\ccal^l} 
			\leq 
			C p^{n + \frac{l - k}{2}} 
			\cdot 
			\exp \big(- c \sqrt{p} \cdot \dist_X(y_1, x) \big),
		\end{aligned}
		\end{equation}
		where the pointwise $\ccal^{l}$-norm at a point $(y_1, y_2) \in Y \times Y$ is the sum of the norms induced by $h^L, h^F$ and $g^{TX}$, evaluated at $(y_1, y_2)$, of the derivatives up to order $l$ with respect to the connection induced by the Chern connections on $L, F$ and the Levi-Civita connection on $TY$, and similar notations are used for the other two norms at points $(y_1, x) \in Y \times X$ and $(x, y_1) \in X \times Y$. 
		The sections $f_i$ (resp. $g_i^h$, $g_i^a$) will later be denoted by $[T_p^X]_i$ (resp. $[T_p^{X|Y}]_i$, $[T_p^{Y|X}]_i$).
		We alternatively call the above operators \textit{Toeplitz type operators} (\textit{with exponential decay}).
	\end{defn}
	\begin{rem}\label{rem_defn_tpl}
		a)
		In Proposition \ref{prop_norm_bnd_distk_expbnd}, we show that (\ref{eq_toepl_off_diag}) implies that for any $k \in \nat$, there is $C > 0$, such that for any $p \geq p_1$, we have $\|  
				T_p^Y
				- 
				\sum_{r = 0}^{k}
				p^{-r}
				T_{f_r, p}^{Y}
			\|
			\leq 
			C p^{- k}$, 
			$\|  
				T_p^{X|Y}
				- 
				\sum_{r = 0}^{k}
				p^{-\frac{r}{2}}
				T_{\llangle g_r^h \rrangle, p}^{X|Y}
			\|
			\leq 
			C p^{-\frac{n - m + k}{2}}$
			,
			$\|  
				T_p^{Y|X}
				- 
				\sum_{r = 0}^{k}
				p^{- \frac{r}{2}}
				T_{\llangle g_r^a \rrangle, p}^{Y|X}
			\|
			\leq 
			C p^{\frac{n - m - k}{2}}$.
		In particular, the sequence of operators $T_p^Y$, $p \in \nat$, forms a Toeplitz operator in the sense of Ma-Marinescu \cite[\S 7]{MaHol}.
		\par b)
		As we show in Corollary \ref{cor_brack_indep}, our definition ultimately doesn't depend on the choice of $\rho$.
		\par 
		c) 
		In Corollary \ref{cor_brack_well_def1}, we show that the sections $f_i$, $i \in \nat$, (resp. $g_i^h$, $g_i^a$), verifying (\ref{eq_toepl_off_diag}), are uniquely defined. Hence, the notation $[\cdot]_i$, $i \in \nat$, from Definition \ref{defn_ttype} is well-defined. 
	\end{rem}
	\par 
	To state our main result, we place ourselves in the notations and assumptions of Theorem \ref{thm_as_trans}.
	\begin{thm}\label{thm_trans}
		The sequence of operators 
		\begin{equation}
			D_p := \ext_p^{X|Y} - \ext_p^{X|W} \circ \ext_p^{W|Y}, \qquad p \in \nat,
		\end{equation}
		forms a Toeplitz operator with exponential decay of type $X|Y$.
		Moreover, we have $[D_p]_0 = 0$.
		Also, under the assumptions (\ref{eq_comp_vol_omeg}) and $dv_W = dv_{g^{TW}}$, we have $[D_p]_1 = 0$, $[D_p]_2 = 0$, $[D_p]_3 \in \ccal^{\infty}_{b}(Y, (N^{X|Y})^{(1, 0)*} \otimes \enmr{\iota^* F})$, for $n \in (N^{X|W})^{(1, 0)}$, $[D_p]_3 \cdot n = 0$, and for $n \in (N^{W|Y})^{(1, 0)}$:
		\begin{equation}
			[D_p]_3 \cdot n = \frac{1}{8 \pi} \frac{\partial}{\partial n} \cdot \big( \textbf{r}^X - \textbf{r}^W \big) \cdot {\rm{Id}}_{F} - \frac{1}{2 \pi \imun} \nabla^{\enmr{E}}_{n} \big( \Lambda_{\omega} [R^F] - \Lambda_{\iota_2^* \omega} [ R^F ] \big).
		\end{equation}
	\end{thm}
	\par 
	\begin{sloppypar}
	To get a better understanding of Theorem \ref{thm_trans} and its relation with Theorem \ref{thm_as_trans}, let us now state asymptotic formulas for the operators, introduced in Definition \ref{defn_ttype}.
	For $g \in \oplus_{k = 0}^{\infty} \ccal^{\infty}_{b}(Y, {\rm{Sym}}^k (N^{X|Y})^* \otimes \enmr{\iota^* F})$, using the coordinate system as in (\ref{eq_brack_defn}), we define the sequence of operators $M_{g, p}^{X|Y} : L^2(Y, \iota^*(L^p \otimes F)) \to L^2(X, L^p \otimes F)$ by
	\begin{equation}\label{eq_ext0_op}
		(M_{g, p}^{X|Y} f)(y, Z_N) = \llangle g \rrangle (y, Z_N) \cdot \exp \Big(- p \frac{\pi}{2} |Z_N|^2 \Big) \cdot (B_p^Y f)(y),
	\end{equation}
	where $f \in L^2(Y, \iota^*(L^p \otimes F))$ and the norm $|Z_N|$, $Z_N \in N^{X|Y}$, is taken with respect to $g^{N^{X|Y}}$.
	We also define an operator $M_{g, p}^{Y|X, \dagger} : L^2(X, L^p \otimes F) \to L^2(Y, \iota^*(L^p \otimes F))$, $p \in \nat$, as follows
	\begin{equation}\label{eq_ext1_op}
		(M_{g, p}^{Y|X, \dagger}  f)(y)
		=
		p^{n - m} 
		\cdot
		B_p^{Y}
		\pi_*
		\Big(
		\llangle g \rrangle (y, Z_N)
		\cdot
		\exp \Big(- p \frac{\pi}{2} |Z_N|^2 \Big) 
		\cdot
		f(y, Z_N)
		\cdot
		dv_{N^{X|Y}}(Z_N)
		\Big),
	\end{equation}
	where we implicitly identified the restriction of $f \in L^2(X, L^p \otimes F)$ to $U$ with an element from $ L^2(U, \pi^*\iota^*(L^p \otimes F))$, and $\pi_*$ is the integration over the fibers of $N^{X|Y}$.
	Remark that the integration is well-defined because the function $\llangle g \rrangle$ has support in a small tubular neighborhood of $Y$.
	\begin{thm}\label{thm_ttype_as}
		There is $p_1 \in \nat^*$, such that for any $g^h \in \oplus_{k = 1}^{\infty} \ccal^{\infty}_{b}(Y, {\rm{Sym}}^k (N^{X|Y})^{(1, 0)*} \otimes \enmr{\iota^* F})$, $g^a \in  \oplus_{k = 1}^{\infty} \ccal^{\infty}_{b}(Y, {\rm{Sym}}^k (N^{X|Y})^{(0, 1)*} \otimes \enmr{\iota^* F})$, there is $C > 0$, such that for any $p \geq p_1$, the following bounds hold
		\begin{equation}\label{eq_ttype_as}
			\big\|
			T_{\llangle g^h \rrangle, p}^{X|Y}
			-
			M_{g^h, p}^{X|Y}
			\big\|
			\leq
			C p^{- \frac{n - m + 1}{2}},
			\qquad
			\big\|
			T_{\llangle g^a \rrangle, p}^{Y|X} 
			-
			M_{g^a, p}^{Y|X, \dagger}
			\big\|
			\leq
			C p^{\frac{n - m - 1}{2}}.
		\end{equation}
	\end{thm}
	\begin{rem}\label{rem_ttype_as}
		a) In Proposition \ref{prop_c_1c_2_calc}, we show that for nonzero $g \in \oplus_{k = 0}^{\infty} \ccal^{\infty}_{b}(Y, {\rm{Sym}}^k (N^{X|Y})^{*} \otimes \enmr{\iota^* F})$, there are $C_1, C_2 > 0$, which can be written explicitly in terms of $g$, such that, as $p \to \infty$, we have
		\begin{equation}\label{eq_mh_norm_operators}
			\big\|
			M_{g, p}^{X|Y}
			\big\|
			\sim
			C_1 p^{- \frac{n - m}{2}},
			\qquad
			\big\|
			M_{g, p}^{Y|X, \dagger}
			\big\|
			\sim
			C_2 p^{\frac{n - m}{2}}.
		\end{equation}
		Hence, by (\ref{eq_ttype_as}), the operators $M_{g^h, p}^{X|Y}$, $M_{g^a, p}^{Y|X, \dagger}$, are asymptotic to  $T_{\llangle g^h \rrangle, p}^{X|Y}$ and $T_{\llangle g^a \rrangle, p}^{Y|X}$ respectively.
		\par 
		b)
		From Theorem \ref{thm_ttype_as} and Remark \ref{rem_ttype_as}a), we see that Theorem \ref{thm_trans} largely refines Theorem \ref{thm_as_trans}.
	\end{rem}
	\end{sloppypar}
	\par 
	Now, in a slightly different direction, in Theorems \ref{thm_ttype_closure2}, \ref{thm_dual_ttype}, we show that for quasi-isometric embeddings, the set of Toeplitz type operators is closed under taking adjoints, restrictions, extensions and some products.
	This plays a crucial role in our approach to Theorem \ref{thm_trans} and allows us to generalize Theorem \ref{thm_trans} to towers of embeddings of arbitrary length, see Corollary \ref{thm_trans_get} for a precise statement.
	As another direct consequence of our analysis, we obtain the following result.
	\begin{thm}\label{thm_as_ext_res}
		As $p \to \infty$, the following asymptotics holds
		\begin{equation}\label{eq_res_p_as_no_ass}
			\big\| 
				\res_p^{Y|X}
			\big\|
			\sim
			\sup_{y \in Y} \kappa_N^{X|Y}(y)^{-\frac{1}{2}} \cdot  p^{\frac{n - m}{2}}.
		\end{equation}
		Moreover, under assumption (\ref{eq_comp_vol_omeg}), as $p \to \infty$, we even have
		\begin{equation} \label{eq_ep_norm}
			\big\|
			\ext_p^{X|Y}
			\big\|
			-
			\frac{1}{p^{\frac{n - m}{2}}}
			\sim
			\frac{C_3}{p^{\frac{n - m + 2}{2}}},
			\qquad 
			\big\|
			\res_p^{Y|X}
			\big\|
			-
			p^{\frac{n - m}{2}}
			\sim
			C_4 \cdot p^{\frac{n - m - 2}{2}},
		\end{equation}
		where the constants $C_3, C_4$ are defined as follows
		\begin{equation}\label{eq_norms_res_and_ext}
		\begin{aligned}
			&
			C_3 := - \frac{1}{2} \inf_{y \in Y} \bigg( \frac{\textbf{r}^X_y - \textbf{r}^Y_y}{8 \pi} - \lambda_{\max} \Big( \frac{\Lambda_{\omega} [R^F_y] - \Lambda_{\iota^* \omega} [ R^F_y ]}{2 \pi \imun} \Big) \bigg),
			\\
			&
			C_4 := \frac{1}{2} \sup_{y \in Y} \bigg( \frac{\textbf{r}^X_y - \textbf{r}^Y_y}{8 \pi} - \lambda_{\min} \Big( \frac{\Lambda_{\omega} [R^F_y] - \Lambda_{\iota^* \omega} [ R^F_y ]}{2 \pi \imun} \Big) \bigg),
		\end{aligned}
		\end{equation}
		where $\lambda_{\max}$ and $\lambda_{\min}$ are the values of the maximal and minimal eigenvalues.
	\end{thm}
	\begin{rem}
		a) The first asymptotics (\ref{eq_ep_norm}) corresponds to the calculation of the optimal constant in Ohsawa-Takegoshi theorem. 
		A less refined version was proved in \cite[Theorem 1.1]{FinOTAs}, see (\ref{eq_ext_p_as_no_ass}).
		\par 
		b) In particular, from (\ref{eq_ext_p_as_no_ass}) and (\ref{eq_res_p_as_no_ass}), we see that the sequence of operators $p^{\frac{n - m}{2}} \cdot \ext_p^{X|Y}$ is an asymptotic isometry if and only if $\kappa_N^{X|Y}|_Y = 1$.
	\end{rem}
	\par 
	In conclusion, let us give here a brief outline of the proof of Theorem \ref{thm_trans}. 
	It essentially consists of three steps.
	The first step is to prove a characterization of Toeplitz type operators in terms of the asymptotic expansion of their Schwartz kernels.
	We do this in Section \ref{sect_as_crit} by relying on Ma-Marinescu \cite{MaMar08a}.
	The second step consists in showing that the sequence $D_p$, $p \geq p_1$, from Theorem \ref{thm_trans} satisfies the assumptions of this asymptotic characterizations.
	We do this in Theorem \ref{thm_ttype_closure2}, relying on the results from our previous article \cite{FinOTAs} on the asymptotics of the extension operator.
	In the third step, we calculate the first significant term of the asymptotic expansion of $D_p$.
	\par
	For more amenable calculations, in steps 1 and 2, instead of $D_p$, we study the asymptotics of the sequence of operators $\res_p^{W|X} \circ D_p$, which is related to $D_p$ by the basic identity 
	\begin{equation}\label{eq_res_ext_rel_dp_oper}
		\ext_p^{X|W} \circ \res_p^{W|X} \circ D_p = D_p.
	\end{equation}
	In step 3, to study $\res_p^{W|X} \circ D_p$, and out of independent interest, in Section \ref{sect_adj_t_oper}, for $p \geq p_1$, we introduce the sequence of operators, denoted here by $A_p^{X|Y}$, which we call the \textit{multiplicative defect}.
	We show that the sequence of operators $\frac{1}{p^{n - m}} A_p^{X|Y}$, $p \geq p_1$, form a Toeplitz type operator with \underline{weak} exponential decay (see Definition \ref{defn_ttype_weak}, under mild quasi-isometry assumption, this notion coincides with Definition \ref{defn_ttype}, see Proposition \ref{prop_qisom_bgeom}), and the asymptotic expansion of it is related to the asymptotic expansion of the Bergman kernel, studied previously by Tian \cite{TianBerg}, Zelditch \cite{ZeldBerg}, Catlin \cite{Caltin}, Lu \cite{LuBergman}, Wang \cite{WangBergmKern} and Dai-Liu-Ma \cite{DaiLiuMa}.
	By this and the calculation of the asymptotics of Bergman kernel from \cite{TianBerg}, \cite{LuBergman}, \cite{WangBergmKern} and \cite{DaiLiuMa}, we calculate the first two terms of the asymptotic expansion of $\frac{1}{p^{n - m}} A_p^{X|Y}$. 
	Then a formula relating $\res_p^{W|X} \circ D_p$ and $\frac{1}{p^{n - m}} A_p^{X|Y}$, proved in Lemma \ref{lem_def_alt_expr}, allows us to deduce from this calculation the first significant term of the asymptotic expansion of $\res_p^{W|X} \circ D_p$. 
	Finally, the composition rules, established in Theorem \ref{thm_ttype_closure2}.6 and (\ref{eq_res_ext_rel_dp_oper}) allow us to pass from the asymptotic expansion of $\res_p^{W|X} \circ D_p$ to the asymptotic expansion of $D_p$, which finishes the proof of the third step.
	We, finally, mention that the general strategy for dealing with semi-classical limits here is inspired by Bismut \cite{BisDem} and Bismut-Vasserot \cite{BVas}.
	\par 	
	This paper is organized as follows. In Section \ref{sect_bound_geom}, we study the geometry of manifolds of bounded geometry.
	In Section \ref{sect_asymp_toepl_type}, we study the asymptotics of Toeplitz type operators and derive asymptotic criteria for them.
	In Section \ref{sect_alg_prop}, we study the algebraic properties of the set of Toeplitz type operators: we show that it is closed under taking adjoints, restrictions, extensions and some products.
	We study the adjoints of Toeplitz type operators and introduce multiplicative defect.
	Finally, in Section \ref{sect_asymp_trans}, using those preparations, we prove Theorems \ref{thm_as_trans}, \ref{thm_trans}, \ref{thm_as_ext_res} and generalize Theorem \ref{thm_trans} to towers of submanifolds of arbitrary length.
	\par {\bf{Notations.}}
	We use notations $X, Y$ for complex manifolds and $M, H$ for real manifolds.
	The complex (resp. real) dimensions of $X, Y$ (resp. $M, H$) are denoted here by $n, m$.
	An operator $\iota$ always means an embedding $\iota : Y \to X$ (resp. $\iota : H \to M$).
	We denote by $\res_Y$ (resp. $\res_H$) the restriction operator from $X$ to $Y$ (resp. $M$ to $H$). 
	\par 
	For a Riemannian manifold $(M, g^{TM})$, we denote the Levi-Civita connection by $\nabla^{TM}$, by $R^{TM}$ the curvature of it, and by $dv_{g^{TM}}$ the Riemannian volume form.
	For a closed subset $W \subset M$, $r \geq 0$, let $B_{W}^{M}(r)$ be the ball of radius $r$ around $W$. 
	\par 
	For a fixed volume form $dv_M$ on $M$, we denote by $L^2(dv_M, h^E)$ the space of $L^2$-sections of $E$ with respect to $dv_M$ and $h^E$. When $dv_M = dv_{g^{TM}}$, we also use the notation $L^2(g^{TM}, h^E)$. When there is no confusion about the data, we also use the simplified notation $L^2(M, E)$ or $L^2(M)$.
	\par 
	For $n \in \nat^*$, we denote by $dv_{\comp^n}$ the standard volume form on $\comp^n$.
	We view $\comp^m$ (resp. $\real^m$) embedded in $\comp^n$ (resp. $\real^n$) by the first $m$ coordinates.
	For $Z \in \real^k$, we denote by $Z_l$, $l = 1, \ldots, k$, the coordinates of $Z$.
	If $Z \in \real^{2n}$, we denote by $z_i$, $i = 1, \ldots, n$, the induced complex coordinates $z_i = Z_{2i - 1} + \imun Z_{2i}$.
	We frequently use the decomposition $Z = (Z_Y, Z_N)$, where $Z_Y = (Z_1, \ldots, Z_{2m})$ and $Z_N = (Z_{2m + 1}, \ldots, Z_{2n})$. 
	For a fixed frame $(e_1, \ldots, e_{2n})$ in $T_{x}X$, $x \in X$, (resp. $y \in Y$) we implicitly identify $Z$ (resp. $Z_Y$, $Z_N$) to an element in $T_xX$ (resp. $T_yY$, $N_y^{X|Y}$) by
	\begin{equation}\label{eq_Z_ident}
		Z = \sum_{i = 1}^{2n} Z_i e_i, \quad Z_Y = \sum_{i = 1}^{2m} Z_i e_i, \quad Z_N = \sum_{i = 2m + 1}^{2n} Z_i e_i.
	\end{equation}
	If the frame $e_i$ satisfies the condition 		
	\begin{equation}\label{eq_cond_jinv}
		J e_{2i - 1} = e_{2i},
	\end{equation}
	we denote $\frac{\partial}{\partial z_i} := \frac{1}{2} (e_{2i-1} - \imun e_{2i})$, $\frac{\partial}{\partial \overline{z}_i} := \frac{1}{2} (e_{2i-1} + \imun e_{2i})$, and identify $z, \overline{z}$ to vectors in $T_xX \otimes_{\real} \comp$ as follows
	\begin{equation}\label{eq_z_ovz_id}
		z = \sum_{i = 1}^{n} z_i \cdot \frac{\partial}{\partial z_i}, \qquad
		\qquad
		\overline{z} = \sum_{i = 1}^{n} \overline{z}_i \cdot \frac{\partial}{\partial \overline{z}_i}.
	\end{equation}
	Clearly, in this identification, $Z = z + \overline{z}$.
	We define $z_Y, \overline{z}_Y \in T_yY \otimes_{\real} \comp$, $z_N, \overline{z}_N \in N_y^{X|Y} \otimes_{\real} \comp$ in a similar way. 
	We sometimes further decompose $Z_N = (Z_{N^{W|Y}}, Z_{N^{X|W}})$ for $Z_{N^{W|Y}} \in \real^{2(l - m)}$, $Z_{N^{X|W}}  \in \real^{2(n - l)}$, $m \leq l \leq n$, and use the analogous identifications.
	Sometimes, we use the notation $Z_W := (Z_Y, Z_{N^{W|Y}})$.
	\par 
	For $\alpha = (\alpha_1, \ldots, \alpha_k) \in \nat^k$, $B = (B_1, \ldots, B_k) \in \comp^k$, we write by 
	\begin{equation}
		|\alpha| = \sum_{i = 1}^{k} \alpha_i, \quad \alpha! = \prod_{i = 1}^{k} (\alpha_i)!, \quad B^{\alpha} = \prod_{i = 1}^{k} B_i^{\alpha_i}.
	\end{equation}
	\par {\bf{Acknowledgments.}} Author would like to warmly thank Jean-Pierre Demailly, the numerous enlightening discussions with whom inspired this article.
	He also thanks Paul Gauduchon for his interest and related discussions.
	This work was initiated in Institut Fourier, where author was supported by the ERC grant ALKAGE number 670846 managed by Jean-Pierre Demailly, and finished in CMLS, where author benefited from the support of CNRS and École Polytechnique.
	Author thanks all the institutions for wonderful working conditions.
	
\section{Bounded geometry and local trivializations}\label{sect_bound_geom}
	The main goal of this section is to study the geometry of manifolds of bounded geometry.
	More precisely, in Section \ref{sect_bnd_geom_cf}, we recall the definitions manifolds (resp. pairs of manifolds, vector bundles) of bounded geometry and the quasi-isometry assumption.
	In Section \ref{sect_exp_int_conv}, we study the convergence of exponential integrals on manifolds of bounded geometry.
	In Section \ref{sect_coord_syst}, we recall some results comparing geodesic and Fermi coordinates and related trivializations of vector bundles.
	Finally, in Section \ref{sect_tower}, we extend those results to towers of submanifolds.
	
\subsection{Introduction to bounded geometry condition}\label{sect_bnd_geom_cf}
	In this section, we recall the definitions of manifolds (resp. pairs of manifolds, vector bundles) of bounded geometry. 
	We also recall the basic facts about the second fundamental form and the quasi-isometry assumption.
	For more detailed overview, refer to \cite{EichBoundG}, \cite{SchBound}, \cite{GrosSchnBound}, cf. \cite{FinOTAs}.
	\begin{defn}\label{defn_bnd_g_man}
		We say that a Riemannian manifold $(M, g^{TM})$ is of bounded geometry if the following two conditions are satisfied.
		\\ \hspace*{0.3cm} \textit{(i)} The injectivity radius of $(M, g^{TM})$ is bounded below by a positive constant $r_M$. 
		\\ \hspace*{0.3cm} \textit{(ii)} For the Riemann curvature tensor $R^{TM}$ of $M$, we have $R^{TM} \in \ccal^{\infty}_b(M, \Lambda^2 T^*M \otimes \enmr{TM})$.
	\end{defn}
	\begin{rem}\label{rem_hopf_rin}
		By Hopf-Rinow theorem, the condition \textit{(i)} implies that $(M, g^{TM})$ is complete.	
	\end{rem}
	\par 
	Now, let $(H, g^{TH})$ be an embedded submanifold of $(M, g^{TM})$, $g^{TH} := g^{TM}|_H$.
	We identify the normal bundle $N^{M|H}$ of $H$ in $M$ to an orthogonal complement of $TH$ in $TM$ as in (\ref{eq_tx_rest}).
	We denote by $g^{N^{M|H}}$ the metric on $N^{M|H}$ induced by $g^{TM}$. 
	We denote by $P_N^{M|H} : TM|_H \to N^{M|H}$,  $P_H^{M|H} : TM|_H \to TH$, the projections induced by (\ref{eq_tx_rest}).
	Clearly, $\nabla^{N^{M|H}} := P_N^{M|H} \nabla^{TM}|_H$ defines a connection on $N^{M|H}$.
	We define the \textit{second fundamental form} $A^{M|H} \in \ccal^{\infty}(H, T^*H \otimes \enmr{TM|_H})$ by
	\begin{equation}\label{eq_sec_fund_f}
		A^{M|H} := \nabla^{TM}|_{H} - \nabla^{TH} \oplus \nabla^{N^{M|H}}.
	\end{equation}
	Recall that the \textit{mean curvature} $\nu^{M|H} \in \ccal^{\infty}(H, N^{M|H})$ of $\iota$ is defined as follows
	\begin{equation}\label{eq_mn_curv_d}
		\nu^{M|H} := \frac{1}{m} \sum_{i=1}^{m} A^{M|H}(e_i)e_i,
	\end{equation}
	where the sum runs over an orthonormal basis of $(TH, g^{TH})$.
	The following statement about the second fundamental form will be used throughout the whole article.
	\begin{prop}\label{prop_prop_sfndform}
		The second fundamental form satisfies the following properties.
		\begin{enumerate}
			\item It takes values in skew-symmetric endomorphisms of $TM|_H$, interchanging $TH$ and $N^{M|H}$.
			\item For any $U, V \in TH$, we have $A^{M|H}(U) V = A^{M|H}(V) U$.
		\end{enumerate}
		Assume, moreover, that $(M, g^{TM})$ is Kähler. Then the following holds.
		\begin{enumerate}[resume]
			\item $A^{M|H}$ commutes with the action of the complex structure.
			\item For any $U \in TH$, $V \in TM$, $U = u + \overline{u}$, $V = v + \overline{v}$, $u, v \in T^{1, 0}M$, we have
			\begin{equation}
			\begin{aligned}
				&
				A^{M|H}(U)v = A^{M|H}(\overline{u}) v, 
				&&
				A^{M|H}(U)\overline{v} = A^{M|H}(u) \overline{v}, 
				&&&
				\text{if } V \in N^{M|H},
				\\
				&
				A^{M|H}(U)v = A^{M|H}(u) v, 
				&&
				A^{M|H}(U)\overline{v} = A^{M|H}(\overline{u}) \overline{v},
				&&&
				\text{if } V \in TH.
			\end{aligned}
			\end{equation}
			\item We have $\nu^{M|H} = 0$.
		\end{enumerate}
	\end{prop}
	\begin{proof}
		The proof is straightforward, so we only highlight the main ideas.
		The point 1 is a consequence of the fact that the Levi-Civita connection preserves the Riemannian metric and the well-known fact that $P_H^{M|H} \nabla^{TM} P_H^{M|H} = \nabla^{TH}$.
		The point 2 is a consequence of the fact that the Levi-Civita connection has no torsion.
		The point 3 follows from the fact that for Kähler manifolds, the complex structure is parallel with respect to the Levi-Civita connection, cf. \cite[Theorem 1.2.8]{MaHol}.
		The point 4 is an easy consequence of points 1, 2, 3.
		Let us now establish point 5. By point 4,
		\begin{equation}
			A^{M|H}(J U)J U = J A^{M|H}(J U)U = J A^{M|H}(U)J U = - A^{M|H}(U) U.
		\end{equation}
		The point 5 now follows from this observation by choosing an orthonormal frame $(e_1, \ldots, e_n)$ in $(TM, g^{TM})$, such that for $i = 1, \ldots, \frac{n}{2}$, the condition (\ref{eq_cond_jinv}) is satisfied.
	\end{proof}
	\begin{defn}\label{defn_bnd_subm}
		We say that the triple $(M, H, g^{TM})$ is of bounded geometry if the following conditions are fulfilled.
		\\ \hspace*{0.3cm} \textit{(i)} The manifold $(M, g^{TM})$ is of bounded geometry.
		\\ \hspace*{0.3cm} \textit{(ii)} The injectivity radius of $(H, g^{TH})$ is bounded below by a positive constant $r_H$. 
		\\ \hspace*{0.3cm} \textit{(iii)} There is a collar around $H$ (a tubular neighborhood of fixed radius), i.e. there is $r_{\perp} > 0$ such that for any $x, y \in H$, the normal geodesic balls $B^{\perp}_{r_{\perp}}(x), B^{\perp}_{r_{\perp}}(y)$, obtained by the application of the exponential mapping to vectors, orthogonal to $H$, of norm bounded by $r_{\perp}$, are disjoint.
		\\ \hspace*{0.3cm} \textit{(iv)} 
		The second fundamental form, $A^{M|H}$, satisfies $A^{M|H} \in \ccal^{\infty}_b(H, T^*M|_H \otimes \enmr{TM|_H})$.
	\end{defn}	
	We will now introduce a coordinate system in $M$ near a fixed point in $H$, which is particularly well-adapted to the study of triples of bounded geometry.
	We fix a point $y_0 \in H$ and an orthonormal frame $(e_1, \ldots, e_m)$ (resp. $(e_{m+1}, \ldots, e_n)$) in $(T_{y_0}H, g_{y_0}^{TH})$ (resp. in $(N_{y_0}^{M|H}, g_{y_0}^{N^{M|H}})$). 
	For $Z = (Z_H, Z_N)$, $Z_H \in \real^{m}$, $Z_N \in \real^{n - m}$, $Z_H = (Z_1, \ldots, Z_m)$, $Z_N = (Z_{m + 1}, \ldots, Z_{n})$, $|Z_H| \leq r_H$, $|Z_N| \leq r_{\perp}$, we define a coordinate system $\psi_{y_0}^{M|H} : B_0^{\real^{m}}(r_H) \times B_0^{\real^{n - m}}(r_{\perp}) \to M$ by 
	\begin{equation}\label{eq_defn_fermi}
		\psi_{y_0}^{M|H}(Z_H, Z_N) := \exp_{\exp_{y_0}^{H}(Z_H)}^{M}(Z_N(Z_H)),
	\end{equation}
	where $Z_N(Z_H)$ is the parallel transport of $Z_N \in N_{y_0}^{M|H}$ along $\exp_{y_0}^{H}(t Z_H)$, $t = [0, 1]$, with respect to the connection $\nabla^{N^{M|H}}$ on $N^{M|H}$.
	The coordinates $\psi_{y_0}^{M|H}$ are called the \textit{Fermi coordinates} at $y_0$.
	Their importance comes from the following proposition.
	\begin{prop}[{\cite[Lemma 3.9]{SchBound}, \cite[Theorem 4.9]{GrosSchnBound}}]\label{prop_bndg_tripl}
		For any triple $(M, H, g^{TM})$ of bounded geometry, for any $r_0 > 0$, there is $D_k > 0$, such that for any $y_0 \in H$, $l = 0, \ldots, k$, we have
		\begin{equation}\label{eq_bndtr_metr_tens}
			\| g_{ij} \|_{\ccal^{l}(B_0^{\real^n}(r_0))} \leq D_k, \quad \| g^{ij} \|_{\ccal^{l}(B_0^{\real^n}(r_0))} \leq D_k.
		\end{equation}
		where $g_{ij}$, $i, j = 1, \ldots, n$, are the coefficients of the metric tensor $\psi_{y_0}^* g^{TM}$, and $g^{ij}$ are the coefficients of the inverse matrix.
	\end{prop}
	\begin{rem}\label{rem_bndg_tripl}
		a) In particular, for a triple of bounded geometry $(M, H, g^{TM})$, the Riemannian manifold $(H, g^{TH})$ has bounded geometry.
		\par 
		b) Clearly, this result along with the assumption (\ref{eq_vol_comp_unif}) imply that in the notations of (\ref{eq_kappan}), we have $\kappa_N^{X|Y} \in \ccal^{\infty}_{b}(U)$.
	\end{rem}
	\par Now, recall that an embedding $\iota : H \to M$ is called \textit{quasi-isometry} if there are $A, B > 0$ such that for any $y_1, y_2 \in H$, we have
	\begin{equation}\label{eq_quasi_isom}
		\dist_H(y_1, y_2)  \leq A \dist_M(\iota(y_1), \iota(y_2)) + B. 
	\end{equation}
	In what follows, for brevity, we omit $\iota$ from the distance function.
	As an application of Proposition \ref{prop_bndg_tripl}, we deduce the following result.
	\begin{prop}\label{prop_qisom_bgeom}
		Assume that a triple $(M, H, g^{TM})$ is of bounded geometry and the embedding $\iota$ is quasi-isometry. 
		Then one can choose $B = 0$ in (\ref{eq_quasi_isom}).
	\end{prop}
	\begin{proof}
		The proof consists in considering three different cases. First, we assume that the points $y_1, y_2 \in H$ satisfy $\dist_H(y_1, y_2) > 2 B$. Then, by (\ref{eq_quasi_isom}), we have $\frac{B}{A} \leq \dist_M(y_1, y_2)$. Hence, we have
		$
			A \dist_M(y_1, y_2) + B \leq 2 A \dist_M(y_1, y_2)
		$.
		\par 
		Now, let $r_0$ be as in Proposition \ref{prop_bndg_tripl}.
		We denote $r_1 := \min \{ r_0, r_H, r_{\perp} \}$, where $r_0$, $r_H$, $r_{\perp}$ are as in Proposition \ref{prop_bndg_tripl} and Definitions \ref{defn_bnd_g_man}, \ref{defn_bnd_subm}.
		Assume that the points $y_1, y_2 \in H$ satisfy $\frac{r_1}{2} \leq \dist_H(y_1, y_2) \leq \frac{B}{A}$.
		The geodesic between $y_1$ and $y_2$ in $M$ has to either pass inside the $\frac{r_1}{2}$-tubular neighborhood of $H$ or outside.
		In the latter case, we clearly have 
		$
			\dist_M(y_1, y_2) \geq r_1
		$.
		\par 
		To treat the former case, we denote by $\Gamma_y$, $y \in H$, the set $\{ \exp_y^M(Z_N) : Z_N \in N_y^{M|H}, |Z_N| \leq \frac{r_1}{2} \}$.
		Bounded geometry condition means, in particular, that those sets do not intersect each other for different $y$.
		From Proposition \ref{prop_bndg_tripl}, we actually see that there is a constant $c > 0$, such that for any $y'_1, y'_2$, verifying $\dist_H(y'_1, y'_2) = \frac{r_1}{2}$, we have $\dist(\Gamma_{y'_1}, \Gamma_{y'_2}) \geq c$.
		Then, since we always stay inside the $\frac{r_1}{2}$-tubular neighborhood of $H$, and $\frac{r_1}{2} \leq \dist_H(y_1, y_2)$, we deduce that
		$
			\dist_M(y_1, y_2) \geq c
		$.
		Overall, we deduce that we can find a constant $C > 0$, such that whenever $y_1, y_2 \in H$ satisfy $\frac{r_1}{2} \leq \dist_H(y_1, y_2) \leq \frac{B}{A}$, we have
		$
			\dist_M(y_1, y_2) \geq C \dist_H(y_1, y_2)
		$.
		\par 
		Now, it is only left to treat the last case, where the points $y_1, y_2 \in H$ satisfy $\dist_H(y_1, y_2) \leq \frac{r_1}{2}$.
		But in this case, directly from Proposition \ref{prop_bndg_tripl}, we can find a constant $C > 0$, such that
		$
			\dist_M(y_1, y_2) \geq C \dist_H(y_1, y_2)
		$.
		Hence, in all cases, $B$ can be taken to be equal to $0$.
	\end{proof}
	Finally, we recall the last definition related to bounded geometry.
	\begin{defn}\label{defn_vb_bg}
		Let $(E, \nabla^E, h^{E})$ be a Hermitian vector bundle with a fixed Hermitian connection over a manifold $(M, g^{TM})$ of bounded geometry.
		We say that $(E, \nabla^E, h^{E})$ is of bounded geometry if $R^E \in \ccal^{\infty}_b(M, \Lambda^2 T^*M \otimes \enmr{E})$.
		\par 
		If $(E, h^{E})$ is a Hermitian vector bundle over a \textit{complex manifold}, we say that it is of bounded geometry if $(E, \nabla^E, h^{E})$ is of bounded geometry for the Chern connection $\nabla^E$ on $(E, h^{E})$.
	\end{defn}
	
\subsection{Convergence of exponential integrals for triples of bounded geometry}\label{sect_exp_int_conv}
The main goal of this section is to study convergence of exponential integrals for triples of bounded geometry.
	More precisely, fix a triple $(M, H, g^{TM})$ of bounded geometry.
	We conserve the notations from Definition \ref{defn_bnd_subm}.
	\begin{prop}[{\cite[Corollary 3.3]{FinOTAs}}]\label{prop_exp_bound_int}
		There are $c, C' > 0$, which depend only on $n$, $m$, $r_M$, $r_N$, $r_{\perp}$ and $\sup$-norm on $R^{TM}$, $R^{TH}$, $A^{M|H}$, such that for any $y_0 \in H$, $l > c$, the following bound holds
		\begin{equation}
			\int_{H} \exp \big(-l \dist_M(y_0, y) \big) dv_{g^{TH}}(y) < \frac{C'}{l^{m}}.
		\end{equation}
	\end{prop}
	Let $(E, h^E)$ be a Hermitian vector bundle over $M$ and $D$ is an operator acting on $L^2(H, \iota^* E)$.
	Assume that there are $c > 0$, $l$ as in Proposition \ref{prop_exp_bound_int} and $C > 0$, such that for any $y_1, y_2 \in H$, the Schwartz kernel of $D$, evaluated with respect to $dv_{g^{TH}}$, satisfies the bound
	\begin{equation}\label{eq_d_bnd}
		\Big| D(y_1, y_2) \big|
		\leq 
		C
		l^m
		\exp \big(-l \dist_M(y_1, y_2) \big).
	\end{equation}
	\begin{cor}\label{cor_norm_bnd_oper}
		For $C' > 0$ as in Proposition \ref{prop_exp_bound_int}, we have $\| D \| \leq C C'$.
	\end{cor}
	\begin{proof}
		From Proposition \ref{prop_exp_bound_int} and (\ref{eq_d_bnd}), we deduce that there is $C' > 0$, as in Proposition \ref{prop_exp_bound_int}, such that for any $y_0 \in H$, we have	
		\begin{equation}\label{eq_bound_int_berg}
			\int_{H} \big| D(y_0, y)  \big|  dv_H(y) \leq C C', 
			\qquad 
			\int_{H} \big| D(y, y_0)  \big|  dv_H(y) \leq C C'.
		\end{equation}
		We conclude directly from (\ref{eq_bound_int_berg}) and Young's inequality for integral operators, cf. \cite[Theorem 0.3.1]{SoggBook} applied for $p, q = 2$, $r = 1$ in the notations of \cite{SoggBook}.
	\end{proof}
	Now, let $c > 0$ be as in Proposition \ref{prop_exp_bound_int}.
	Let $D_i$, $i = 1, \ldots, r$, be operators acting on $L^2(H, \iota^* E)$, such that for some $l \geq 2 c$, $C > 0$, the bound (\ref{eq_d_bnd}) holds for $D := D_i$.
	\begin{cor}[{ \cite[Lemma 3.1]{FinOTAs} }]\label{cor_comp_exp_bound}
		The Schwartz kernel $D_{r + 1}(y_1, y_2)$ of the operator $D_{r + 1} := D_1 \circ D_2 \circ \cdots \circ D_r$, evaluated with respect to $dv_{g^{TH}}$, is well-defined and for any $y_1, y_2 \in H$ and it satisfies the bound (\ref{eq_d_bnd}) for $D := D_{r+1}$, $l := \frac{l}{2}$, $C := (C')^r C^r$ for $C'$ as in Proposition \ref{prop_exp_bound_int}.
	\end{cor}
	
	Now, in addition to the triple $(M, H, g^{TM})$ of bounded geometry, we consider a Riemannian manifold $(K, g^{TK})$ with an embedding $\iota_1 : M \to K$, such that $\iota_1^* g^{TK} = g^{TM}$.
	We assume, moreover, that the triple $(K, M, g^{TK})$ is of bounded geometry.
	Let $(E, h^E)$ be a Hermitian vector bundle over $M$ and $D: L^2(H, \iota^* E) \to L^2(M, E)$ be a fixed linear operator.
	Assume that there is $c > 0$ as in Proposition \ref{prop_exp_bound_int} and $C > 0$, such that for some $l \geq c$ and any $y \in H$, $x \in M$, the Schwartz kernel of $D$, evaluated with respect to $dv_{g^{TH}}$, satisfies the bound
	\begin{equation}\label{eq_d_bnd2}
		\big| D(x, y) \big|
		\leq 
		C
		l^m
		\exp \big(-l \dist_K(x, y) \big).
	\end{equation}
	\begin{prop}\label{prop_norm_bnd_distk_expbnd}
		There is $C' > 0$, which depends on the same data as constants from Proposition \ref{prop_exp_bound_int} and the analogous data on $(K, M, g^{TK})$, such that 
		\begin{equation}
			\| D \| \leq \frac{C' C}{l^{\frac{n - m}{2}}}.
		\end{equation}
	\end{prop}
	\begin{rem}
		Clearly, Corollary \ref{cor_norm_bnd_oper} is a special case of Proposition \ref{prop_norm_bnd_distk_expbnd} for $M := H$.
	\end{rem}
	\begin{proof}
		First of all, let us establish this result for $K := M$.
		We consider the exponential map $\exp_N^{M|H} : N^{M|H} \to M$.
		As $M$ is complete, see Remark \ref{rem_hopf_rin}, the map $\exp_N^{M|H}$ is surjective.
		We consider a subset $V$ of $N^{M|H}$, consisting of points $u \in N^{M|H}$, such that $\dist( \exp_N^{M|H}(u), Y) = |u|$, where $|u|$ is the norm of $u$ with respect to the induced metric on $N^{M|H}$.
		\par 
		We let $\delta = \min \{ \inf_{x, T} {\rm{sec}}(x, T), -1 \}$, where ${\rm{sec}}(x, T)$ is the sectional curvature of $(M, g^{TM})$, evaluated at $x \in M$ for the two-dimensional subspace $T \in T_x M$ of the tangent bundle, and the infimum is taken over all possible choices of $x$ and $T$.
		Bounded geometry condition implies that $\delta$ is a finite constant.
		For $r > 0$, we denote $s_{\delta}(r) := \frac{1}{|\delta|^{\frac{1}{2}}} \sinh (|\delta|^{\frac{1}{2}} r)$ and $c_{\delta}(r) := s'_{\delta}(r)$.
		Then from the result of Heintze-Karcher \cite[Corollary 3.3.1]{HeintzKarch}, we obtain that for $u \in V$, $|u| = t$, $t > 0$,
		\begin{equation}
			\big| \det (d \exp_N^{M|H})_u \big| 
			\leq
			\Big(
			\frac{s_{\delta}(t)}{t} 
			\Big)^{n - m - 1}
			\cdot
			\Big(
				c_{\delta}(t) - \scal{\nu^{M|H}}{u} \cdot s_{\delta}(t)
			\Big)^n,
		\end{equation}
		where $\nu^{M|H}$ was defined in (\ref{eq_mn_curv_d}).
		In particular, due to bounded geometry assumption, we obtain that there are $c, C' > 0$, as in Proposition \ref{prop_exp_bound_int}, such that
		\begin{equation}
			\big| \det (d \exp_N^{M|H})_u \big| 
			\leq
			C' \exp(c t).
		\end{equation}
		\par 
		Using this fact and Remark \ref{rem_bndg_tripl}b), we deduce that there is $C > 0$ such that for any $f \in L^2(H, \iota^* E)$, we have
		\begin{multline}\label{eq_df_norm_bounds}
			\| D f \|_{L^2}^2
			\leq
			C
			\int_V
			\exp(c |u|)
			\cdot
			\\
			\cdot
			 \Big(
			 	l^m
				\int_H
				\exp \big( - l \dist_M(y_1, \exp^M_{y_2}(u) ) \big) f(y) 
				dv_{g^{TH}}(y_1)  
			\Big)^2  dv_{N^{M|H}_{y_2}}(u) dv_{g^{TH}}(y_2).
		\end{multline}
		However, clearly, we have $\dist_M(y_1, \exp^M_{y_2}(u) ) \geq \dist_M(Y, \exp^M_{y_2}(u) )$. 
		But by the definition of the subset $V$, we have $\dist_M(Y, \exp^M_{y_2}(u) ) = \dist_M(y_2, \exp^M_{y_2}(u) ) = |u|$.
		From this and (\ref{eq_df_norm_bounds}), for $l \geq 8c$, we obtain that
		\begin{multline}\label{eq_df_norm_bounds2}
			\| D f \|_{L^2}^2
			\leq
			C
			\int_V
			\exp \big( - \frac{l}{2} |u| \big)
			\cdot
			\\
			\cdot
			 \Big(
			 	l^m
				\int_H
				\exp \big( - \frac{l}{8} \dist_M(y_1, y_2 ) \big) f(y) 
				dv_{g^{TH}}(y_1)  
			\Big)^2  dv_{N_{y_2}^{M|H}}(u) dv_{g^{TH}}(y_2) .
		\end{multline}
		From the boundness of the exponential integral, for any $y \in Y$, we obtain that
		\begin{equation}
			\int_{V \cap N_y^{M|H}}
			\exp( - \frac{l}{2} |u|) dv_{N^{M|H}_y}(u)
			\leq
			\frac{C}{l^{n - m}}.
		\end{equation}
		By combining with (\ref{eq_df_norm_bounds2}), it gives us
		\begin{equation}\label{eq_df_norm_bounds3}
			\| D f \|_{L^2}^2
			\leq
			\frac{C}{l^{n - m}}
			\int_H
			 \Big(
			 	l^m
				\int_H
				\exp \big( - \frac{l}{4} \dist_M(y_1, y_2 ) \big) f(y) 
				dv_{g^{TH}}(y_1)  
			\Big)^2  dv_{g^{TH}}(y_2).
		\end{equation}
		Now, from Corollary \ref{cor_norm_bnd_oper}, we obtain that there is $C > 0$, verifying
		\begin{equation}\label{eq_df_norm_bounds4}
			\int_H
			 \Big(
			 	l^m
				\int_H
				\exp \big( - \frac{l}{4} \dist_M(y_1, y_2 ) \big) f(y) 
				dv_{g^{TH}}(y_1)  
			\Big)^2  dv_{g^{TH}}(y_2)
			\leq
			C \| f \|_{L^2}^{2}.
		\end{equation}
		A combination of (\ref{eq_df_norm_bounds3}) and (\ref{eq_df_norm_bounds4}) finishes the proof for the case $K := M$.
		An easy verification shows that all the constants can be chosen as described in the statement of the proposition we are proving.
		The proof of the general case reduces to the case considered above through the use of the tubular neighborhood of $M$ in $K$ in the same way as in the proof of \cite[Corollary 3.3]{FinOTAs}.
	\end{proof}
	
\subsection{Fermi and geodesic coordinates; related trivializations of vector bundles}\label{sect_coord_syst}
	In this section, we recall some results comparing geodesic and Fermi coordinates and trivializations of vector bundles adapted to those coordinate systems.
	We place ourselves in the setting of a triple $(M, H, g^{TM})$ of bounded geometry.
	\par 
	Let us fix $x_0 \in M$ and an orthonormal frame $(e_1, \ldots, e_n)$ of $(T_{x_0}M, g^{TM}_{x_0})$.
	We define the map $\phi_{x_0}^M : \real^n \to M$, $x_0 \in M$, as follows
	\begin{equation}\label{eq_phi_defn}
		\phi_{x_0}^{M}(Z) := \exp^{M}_{x_0}(Z).
	\end{equation}
	\par 
	Define the constant $R > 0$ as follows
	\begin{equation}\label{eq_r_defn_const}
		R := \min \Big\{ 
			\frac{r_M}{2}, \frac{r_H}{4}, \frac{r_{\perp}}{4}		
		\Big\}.
	\end{equation}
	Assume now $x_0 = y_0$, where $y_0 \in H$ and let $(e_1, \ldots, e_n)$ be as in (\ref{eq_defn_fermi}).
	Recall that Fermi coordinates $\psi_{y_0}^{M|H}$ were defined in (\ref{eq_defn_fermi}).
	Clearly, there is a (unique) diffeomorphism $h_{y_0}^{M|H} : B_0^{\real^n}(R) \to \real^n$, $h_{y_0}^{M|H}(0) = 0$, such that the following identity holds
	\begin{equation}\label{eq_h_defn_tr_m}
		\psi_{y_0}^{M|H} = \phi_{y_0}^{M} \circ h_{y_0}^{M|H}.
	\end{equation}
	We recall that $A^{M|H} \in \ccal^{\infty}(H, T^*H \otimes \enmr{TM|_H})$ was defined in (\ref{eq_sec_fund_f}). 
	Let an auxiliary section $B^{M|H} \in \ccal^{\infty}(H, {\rm{Sym}}^2(T^*M|_H) \otimes TM|_H)$, for $Z \in TM|_H$, be defined as
	\begin{equation}\label{eq_b_defn}
		B^{M|H}(Z) := B^{M|H}(Z, Z) := \frac{1}{2} A^{M|H}(Z_H) Z_H + A^{M|H}(Z_H) Z_N.
	\end{equation}
	\begin{prop}[ {\cite[Proposition 2.18]{FinOTAs} }]\label{prop_diff_exp}
		The diffeomorphism $h^{M|H}_{y_0}$ admits the Taylor expansion
		\begin{equation}\label{eq_prop_diff_exp}
			h^{M|H}_{y_0}(Z) = Z + B^{M|H}(Z) + O(|Z|^3).
		\end{equation}
		Moreover, the derivatives of $h^{M|H}_{y_0}$ are bounded uniformly on $y_0 \in Y$, $|Z| \leq R$.
	\end{prop}
	\par 
	In the second part of this section, we recall the comparison between two trivializations of vector bundles, done using parallel transport adapted to the above coordinate systems.
	\par 
	We fix an orthonormal frame $f_1, \ldots, f_r \in (E_{x_0}, h^E_{x_0})$.
	Let $\tilde{f}'_1{}^{M}, \ldots, \tilde{f}'_r{}^{M}$ be a frame of $E$ over $B_{x_0}^{M}(r_M)$, obtained by the parallel transport of $f_1, \ldots, f_r$ along the curve $\phi_{x_0}^M(tZ)$, $t \in [0, 1]$, $Z \in T_{x_0}M$, $|Z| < r_M$.
	Assume now $x_0 = y_0$, where $y_0 \in H$ and let $(e_1, \ldots, e_n)$ be as in (\ref{eq_defn_fermi}).
	We define $\tilde{f}_1^{M|H}, \ldots, \tilde{f}_r^{M|H}$ by the parallel transport of $f_1, \ldots, f_r$ with respect to the connection $\nabla^E$, first along the path $\psi_{y_0}^{M|H}(t Z_Y, 0)$, $t \in [0, 1]$, and then along the path $\psi_{y_0}^{M|H}(Z_Y, tZ_N)$, $t \in [0, 1]$, $Z_Y \in \real^{2m}$, $Z_N \in \real^{2(n-m)}$, $|Z_Y| < r_Y$, $|Z_N| < r_{\perp}$.
	\par 
	Let $\xi_E^{M|H}$ be the unique smooth function over $B_{y_0}^{M}(R)$, with values in $\enmr{\comp^r}$, such that $\xi_E^{M|H}(0) = 0$, and the following identity holds
	\begin{equation}\label{eq_frame_tilde}
		(\tilde{f}_1^{M|H}, \ldots, \tilde{f}_r^{M|H}) = \exp(\xi_E^{M|H}) \cdot (\tilde{f}'_1{}^{M}, \ldots, \tilde{f}'_r{}^{M}),
	\end{equation}
	where we view $(\tilde{f}_1^{M|H}, \ldots, \tilde{f}_r^{M|H})$ and $(\tilde{f}'_1{}^{M}, \ldots, \tilde{f}'_r{}^{M})$ as $r \times 1$ matrices.
	\begin{prop}[{ \cite[Proposition 2.22]{FinOTAs} }]\label{prop_phi_fun_exp}
		The following asymptotic holds
		\begin{equation}\label{eq_phi_fun_exp1}
			\xi_E^{M|H}(\psi^{M|H}_{y_0}(Z)) = O(|Z|^2).
		\end{equation}
		Also, derivatives of $\xi_E^{M|H}$ are bounded uniformly on $y_0 \in Y$, $|Z| \leq R$.
		If, moreover, we assume that $(E, \nabla^E, h^E) := (L, \nabla^L, h^L)$ is a line bundle, and that there is a skew-adjoint endomorphism $Q$ of $TM$, which is parallel with respect to $\nabla^{TM}$ (i.e. $\nabla^{TM} Q = 0$), which commutes with $A^{M|H}$, the restriction of which to $H$ respects the decomposition (\ref{eq_tx_rest}), and such that for the curvature $R^L$ of $\nabla^L$, and for any $u, v \in TM$, we have
		\begin{equation}\label{eq_rl_q_op}
			\frac{\imun}{2 \pi}
			R^L (u, v) 
			=
			g^{TM}(Q u, v),
		\end{equation}
		then the following more precise bound holds
		\begin{equation}\label{eq_phi_fun_exp12}
		\begin{aligned}
			& \xi_L^{M|H}(\psi^{M|H}(Z)) =-\frac{1}{4} R^L_{y_0} \big(Z_N, A^{M|H}(Z_H)Z_H \big)
			+
			O(|Z|^4).
		\end{aligned}
		\end{equation}
	\end{prop}
	\begin{rem}
		Assume $(M, g^{TM})$ is endowed with a complex structure $J$, and $g^{TM}$ is invariant under the action of it. 
		Assume, moreover, that (\ref{eq_rl_q_op}) holds for $Q := J$ as in (\ref{eq_gtx_def}). 
		Then all the requirements are satisfied for $Q := J$, cf. \cite[Remark 2.23]{FinOTAs}.
	\end{rem}

\subsection{Towers of embeddings, associated coordinates and holonomies}\label{sect_tower}
	In this section, we compare natural coordinate systems and trivializations of vector bundles associated to towers of submanifolds.
	More precisely, we fix a tower of (real) manifolds $H \hookrightarrow K \hookrightarrow M$ of dimensions $m, l$ and $n$ respectively.
	Endow $M$ with the Riemannian metric $g^{TM}$ and induce the metrics $g^{TK}$ and $g^{TH}$ on $K$ and $H$.
	We fix $y_0 \in H$ and an orthonormal frame $(e_1, \ldots, e_m)$ (resp. $(e_{m+1}, \ldots, e_l)$, $(e_{l+1}, \ldots, e_n)$) of $(T_{y_0}H, g^{TH})$ (resp. $(N_{y_0}^{K|H}, g_{y_0}^{N^{K|H}})$, $(N_{y_0}^{M|K}, g_{y_0}^{N^{M|K}})$). 
	Recall that Fermi coordinates were defined in (\ref{eq_defn_fermi}).
	Clearly, there is a (unique) embedding $\sigma : \real^{l} \to \real^{n}$, $\sigma(0) = 0$, such that for any $Z \in \real^{l}$, small enough, we have
	\begin{equation}\label{eq_defn_sigma}
		\psi_{y_0}^{M|H}(\sigma(Z))
		=
		\psi_{y_0}^{K|H}(Z).
	\end{equation}
	\begin{prop}\label{prop_sigma_texp}
		The function $\sigma$ satisfies $\sigma(Z) = Z + O(|Z|^2)$.
		Moreover, there is $\epsilon > 0$, such that the derivatives of $\sigma$ are bounded uniformly on $y_0 \in H$, $|Z| \leq \epsilon$.
	\end{prop}
	\begin{proof}
		From (\ref{eq_h_defn_tr_m}) and (\ref{eq_defn_sigma}), we deduce
		\begin{equation}\label{eq_sigma_texp1}
			\psi_{y_0}^{K|H}(Z)
			=
			\phi_{y_0}^{M}\big(h^{M|K}(h^{K|H}(Z))\big).
		\end{equation}
		Hence, from (\ref{eq_sigma_texp1}), we infer
		\begin{equation}
			\sigma(Z)
			=
			(h^{M|H})^{-1}
			\big(h^{M|K}
			\big(
			h^{K|H}(Z)
			\big)\big).
		\end{equation}
		From this and Proposition \ref{prop_diff_exp}, we conclude.
	\end{proof}
	\par 
	Next, we compare two natural trivializations of vector bundles for towers of submanifolds associated to Fermi coordinates: one for the pair $(M, H)$, another one for $(K, H)$.
	\par 
	Let $(E, \nabla^E, h^{E})$ be a Hermitian vector bundle of bounded geometry and rank $r$ over $(M, g^{TM})$.
	We fix an orthonormal frame $f_1, \ldots, f_r \in (E_{y_0}, h^E_{y_0})$. 
	Recall that the frames $\tilde{f}_1^{M|H}, \ldots, \tilde{f}_r^{M|H}$; $\tilde{f}_1^{K|H}, \ldots, \tilde{f}_r^{K|H}$ were defined before (\ref{eq_frame_tilde}).
	Let $\tau_E$ be the (unique) function, defined in $B_{y_0}^{M}(R)$, with values in $\enmr{\comp^r}$, such that $\tau_E(0) = 0$, and 
	\begin{equation}\label{eq_frame_tilde2}
		\res_K (\tilde{f}_1^{M|H}, \ldots, \tilde{f}_r^{M|H}) = \exp(\tau_E) \cdot (\tilde{f}_1^{K|H}, \ldots, \tilde{f}_r^{K|H}),
	\end{equation}
	where we view $\res_K (\tilde{f}_1^{M|H}, \ldots, \tilde{f}_r^{M|H}) $ and $(\tilde{f}_1^{K|H}, \ldots, \tilde{f}_r^{K|H})$ as $r \times 1$ matrices.
	\begin{prop}\label{prop_tau}
		In the notations of Proposition \ref{prop_phi_fun_exp}, the following asymptotic holds
		\begin{equation}\label{eq_tau_1}
			\tau_E(\psi^{K|H}(Z)) = O(|Z|^2), \qquad \tau_L(\psi^{K|H}(Z)) = O(|Z|^4).
		\end{equation}
		Moreover, there is $\epsilon > 0$, such that the derivatives of $\tau_E$ are bounded uniformly on $y_0 \in H$, $|Z| \leq \epsilon$.
	\end{prop}
	\begin{proof}
		The proof is done by a repetitive use of Propositions \ref{prop_diff_exp}, \ref{prop_phi_fun_exp}, \ref{prop_sigma_texp}.
		By (\ref{eq_frame_tilde}), we have
		\begin{equation}\label{eq_tilde_tower_fi}
		\begin{aligned}
			&
			(\tilde{f}_1^{M|H}, \ldots, \tilde{f}_r^{M|H}) = \exp(\xi_E^{M|H}) \cdot (\tilde{f}'_1{}^M, \ldots, \tilde{f}'_r{}^M),
			\\
			&
			(\tilde{f}'_1{}^M, \ldots, \tilde{f}'_r{}^M) = \exp(-\xi_E^{M|K}) (\tilde{f}_1^{M|K}, \ldots, \tilde{f}_r^{M|K}),
			\\
			&
			\res_K (\tilde{f}_1^{M|K}, \ldots, \tilde{f}_r^{M|K}) = (\tilde{f}'_1{}^K, \ldots, \tilde{f}'_r{}^K),
			\\
			&
			(\tilde{f}'_1{}^K, \ldots, \tilde{f}'_r{}^K) = \exp(-\xi_E^{K|H}) (\tilde{f}_1^{K|H}, \ldots, \tilde{f}_r^{K|H}).
		\end{aligned}
		\end{equation}
		Now, from (\ref{eq_frame_tilde2}) and  (\ref{eq_tilde_tower_fi}), we obtain 
		\begin{equation}
			\exp(\tau_E) 
			=
			\res_K \big( \exp(\xi_E^{M|H}) \cdot \exp(-\xi_E^{M|K}) \big)
			\cdot
			\exp(-\xi_E^{K|H}).
		\end{equation}
		From Propositions \ref{prop_diff_exp}, \ref{prop_phi_fun_exp}, \ref{prop_sigma_texp}, we deduce the first part of (\ref{eq_tau_1}).
		Remark the basic identity
		\begin{equation}\label{eq_sec_f_form_trans}
			A^{M|K}(U)V + A^{K|H}(U)V = A^{M|H}(U)V, \qquad \text{for} \quad U, V \in TH.
		\end{equation}
		The second part is now obtained by the same means, one only has to use (\ref{eq_sec_f_form_trans}) in addition to previous considerations.
	\end{proof}
	\par 
	We will now relate the Fermi coordinates for the pairs $(M, K)$ and $(M, H)$.
	Clearly, there is a unique diffeomorphism $\upsilon : \real^{n} \to \real^{n}$, $\upsilon(0) = 0$, so that for any $Z \in \real^{n}$ small enough, we have
	\begin{equation}\label{eq_defn_ups}
		\psi_{y_0}^{M|K}(\upsilon(Z))
		=
		\psi_{y_0}^{M|H}(Z).
	\end{equation}
	\begin{prop}\label{prop_ups_texp}
		The function $\upsilon$ has the following Taylor expansion
		\begin{equation}\label{eq_ups_texp1}
			\upsilon(Z)
			=
			Z + O(|Z|^2).
		\end{equation}
			Moreover, there is $\epsilon > 0$, such that the derivatives of $\upsilon$ are bounded uniformly on $y_0 \in H$, $|Z| \leq \epsilon$.
		\end{prop}
		\begin{proof}
			By (\ref{eq_h_defn_tr_m})  and (\ref{eq_defn_ups}), we have
			\begin{equation}\label{eq_upsilon_form}
				\upsilon (Z)
				=
				(h^{M|K})^{-1} 
				\big( h^{M|H}(Z) \big).
			\end{equation}
			We conclude from Proposition \ref{prop_diff_exp} and (\ref{eq_upsilon_form}).
		\end{proof}
		\par 
		Our next goal is to compare two trivializations of a vector bundle: one associated to the trivializations in Fermi coordinates for the pair $(M, K)$, another one for $(M, H)$.
		Let $\chi_E$ be the (unique) function, defined over $B_{y_0}^{M}(R)$, with values in $\enmr{\comp^r}$, such that $\chi_E(0) = 0$ and
		\begin{equation}\label{eq_frame_chi}
			(\tilde{f}_1^{M|K}, \ldots, \tilde{f}_r^{M|K}) = \exp(\chi_E) \cdot (\tilde{f}_1^{M|H}, \ldots, \tilde{f}_r^{M|H}),
		\end{equation}
		where we view $(\tilde{f}_1^{M|K}, \ldots, \tilde{f}_r^{M|K}) $ and $(\tilde{f}_1^{M|H}, \ldots, \tilde{f}_r^{M|H})$ as $r \times 1$ matrices.
		\begin{prop}\label{prop_chi}
			In the notations of Proposition \ref{prop_phi_fun_exp}, the following asymptotic holds
			\begin{equation}\label{eq_chi_1}
				\chi_E(\psi^{M|H}(Z)) =  O(|Z|^2), \qquad \chi_L(\psi^{M|H}(Z)) =  O(|Z|^3).
			\end{equation}
			Moreover, there is $\epsilon > 0$, such that the derivatives of $\chi_E$, are bounded uniformly on $y_0 \in H$, $|Z| \leq \epsilon$.
		\end{prop}
		\begin{proof}
			From (\ref{eq_tilde_tower_fi}), (\ref{eq_defn_ups}) and (\ref{eq_frame_chi}), we deduce that
			\begin{equation}\label{eq_chi_form_pres}
				\exp(\chi_E)(\psi^{M|H}(Z))
				=
				\exp(\xi_E^{M|K})
				\big(
					\psi^{M|K}(\upsilon(Z))
				\big)
				\cdot
				\exp(-\xi_E^{M|H})
				(
					\psi^{M|H}(Z)
				).
			\end{equation}
			The result now follows Propositions \ref{prop_phi_fun_exp}, \ref{prop_chi} and (\ref{eq_chi_form_pres}).
		\end{proof}

\section{Asymptotics of Toeplitz type operators and kernel calculus}\label{sect_asymp_toepl_type}

	The main goal of this section is to study asymptotics of Schwartz kernels of Toeplitz type operators.
	More precisely, in Section \ref{sect_model_calc}, we consider the model situation, for which an explicit formula for the Schwartz kernels of Bergman projectors, the extension and restriction operators can be given. 
	We then study the composition rules for the operators with related kernels.
	Those results play the foundational role in this article.
	In Section \ref{sect_schw_kern_ber_ext}, we recall the asymptotics of Schwartz kernels of Bergman projectors and the extension operator.
	In Section \ref{sect_asympt_study_toepl}, we study the asymptotic expansion of basic Toeplitz type operators, in particular, we establish Theorem \ref{thm_ttype_as}.
	Finally, in Section \ref{sect_as_crit}, we establish asymptotic characterization of Toeplitz type operators with exponential decay.
	
	\subsection{Model operators on the complex vector space}\label{sect_model_calc}
		In this section, we consider the model situation, for which an explicit formula for the Schwartz kernels of Bergman projectors, the extension and restriction operators can be given.
		We then use those explicit formulas to give a description for compositions of operators, the Schwartz kernels of which can be expressed using the above kernels.
		This section is motivated in many ways by the works of Ma-Marinescu \cite{MaMarToepl}, \cite{MaHol} and Dai-Liu-Ma \cite{DaiLiuMa}.
		\par 
		Endow $\comp^n$ with the standard Riemannian metric and consider a trivialized holomorphic line bundle $L_0$ on $\comp^n$.
		We endow $L_0$ with the Hermitian metric $h^{L_0}$, given by 
		\begin{equation}\label{eq_model_metrl}
			\| 1 \|_{h^{L_0}}(Z) = \exp \Big(- \frac{\pi}{2} |Z|^2 \Big),
		\end{equation}
		where $Z$ is the natural real coordinate on $\comp^n$, and $1$ is the trivializing section of $L_0$.
		An easy verification shows that (\ref{eq_model_metrl}) implies that (\ref{eq_gtx_def}) holds in our setting.
		Recall that \cite[\S 4.1.6]{MaHol} shows that the Kodaira Laplacian on $\ccal^{\infty}(X, L_{0})$, multiplied by 2, which we denote here by $\mathscr{L}$, and view as an operator on $\ccal^{\infty}(X)$ using the orthonormal trivialization by $1 \exp (\frac{\pi}{2} |Z|^2)$ of $L_0$,  is given by
		\begin{equation}\label{eq_mathscr_l_op}
			\mathscr{L} = \sum_{i = 1}^{n} b_i b_i^{+},
		\end{equation}
		where $b_i$, $b_i^{+}$ are \textit{creation} and \textit{annihilation} operators, defined as
		\begin{equation}\label{eq_creat_annih}
			b_i = -2 \frac{\partial}{\partial z_i} + \pi \overline{z}_i, \qquad b_i^{+} = 2 \frac{\partial}{\partial  \overline{z}_i} + \pi z_i. 
		\end{equation}
		We verify easily that 
		\begin{equation}\label{eq_bj_commut}
			\big[ g(z, \overline{z}), b_j \big] = 2 \frac{\partial}{\partial z_j} g(z, \overline{z}).
		\end{equation}
		\par 
		From \cite[Theorem 4.1.20]{MaHol}, we know that for a multiindex $\alpha \in \nat^{n}$, the functions 
		\begin{equation}\label{eq_spec_l_d}
			b^{\alpha} \Big(
				z^{\beta} \exp( - \frac{\pi}{2} \sum_{i = 1}^{n} |z_i|^2)
			\Big),
		\end{equation}
		form vector spaces of orthogonal eigenvectors of $\mathscr{L}$ (viewed as sections of $\ccal^{\infty}(X, L_{0})$ using the above orthonormal trivialization) corresponding to the eigenvalues $4 \pi \sum_{i = 1}^{n} \alpha_i$.
		From this, cf. \cite[Theorem 1.15]{MaMar08a}, cf. \cite[(4.1.84)]{MaHol}, the Bergman kernel $\mathscr{P}_{n}$ of $\comp^n$ is given by
		\begin{equation}\label{eq_berg_k_expl}
			\mathscr{P}_n(Z, Z') = \exp \Big(
				-\frac{\pi}{2} \sum_{i = 1}^{n} \big( 
					|z_i|^2 + |z'_i|^2 - 2 z_i \overline{z}'_i
				\big)
			\Big), \quad \text{for } Z, Z' \in \comp^n.
		\end{equation}
		In particular, we deduce that
		\begin{equation}\label{eq_bj_pn}
			b_{j, z} \mathscr{P}_n(Z, Z')
			=
			2 \pi (\overline{z}_j - \overline{z}'_j) \mathscr{P}_n(Z, Z').
		\end{equation}
		\par 
		Also, we see easily, cf. \cite[(3.28), (3.29)]{FinOTAs}, that Schwartz kernels of the orthogonal Bergman kernel, $\mathscr{P}_{n, m}^{\perp}$, corresponding to the projection onto holomorphic sections orthogonal to the holomorphic sections which vanish along $\comp^m$, and the $L^2$-extension operator $\mathscr{E}_{n, m}$, extending each element from $(\ker \mathscr{L})|_{\comp^m}$ to an element from $\ker \mathscr{L}$ with the minimal $L^2$-norm, are given by
		\begin{equation}\label{eq_pperp_defn_fun}
		\begin{aligned}
			&
				\mathscr{P}_{n, m}^{\perp}(Z, Z')
				=
				 \exp \Big(
				 	-\frac{\pi}{2} \sum_{i = 1}^{m} \big( 
						|z_i|^2 + |z'_i|^2 - 2 z_i \overline{z}'_i
					\big)
					-
					\frac{\pi}{2}
					 \sum_{i = m+1}^{n} \big( |z_i|^2 + |z'_i|^2 \big)
				\Big),
				\\
				&\mathscr{E}_{n, m}(Z, Z'_Y)
				=
				 \exp \Big(
					-\frac{\pi}{2} \sum_{i = 1}^{m} \big( 
						|z_i|^2 + |z'_i|^2 - 2 z_i \overline{z}'_i
					\big)
					-
					\frac{\pi}{2}
					 \sum_{i = m+1}^{n} |z_i|^2
					\Big).
			\end{aligned}
			\end{equation}
	We, finally, remark that the Schwartz kernel of the operator $\res_{\comp^m} \circ \mathscr{P}_n$ is given by 
	\begin{equation}\label{eq_res_nm_kernel}
		\mathscr{Res}_{n, m}(Z_Y, Z') = 
				\exp \Big(
					-\frac{\pi}{2} \sum_{i = 1}^{m} \big( 
						|z_i|^2 + |z'_i|^2 - 2 z_i \overline{z}'_i
					\big)
					-
					\frac{\pi}{2}
					\sum_{i = m+1}^{n} |z'_i|^2
				\Big).
	\end{equation}
	Remark the trivial identity
	\begin{equation}\label{eq_res_enmdual}
		\mathscr{Res}_{n, m}
		=
		\mathscr{E}_{n, m}^*.
	\end{equation}
	\par 
	Now, a lot of calculations in this article will have something to do with compositions of operators having Schwartz kernels, given by the product of polynomials with the above kernels.
	For that reason, the following lemma will be of utmost importance in what follows.
	\begin{lem}\label{lem_comp_poly}
			For any polynomials $A_1(Z, Z'), A_2(Z, Z')$, $Z, Z' \in \real^{2n}$, there is a polynomial $A_3 := \mathcal{K}_{n, m}[A_1, A_2]$, the coefficients of which are polynomials of the coefficients of $A_1, A_2$, such that
			\begin{equation}\label{eq_lem_comp_poly_1}
				(A_1 \cdot \mathscr{P}_{n, m}^{\perp}) \circ (A_2 \cdot \mathscr{P}_{n, m}^{\perp})
				=
				A_3 \cdot \mathscr{P}_{n, m}^{\perp}.
			\end{equation}
			Moreover, $\deg A_3 \leq \deg A_1 + \deg A_2$. Also, if both polynomials $A_1$, $A_2$ are even or odd (resp. one is even, another is odd), then the polynomial $A_3$ is even (resp. odd).
			Similarly, there is a polynomial $A'_3 := \mathcal{K'}_{n, m}[A_1, A_2]$ with the same properties as $A_3$, such that 
			\begin{equation}\label{eq_lem_comp_poly_2}
				(A_1 \cdot \mathscr{P}_n) \circ (A_2 \cdot \mathscr{P}_{n, m}^{\perp})
				=
				A'_3 \cdot \mathscr{P}_{n, m}^{\perp}.
			\end{equation}
			Also, for any polynomials $A(Z, Z'_Y)$, $D(Z_Y, Z'_Y)$, $Z \in \real^{2n}$, $Z_Y, Z'_Y \in \real^{2m}$, there is a polynomial $A''_3 := \mathcal{K}_{n, m}^{EP}[A, D]$ with the same properties as $A_3$, such that 
			\begin{equation}\label{eq_comp_ext_pmn222}
				(A \cdot \mathscr{E}_{n, m}) \circ (D \cdot\mathscr{P}_{m})
				=
				A''_3 \cdot \mathscr{E}_{n, m}.
			\end{equation}
			For polynomials $B(Z, Z')$, $C(Z_Y, Z')$, $Z, Z' \in \real^{2n}$, $Z_Y \in \real^{2m}$, the following identities hold
			\begin{equation}\label{eq_comp_ext_pmn1}
			\begin{aligned}
				&
				(B \cdot \mathscr{P}_{n}) \circ (A \cdot \mathscr{E}_{n, m})
				=
				(\mathcal{K}_{n, n}[B, A] \circ \res_m ) \cdot \mathscr{E}_{n, m},
				\\
				&
				(B \cdot \mathscr{P}_{n, m}^{\perp}) \circ (A \cdot \mathscr{E}_{n, m})
				=
				( \mathcal{K}_{n, m}[B, A] \circ \res_m )  \cdot \mathscr{E}_{n, m},
				\\
				&
				(C \cdot \mathscr{Res}_{n, m}) \circ (A \cdot \mathscr{E}_{n, m})
				=
				( \res_m \circ \mathcal{K}_{n, m}[C, A] \circ \res_m )  \cdot \mathscr{P}_{m, m},
			\end{aligned}
			\end{equation}
			Finally, for any polynomials $A_4(Z, Z'_W)$, $A_5(Z_W, Z'_Y)$, $Z \in \real^{2n}$, $Z'_Y \in \real^{2m}$, $Z_W, Z'_W \in \real^{2l}$, there is a polynomial $A'''_3 := \mathcal{K}_{n, m}^{E}[A_4, A_5]$, with the same properties as $A_3$, such that
			\begin{equation}\label{eq_comp_diff_ext_1}
				(A_4 \cdot \mathscr{E}_{n, l}) \circ (A_5 \cdot \mathscr{E}_{l, m})
				=
				A'''_3 \cdot \mathscr{E}_{n, m}.
			\end{equation}
		\end{lem}
		\begin{rem}
			The statement (\ref{eq_lem_comp_poly_1}) for $n = m$ is due to Ma-Marinescu \cite[Lemma 7.1.1, (7.1.6)]{MaHol}.
		\end{rem}
		\begin{proof}
			We decompose polynomials $A_1$, $A_2$ as follows
			\begin{equation}\label{eq_compa_3}
				A_1(Z, Z')
				=
				\sum_{\alpha} Z_N^{\alpha} \cdot A_1^{\alpha}(Z_Y, Z'), 
				\qquad
				A_2(Z, Z')
				=
				\sum_{\alpha'} A_2^{\alpha'}(Z, Z'_Y) Z'_N{}^{\alpha'},
			\end{equation}
			where $\alpha, \alpha' \in \nat^{2(n - m)}$ verify $|\alpha| \leq \deg A_1, |\alpha'| \leq \deg A_2$.
			In \cite[(3.36)]{FinOTAs}, we established the first statement and proved that
			\begin{equation}\label{eq_knm_from_knn}
				 \mathcal{K}_{n, m}[A_1, A_2]
				 =
				 \sum_{\alpha} \sum_{\alpha'} Z_N^{\alpha} Z'_N{}^{\alpha'}
				  \mathcal{K}_{n, n}[A_1^{\alpha}, A_2^{\alpha'}].
			\end{equation}
			Along the same lines, we obtained in \cite[(3.38)]{FinOTAs} the second statement and
			\begin{equation}\label{eq_knmpr_from_knn}
					\mathcal{K'}_{n, m}[A_1, A_2](Z, Z')
					=
					\sum_{\alpha'} Z'_N{}^{\alpha'}
					\cdot
					\mathcal{K}_{n, n}[A_1, A_2^{\alpha'}](Z, Z'_Y).
			\end{equation}
			\par 
			To get the third statement, we represent $A(Z, Z'_Y) := \sum Z_N^{\alpha} \cdot A^{\alpha}(Z_Y, Z'_Y)$. 
			Then, from (\ref{eq_berg_k_expl}) and (\ref{eq_pperp_defn_fun}), the following equation holds
			\begin{equation}
				(A^{\alpha} \cdot \mathscr{E}_{n, m}) \circ (D \cdot \mathscr{P}_{m})
				=
				 \exp \Big(- \frac{\pi}{2} |Z_N|^2 \Big)
				 \cdot
				(A^{\alpha} \cdot \mathscr{P}_{m}) \circ (D \cdot \mathscr{P}_{m}).
			\end{equation}
			By this and (\ref{eq_lem_comp_poly_1}), we clearly have (\ref{eq_comp_ext_pmn222}) for 
			\begin{equation}\label{eq_kep_formula}
				\mathcal{K}_{n, m}^{EP}[A, D]
				=
				\sum_{\alpha} Z_N^{\alpha} \cdot \mathcal{K}_{m, m}[A^{\alpha}, D].
			\end{equation}
			Now, we note that an easy verification, based on (\ref{eq_berg_k_expl}), (\ref{eq_pperp_defn_fun}) and (\ref{eq_res_nm_kernel}), shows that 
			\begin{equation}
			\begin{aligned}
				&
				\Big( (B \cdot \mathscr{P}_n) \circ (A \cdot \mathscr{E}_{n, m}) \Big)(Z, Z'_Y)
				=
				\Big( (B \cdot \mathscr{P}_n) \circ (A \cdot \mathscr{P}_{n}) \Big)(Z, Z'_Y),
				\\
				&
				\Big( (B \cdot \mathscr{P}_{n, m}^{\perp}) \circ (A \cdot \mathscr{E}_{n, m}) \Big)(Z, Z'_Y)
				=
				\Big( (B \cdot \mathscr{P}_{n, m}^{\perp}) \circ (A \cdot \mathscr{P}_{n, m}^{\perp}) \Big)(Z, Z'_Y),
				\\
				&
				\Big( (C \cdot \mathscr{Res}_{n, m}) \circ (A \cdot \mathscr{E}_{n, m}) \Big) (Z_Y, Z'_Y)
				=
				\Big( (C \cdot \mathscr{P}_{n, m}^{\perp}) \circ (A \cdot \mathscr{P}_{n, m}^{\perp}) \Big) (Z_Y, Z'_Y).
			\end{aligned}
			\end{equation}
			This clearly implies (\ref{eq_comp_ext_pmn1}) by (\ref{eq_lem_comp_poly_1}) and (\ref{eq_lem_comp_poly_2}).
			\par
			It is now only left to prove the fifth statement. For this, for $Z = (Z_Y,Z_{N^{W|Y}}, Z_{N^{X|W}})$, $Z_Y, Z'_Y \in \real^{2m}$, $Z_{N^{W|Y}} \in \real^{2(l - m)}$, $Z_{N^{X|W}}  \in \real^{2(n - l)}$, $Z'_W \in \real^{2l}$, $Z_W := (Z_Y,Z_{N^{W|Y}})$, we decompose the polynomial $A_4$ as follows
			\begin{equation}\label{eq_a4_alt_epxres}
				A_4(Z, Z'_W)
				=
				\sum_{\alpha} Z_{N^{X|W}}^{\alpha} \cdot A_4^{\alpha}(Z_W, Z'_W).
			\end{equation}
			An easy verification shows that 
			\begin{equation}
				\Big( (A_4^{\alpha} \cdot \mathscr{E}_{n, l}) \circ (A_5 \cdot \mathscr{E}_{l, m}) \Big) (Z, Z'_Y)
				=
				\exp \Big(- \frac{\pi}{2} |Z_{N^{X|W}}|^2 \Big)
				\cdot
				\Big( (A_4^{\alpha} \cdot \mathscr{P}_l) \circ (A_5 \cdot \mathscr{E}_{l, m}) \Big) (Z_W, Z'_Y).
			\end{equation}
			From (\ref{eq_comp_ext_pmn1}), we obtain
			\begin{equation}\label{eq_knme_form}
				\mathcal{K}_{n, m}^{E}[A_4, A_5]
				=
				\sum_{\alpha} Z_N^{\alpha} \cdot ( \mathcal{K}_{l, l}[A_4^{\alpha}, A_5] \circ \res_m ),
			\end{equation}
			which finishes the proof.
			The statements about the degrees of $A_3$, etc., follow from the validity of the corresponding statements for $\mathcal{K}_{n, n}$, proved by Ma-Marinescu in \cite{MaMarToepl} and expressions (\ref{eq_knm_from_knn}), (\ref{eq_knmpr_from_knn}), (\ref{eq_kep_formula}), (\ref{eq_a4_alt_epxres}), (\ref{eq_knme_form}).
		\end{proof}
		\begin{sloppypar}
			From the above, we see that to compute the polynomials from Lemma \ref{lem_comp_poly}, it suffices to give an algorithm for the calculation of $\mathcal{K}_{n, m}$.
			Below, we explain how to do this. 
			Directly from the definitions, we see that $\mathcal{K}_{n, m}[1 \cdot P(Z'), A] = \mathcal{K}_{n, m}[1, P(Z) \cdot A]$ for any polynomial $A$.
			Also, we trivially have $\mathcal{K}_{n, m}[P(Z) \cdot A(Z, Z'), A'(Z, Z')] = P(Z) \mathcal{K}_{n, m}[A(Z, Z'), A'(Z, Z')]$ for any polynomials $P, A, A'$.
			Hence, it is enough to give an algorithm for the calculation of $\mathcal{K}_{n, m}$ where the first argument is given by $1$.
			For this, remark that for any $i = 1, \ldots, n$, $a, b \in \nat$, we have
			\begin{equation}\label{eq_k_calc_2}
				\mathcal{K}_{n, m}[1, P_i(Z) z_i^a \overline{z}_i^b]
				=
				\mathcal{K}_{n, m}[1, P_i(Z)]
				\cdot
				\mathcal{K}_{n, m}[1, z_i^a \overline{z}_i^b],
			\end{equation}
			where the polynomial $P_i(Z)$ doesn't depend on $z_i$ and $\overline{z}_i$.
			Hence, to understand $\mathcal{K}_{n, m}$, it suffices to know how to calculate it for polynomials $z_i^a \overline{z}_i^b$.
			We describe below the general formula.
			\par Using (\ref{eq_bj_commut}) and (\ref{eq_bj_pn}), we see that for any $a, b \in \nat$, $i \leq n$, we get
			\begin{equation}
				\mathcal{P}_n \circ (z_i^a \overline{z}_i^b \cdot \mathcal{P}_n)
				=
				\mathcal{P}_n \circ \Big( z_i^a \cdot \big( \frac{b_i}{2 \pi} + \overline{z}'_i \big)^b  \mathcal{P}_n \Big).
			\end{equation}
			Remark that by (\ref{eq_spec_l_d}), we have $\mathcal{P}_n \circ (b^{\alpha} z^{\beta} \cdot \mathcal{P}_n) = 0$ as long as $\alpha \neq 0$. Hence, from (\ref{eq_bj_commut}), we have
			\begin{equation}
				\mathcal{P}_n \circ \Big( z_i^a \cdot b_i^k  \mathcal{P}_n \Big)
				=
				\begin{cases}
				
				\frac{a! 2^k}{(a-k)!}
				z_i^{a-k} \mathcal{P}_n,
				&
				\text{for } a \geq k, 
				\\
				0,
				&
				\text{otherwise.}
				\end{cases}
			\end{equation}
			From this, for $i \leq m$, we deduce directly the following identity
			\begin{equation}\label{eq_kmn_poly}
				\mathcal{K}_{n, m}[1, z_i^a \overline{z}_i^b]
				=
				\sum_{l + k = b}
				\frac{1}{\pi^{k}}
				 \frac{a! b!}{(a - k)! l! k!}
				 z_i^{a - k} \overline{z}'_i{}^{l}
				.
			\end{equation}
			Recall the following famous integral calculation: for $z = x + \imun y$, we have
			\begin{equation}\label{eq_gauss_integral}
				\int_{\comp} \exp(- \pi |z|^2) z^a \overline{z}^b dx dy = \delta_{ab} \frac{a!}{\pi^a}.
			\end{equation}			 
			For $m+1 \leq i \leq n$, (\ref{eq_gauss_integral}) shows that
			\begin{equation}\label{eq_kmn_poly23}
				\mathcal{K}_{n, m}[1, z_i^a \overline{z}_i^b]
				=
				\delta_{ab}
				\frac{a!}{\pi^a}
				.
			\end{equation}
		\end{sloppypar}
		\begin{cor}\label{cor_poly_incomp}
			Assume that for a polynomial $A(Z, Z'_Y)$, $Z \in \real^{2n}$, $Z'_Y \in \real^{2m}$, the following equality holds
			$
				\mathscr{P}_{n} \circ (A \cdot \mathscr{E}_{n, m}) \circ \mathscr{P}_{m}
				=
				A \cdot \mathscr{E}_{n, m}
			$.
			Then $A$ is a polynomial in $z, \overline{z}'_Y$.
		\end{cor}
		\begin{proof}
			First of all, our assumption clearly implies that $
				\mathscr{P}_{n} \circ (A \cdot \mathscr{E}_{n, m}) =
				A \cdot \mathscr{E}_{n, m}
			$.
			Hence, from (\ref{eq_comp_ext_pmn1}), we deduce $\mathcal{K}_{n, n}[1, A] \circ \res_m = A$.
			This along with (\ref{eq_kmn_poly}) imply that $A$ is a polynomial in $z$ and $Z'_Y$.
			Now, again, our assumption implies that $
				(A \cdot \mathscr{E}_{n, m}) \circ \mathscr{P}_{m} =
				A \cdot \mathscr{E}_{n, m}
			$.
			Hence $\mathcal{K}^{EP}_{n, n}[A, 1] = A$. This, in conjunction with (\ref{eq_kep_formula}) and (\ref{eq_kmn_poly}), implies that $A$ is a polynomial in $z$ and $\overline{z}'_Y$, which concludes the proof.
		\end{proof}

\subsection{Schwartz kernels of Bergman projectors and extension operator}\label{sect_schw_kern_ber_ext}

	The main goal of this section is to recall the results about the asymptotics of Schwartz kernels of Bergman projectors and the extension operator.
	We use notations from Section \ref{sect_intro} and assume that the triple $(X, Y, g^{TX})$ is of bounded geometry.
	Let us recall first the results about the exponential decay of those Schwartz kernels.
	\begin{thm}[{Dai-Liu-Ma \cite[Theorem 4.18]{DaiLiuMa}, Ma-Marinescu \cite[Theorem 1]{MaMarOffDiag}}]\label{thm_bk_off_diag}
		There are $c > 0$, $p_1 \in \nat^*$,  such that for any $k \in \nat$, there is $C > 0$, such that for any $p \geq p_1$, $x_1, x_2 \in X$, the following estimate holds
		\begin{equation}\label{eq_bk_off_diag}
			\big|  B_p^X(x_1, x_2) \big|_{\ccal^k} \leq C p^{n + \frac{k}{2}} \cdot \exp \big(- c \sqrt{p} \cdot \dist(x_1, x_2) \big),
		\end{equation}
		where $\ccal^k$-norm here is interpreted as in Definition \ref{defn_ttype}.
	\end{thm}
	\par 
	\begin{thm}[{\cite[Theorems 1.5, 1.8]{FinOTAs} }]\label{thm_ext_exp_dc}
		There are $c > 0$, $p_1 \in \nat^*$,  such that for any $k, l \in \nat$, there is $C > 0$, such that for any $p \geq p_1$, $x_1, x_2 \in X$, $y \in Y$, the following estimates hold
		\begin{equation}\label{eq_ext_exp_dc}
		\begin{aligned}
			&
			\big|  \ext_p^{X|Y}(x_1, y) \big|_{\ccal^k} \leq C p^{m + \frac{k}{2}} \exp \big(- c \sqrt{p} \cdot \dist(x_1, y) \big),
			\\
			&
			\big|  B_p^{X|Y \perp}(x_1, x_2) \big|_{\ccal^k} \leq C p^{n + \frac{k}{2}} \exp \big(- c \sqrt{p} \cdot (  \dist(x_1, x_2) + \dist(x_1, Y) + \dist(x_2, Y) ) \big).
		\end{aligned}
		\end{equation}
	\end{thm}
	\par Let us now give the first direct application of Theorems \ref{thm_bk_off_diag}, \ref{thm_ext_exp_dc}.
	 Assume that we start with two different choices of functions $\rho_1, \rho_2$ as in (\ref{defn_rho_fun}), and form two different brackets $\llangle \cdot \rrangle_1$, $\llangle \cdot \rrangle_2$, as in (\ref{eq_brack_defn}), corresponding to those choices.
	\begin{cor}\label{cor_brack_indep}
		There is $p_1 \in \nat^*$, such that for any $g \in \ccal^{\infty}_{b}(Y, {\rm{Sym}}^k (N^{X|Y})^*)$, $k \in \nat$, there are $c, C > 0$, such that for any $p \geq p_1$, $x \in X$, $y \in Y$, we have
		\begin{equation}
		\begin{aligned}
			&
			\Big| 
				T_{\llangle g \rrangle_1}^{X|Y}(x, y) - T_{\llangle g \rrangle_2}^{X|Y}(x, y)
			\Big|_{\ccal^l} 
			\leq 
			C
			\exp \big( - c \sqrt{p} \cdot ( 1 + \dist(x, y) ) \big),
			\\
			&
			\Big| 
				T_{\llangle g \rrangle_1}^{Y|X}(y, x) - T_{\llangle g \rrangle_2}^{Y|X}(y, x)
			\Big|_{\ccal^l} 
			\leq 
			C 
			\exp \big( - c \sqrt{p} \cdot ( 1 + \dist(x, y) ) \big).
		\end{aligned}
		\end{equation}
		In particular, Definition \ref{defn_ttype} ultimately doesn't depend on the choice of $\rho$.
	\end{cor}
	\begin{proof}
		Since the support of the function $\rho_1 (\frac{|Z_N|}{r_{\perp}}) - \rho_2 (\frac{|Z_N|}{r_{\perp}})$ is located away from a neighborhood of $Y$, by the boundness of the function $u^k \exp(-u)$, $k \in \nat$, $u \in \real_+$, we deduce from Theorem \ref{thm_ext_exp_dc} that the following estimate for the Schwartz kernel is satisfied
		\begin{equation}\label{eq_diff_brack_1}
			\big( ( \llangle g \rrangle_1 - \llangle g \rrangle_2 ) \cdot R_p \big) (x, y) 
			\leq 
			C
			\cdot 
			\exp \big( - c \sqrt{p} \cdot ( 1 + \dist(x, y) ) \big),
		\end{equation}
		where $R_p$ is $\ext_p^{X|Y}$, $B_p^X$ or $B_p^{X|Y \perp}$.
		We conclude by Theorem \ref{thm_bk_off_diag}, Corollary \ref{cor_comp_exp_bound} and (\ref{eq_diff_brack_1}).
	\end{proof}
	\begin{sloppypar} 
	Theorem \ref{thm_ext_exp_dc} shows that to understand fully the asymptotics of the Schwartz kernel of the Bergman projector (resp. orthogonal Bergman projector and the extension operator), it suffices to do so in a neighborhood of a fixed point on the diagonal of $X$ (resp. $Y$), embedded in $X \times X$ (resp.   $X \times X$ and  $X \times Y$). 
	Let us recall the results in this direction, showing that Schwartz kernels of our operators are essentially equal, up to a recalling, to Schwartz kernel of the model operators considered in Section \ref{sect_model_calc}. 
	Before this, let us fix some notation.
	\par 
	We fix $x_0 \in X$ and an orthonormal frame $(e_1, \ldots, e_{2n})$ of $(T_{x_0}X, g^{TX}_{x_0})$, verifying (\ref{eq_cond_jinv}).
	Recall that geodesic coordinates were defined in (\ref{eq_phi_defn}).
	Define the function $\kappa_{\phi, x_0}^{X} : B_0^{\real^{2n}}(r_X) \to \real$, by
	\begin{equation}\label{eq_defn_kappaxy1}
		((\phi_{x_0}^X)^* dv_X) (Z)
		=
		\kappa_{\phi, x_0}^{X}
		d Z_1 \wedge \cdots \wedge d Z_{2n}.
	\end{equation}
	Now, let $x_0 = y_0$, where $y_0 \in Y$, and $(e_1, \ldots, e_{2n})$ be as (\ref{eq_defn_fermi}). 
	Recall that Fermi coordinates were defined in (\ref{eq_defn_fermi}).
	Define the function $\kappa_{\psi, y_0}^{X|Y} : B_0^{\real^{2m}}(r_Y) \times B_0^{\real^{2(n - m)}}(r_{\perp}) \to \real$ by
	\begin{equation}\label{eq_defn_kappaxy2}
		((\psi_{y_0}^{X|Y})^* dv_X) (Z)
		=
		\kappa_{\psi, y_0}^{X|Y}
		d Z_1 \wedge \cdots \wedge d Z_{2n}.
	\end{equation}
	Recall that the function $\kappa_N^{X|Y}$ was defined in (\ref{eq_kappan}).
	Clearly, for $Z = (Z_Y, Z_N) \in \real^{2n}$, $Z_Y \in \real^{2m}$, we have the following relation between different $\kappa$-functions
	\begin{equation}\label{eq_kappa_relation}
		\kappa_{\psi, y_0}^{X|Y}(Z)
		=
		\kappa_N^{X|Y}(\psi_{y_0}^{X|Y}(Z))
		\cdot
		\kappa_{\phi, y_0}^{Y}(Z_Y).
	\end{equation}
	Also, under assumptions (\ref{eq_comp_vol_omeg}), we have $\kappa_{\psi, y_0}^{X|Y}(0) = \kappa_{\phi, y_0}^{Y}(0) = 1$.
	\par 
	Recall that the second fundamental form $A^{X|Y} \in \ccal^{\infty}(Y, T^*Y \otimes \enmr{TX|_Y})$ was defined in (\ref{eq_sec_fund_f}).
	Recall that the functions $\mathscr{P}_{n}$, $\mathscr{P}_{n, m}^{\perp}$, $\mathscr{E}_{n, m}$, were defined in (\ref{eq_berg_k_expl}), (\ref{eq_pperp_defn_fun}).
	\par 
	We fix an orthonormal frame $(f_1, \ldots, f_r)$ of $(F_{y_0}, h^{F}_{y_0})$ and define the orthonormal frames $(\tilde{f}_1^{X|Y}, \ldots, \tilde{f}_r^{X|Y})$, $(\tilde{f}'_1{}^{X}, \ldots, \tilde{f}'_r{}^{X})$, of $(F, h^F)$ in a neighborhood of $y_0$, as in Section \ref{sect_bnd_geom_cf}.
	\par \textbf{Notation.} For $g \in \ccal^{\infty}(X, F)$, by an abuse of notation, we write $g(\phi_{y_0}^X(Z)) \in \real^r$, $Z \in \real^{2n}$, $|Z| \leq R$, for coordinates of $g$ in the frame $(\tilde{f}'_1{}^{X}, \ldots, \tilde{f}'_r{}^{X})$. 
	We identify $g(\phi_{y_0}^X(Z))$ with an element in $F_{y_0}$ using the frame $(f_1, \ldots, f_r)$.
	Similarly, we denote by $g(\psi_{y_0}^{X|Y}(Z)) \in \real^r$ the coordinates in the frame $(\tilde{f}_1^{X|Y}, \ldots, \tilde{f}_r^{X|Y})$ and identify them with an element from $F_{y_0}$.
	Similar notations are used for sections of $F^*$, $F \otimes L^p$, $(F \otimes L^p)^*$, $F \boxtimes F^*$, etc.
	\end{sloppypar}
	\begin{thm}\label{thm_berg_off_diag}
		For any $r \in \nat$, $y_0 \in Y$, there are $J_r^{X|Y}(Z, Z') \in \enmr{F_{y_0}}$ polynomials in $Z, Z' \in \real^{2n}$, with the same parity as $r$ and $\deg J_r^{X|Y} \leq 3r$, 
		whose coefficients are polynomials in $\omega$, $R^{TX}$, $A^{X|Y}$, $R^F$, $(dv_X / dv_{g^{TX}})^{\pm \frac{1}{2n}}$,  $(dv_Y / dv_{g^{TY}})^{\pm \frac{1}{2n}}$,  and their derivatives of order $\leq 2r$, all evaluated at $y_0$, such that for the functions $F_r^{X|Y} := J_r^{X|Y} \cdot \mathscr{P}_{n}$ over $\real^{2n} \times \real^{2n}$, the following holds. 
		There are $\epsilon, c > 0$, $p_1 \in \nat^*$, such that for any $k, l, l' \in \nat$, there exists $C > 0$, such that for any $y_0 \in Y$, $p \geq p_1$, $Z, Z' \in \real^{2n}$, $|Z|, |Z'| \leq \epsilon$, $\alpha, \alpha' \in \nat^{2n}$, $|\alpha|+|\alpha'| \leq l$, $Q^1_{k, l, l'} := 3 (n + k + l' + 2) + l$:
			\begin{multline}\label{eq_berg_off_diag}
				\bigg| 
					\frac{\partial^{|\alpha|+|\alpha'|}}{\partial Z^{\alpha} \partial Z'{}^{\alpha'}}
					\bigg(
						\frac{1}{p^n} B_p^X\big(\psi_{y_0}^{X|Y}(Z), \psi_{y_0}^{X|Y}(Z') \big)
						\\
						-
						\sum_{r = 0}^{k}
						p^{-\frac{r}{2}}						
						F_r^{X|Y}(\sqrt{p} Z, \sqrt{p} Z') 
						\kappa_{\psi}^{X|Y}(Z)^{-\frac{1}{2}}
						\kappa_{\psi}^{X|Y}(Z')^{-\frac{1}{2}}
					\bigg)
				\bigg|_{\ccal^{l'}}
				\\
				\leq
				C p^{-\frac{k + 1 - l}{2}}
				\Big(1 + \sqrt{p}|Z| + \sqrt{p} |Z'| \Big)^{Q^1_{k, l, l'}}
				\exp\Big(- c \sqrt{p} |Z - Z'| \Big),
			\end{multline}
			where the $\ccal^{l'}$-norm is taken with respect to $y_0$.
			Also, the following identity holds
			\begin{equation}\label{eq_jo_expl_form}
				J_0^{X|Y}(Z, Z') = {\rm{Id}}_{F_{y_0}}.
			\end{equation}
			Moreover, under the assumption (\ref{eq_comp_vol_omeg}), we have
			\begin{multline}\label{eq_j1_expl_form}
				J_1^{X|Y}(Z, Z') = {\rm{Id}}_{F_{y_0}} \cdot \pi 
				\Big(
			 	g \big(z_N, A^{X|Y}(\overline{z}_Y - \overline{z}'_Y) (\overline{z}_Y - \overline{z}'_Y) \big)
			 	\\
			 	+
			 	g \big(\overline{z}'_N, A^{X|Y}(z_Y - z'_Y) (z_Y - z'_Y) \big)			 
			 	\Big).
			\end{multline}
	\end{thm}
	\begin{proof}
		For $X = Y$, the result is due to Dai-Liu-Ma \cite{DaiLiuMa} and the calculation of $J_1^{X|X}$ is due to Ma-Marinescu \cite[Remark 4.1.26]{MaHol}.
		The proof of the general case is done in \cite[Theorem 5.5]{FinOTAs} by relying on the result of \cite{DaiLiuMa} and some local calculations.
	\end{proof}
	\begin{thm}[{ \cite[Theorem 1.6]{FinOTAs}}]\label{thm_ext_as_exp}
		For any $r \in \nat$, $y_0 \in Y$, there are polynomials $J_r^{X|Y, E}(Z, Z'_Y) \in \enmr{F_{y_0}}$ in $Z \in \real^{2n}$, $Z'_Y \in \real^{2m}$, with the same properties as in Theorem \ref{thm_berg_off_diag}, such that for $F_r^{X|Y, E} := J_r^{X|Y, E} \cdot \mathscr{E}_{n, m}$, the following holds.
		\par 
		There are $\epsilon, c > 0$, $p_1 \in \nat^*$, such that for any $k, l, l' \in \nat$, there is $C  > 0$, such that for any $y_0 \in Y$, $p \geq p_1$, $Z = (Z_Y, Z_N)$, $Z_Y, Z'_Y \in \real^{2m}$, $Z_N \in \real^{2(n - m)}$, $|Z|, |Z'_Y| \leq \epsilon$, $\alpha \in \nat^{2n}$, $\alpha' \in \nat^{2m}$, $|\alpha| + |\alpha'| \leq l$, for  $Q^2_{k, l, l'} := 6(16(n+2)(k+1) + l') + 2l$, the following bound holds
		\begin{multline}\label{eq_ext_as_exp}
			\bigg| 
				\frac{\partial^{|\alpha|+|\alpha'|}}{\partial Z^{\alpha} \partial Z'_Y{}^{\alpha'}}
				\bigg(
					\frac{1}{p^m} \ext_p^{X|Y} \big(\psi_{y_0}^{X|Y}(Z), \psi_{y_0}^{X|Y}(Z'_Y) \big)
					\\
					-
					\sum_{r = 0}^{k}
					p^{-\frac{r}{2}}						
					F_r^{X|Y, E}(\sqrt{p} Z, \sqrt{p} Z'_Y) 
					\kappa_{\psi}^{X|Y}(Z)^{-\frac{1}{2}}
					\kappa_{\phi}^{Y}(Z'_Y)^{-\frac{1}{2}}
				\bigg)
			\bigg|_{\ccal^{l'}}
			\\
			\leq
			C p^{- \frac{k + 1 - l}{2}}
			\Big(1 + \sqrt{p}|Z| + \sqrt{p} |Z'_Y| \Big)^{Q^2_{k, l, l'}}
			\exp\Big(- c \sqrt{p} \big( |Z_Y - Z'_Y| + |Z_N| \big) \Big),
		\end{multline}
		where the $\ccal^{l'}$-norm is taken with respect to $y_0$.
		Also, the following identity holds
		\begin{equation}\label{eq_je0_exp}
			J_0^{X|Y, E}(Z, Z'_Y) = {\rm{Id}}_{F_{y_0}} \cdot \kappa_N^{X|Y}(y_0)^{\frac{1}{2}}.
		\end{equation}
 	\end{thm}
	\begin{thm}[{ \cite[Theorem 1.9]{FinOTAs}}]\label{thm_berg_perp_off_diag}
		For any $r \in \nat$, $y_0 \in Y$, there are polynomials $J_r^{X|Y, \perp}(Z, Z') \in \enmr{F_{y_0}}$, $Z, Z' \in \real^{2n}$, with the same properties as in Theorem \ref{thm_berg_off_diag}, such that for $F_r^{X|Y, \perp} := J_r^{X|Y, \perp} \cdot \mathscr{P}_{n, m}^{\perp}$, the following holds.
		\par 
		There are $\epsilon, c > 0$, $p_1 \in \nat^*$, such that for any $k, l, l' \in \nat$, there is $C  > 0$, such that for any $y_0 \in Y$, $p \geq p_1$, $Z = (Z_Y, Z_N)$, $Z' = (Z'_Y, Z'_N)$, $Z_Y, Z'_Y \in \real^{2m}$, $Z_N, Z'_N \in \real^{2(n-m)}$, $|Z|, |Z'| \leq \epsilon$, $\alpha, \alpha' \in \nat^{2n}$, $|\alpha|+|\alpha'| \leq l$,  for  $Q^3_{k, l, l'} := 3(8(n+2)(k+1) + l') + l$, we have
		\begin{multline}\label{eq_berg_perp_off_diag}
			\bigg| 
				\frac{\partial^{|\alpha|+|\alpha'|}}{\partial Z^{\alpha} \partial Z'{}^{\alpha'}}
				\bigg(
					\frac{1}{p^n} B_p^{X|Y \perp} \big(\psi_{y_0}^{X|Y}(Z), \psi_{y_0}^{X|Y}(Z') \big)
					\\
					-
					\sum_{r = 0}^{k}
					p^{-\frac{r}{2}}						
					F_r^{X|Y, \perp}(\sqrt{p} Z, \sqrt{p} Z') 
					\kappa_{\psi}^{X|Y}(Z)^{-\frac{1}{2}}
					\kappa_{\psi}^{X|Y}(Z')^{-\frac{1}{2}}
				\bigg)
			\bigg|_{\ccal^{l'}}
			\\
			\leq
			C p^{- \frac{k + 1 - l}{2}}
			\Big(1 + \sqrt{p}|Z| + \sqrt{p} |Z'| \Big)^{Q^3_{k, l, l'}}
			\exp\Big(- c \sqrt{p} \big( |Z_Y - Z'_Y| + |Z_N| + |Z'_N| \big) \Big),
		\end{multline}
		where the $\ccal^{l'}$-norm is taken with respect to $y_0$.
		Also, we have
		\begin{equation}\label{eq_jopep_0}
			J_0^{X|Y, \perp}(Z, Z') = {\rm{Id}}_{F_{y_0}}.
		\end{equation}
	\end{thm}
		
	\subsection{Basic Toeplitz type operators and their asymptotics}\label{sect_asympt_study_toepl}
	The main goal of this section is to study asymptotic expansions of basic Toeplitz type operators.
	In particular, we prove Theorem \ref{thm_ttype_as} and give precise formulas for the constants $C_1, C_2$ from (\ref{eq_mh_norm_operators}).
	The following two lemmas will be crucial in what follows.
	\begin{lem}\label{lem_exp_dec_toepl}
		Let $(X, Y, g^{TX})$ be of bounded geometry.
		There is $p_1 \in \nat$, such that for any $f \in \ccal^{\infty}_{b}(Y, \enmr{\iota^* F})$, $g \in \oplus_{k = 0}^{\infty} \ccal^{\infty}_{b}(Y, {\rm{Sym}}^k (N^{X|Y})^* \otimes \enmr{\iota^* F})$, $l \in \nat$, there is $C > 0$, such that for any $p \geq p_1$, the Schwartz kernels $T_{f, p}^{Y}(y_1, y_2)$, $T_{\llangle g \rrangle, p}^{X|Y} (x, y_1)$, $T_{\llangle g \rrangle, p}^{Y|X}(y_1, x)$; $x \in X$, $y_1, y_2 \in Y$, of $T_{f, p}^{Y}$,  $T_{\llangle g \rrangle, p}^{X|Y}$, $T_{\llangle g \rrangle, p}^{Y|X}$, evaluated with respect to $dv_Y$, $dv_Y$ and $dv_X$ respectively, satisfy
		\begin{equation}\label{eq_exp_est_ass_1_expl}
		\begin{aligned}
			&
			\Big|  
				T_{f, p}^{Y} (y_1, y_2) 
			\Big|_{\ccal^l} 
			\leq 
			C p^{m + \frac{l}{2}} 
			\cdot 
			\exp \big(- c \sqrt{p} \cdot \dist_Y(y_1, y_2) \big),
			\\
			&
			\Big|  
			T_{\llangle g \rrangle, p}^{X|Y} (x, y_1) 
			\Big|_{\ccal^l} 
			\leq 
			C p^{m + \frac{l}{2}} 
			\cdot 
			\exp \big(- c \sqrt{p} \cdot \dist_X(x, y_1) \big),
			\\
			&
			\Big|  
			T_{\llangle g\rrangle, p}^{Y|X}(y_1, x)
			\Big|_{\ccal^l} 
			\leq 
			C p^{n + \frac{l}{2}} 
			\cdot 
			\exp \big(- c \sqrt{p} \cdot \dist_X(x, y_1) \big).
		\end{aligned}
		\end{equation}
	\end{lem}
	\begin{proof}
		The first and third parts follow trivially from Theorem \ref{thm_bk_off_diag} and Corollary \ref{cor_comp_exp_bound}.
		The second part is a consequence of Theorems \ref{thm_bk_off_diag}, \ref{thm_ext_exp_dc} and Corollary \ref{cor_comp_exp_bound}.
	\end{proof}
	We fix $y_0 \in Y$, a unitary frame $(f_1, \ldots, f_r)$ of $(F_{y_0}, h^{F}_{y_0})$ and use the notational conventions introduced before Theorem \ref{thm_berg_off_diag}.
	\begin{sloppypar}
	\begin{lem}\label{lem_toepl_tay_type}
		For any $f \in \ccal^{\infty}_{b}(Y, \enmr{\iota^* F})$, (resp. $g \in \ccal^{\infty}_{b}(Y, {\rm{Sym}}^k (N^{X|Y})^* \otimes \enmr{\iota^* F})$, $k \in \nat$), $y_0 \in Y$, $r \in \nat$, there are $J_{r, f}^{Y}(Z_Y, Z'_Y) \in \enmr{F_{y_0}}$ (resp. $J_{r, g}^{E}(Z, Z'_Y) \in \enmr{F_{y_0}}$, $J_{r, g}^{R}(Z_Y, Z') \in \enmr{F_{y_0}}$), polynomials in $Z_Y, Z'_Y \in \real^{2m}$, $Z, Z' \in \real^{2n}$ of the same parity as $r$ (resp. $r + k$), such that the coefficients of $J_{r, f}^{Y}$ (resp. $J_{r, g}^{E}$, $J_{r, g}^{R}$) lie in $\ccal^{\infty}_{b}(Y, \enmr{\iota^* F})$, and for $F_{r, f}^{Y} := J_{r, f}^Y \cdot \mathscr{P}_{m}$ (resp. $F_{r, g}^{E} := J_{r, g}^E \cdot \mathscr{E}_{n, m}$, $F_{r, g}^{R} := J_{r, g}^R \cdot \mathscr{Res}_{n, m}$), the following holds.
		\par 
		There are $\epsilon, c > 0$, $p_1 \in \nat^*$, such that for any $k, l, l' \in \nat$, there are $C, Q > 0$, such that for any $y_0 \in Y$, $p \geq p_1$, $Z, Z' \in \real^{2n}$, $Z = (Z_Y, Z_N)$, $Z' = (Z'_Y, Z'_N)$, $Z_Y, Z'_Y \in \real^{2m}$, $|Z|, |Z'| \leq \epsilon$, $\alpha, \alpha' \in \nat^{2n}$, $\alpha_0, \alpha'_0 \in \nat^{2m}$, $|\alpha|+|\alpha'_0|, |\alpha_0|+|\alpha'|, |\alpha_0|+|\alpha'_0| \leq l$, the following bounds hold
			\begin{multline}\label{eq_tpy_defn_exp_tay}
				\bigg| 
					\frac{\partial^{|\alpha_0|+|\alpha'_0|}}{\partial Z_Y^{\alpha_0} \partial Z'_Y{}^{\alpha'_0}}
					\bigg(
						\frac{1}{p^m} T_{p, f}^Y \big(\phi_{y_0}^{Y}(Z_Y), \phi_{y_0}^{Y}(Z'_Y) \big)
						\\
						-
						\sum_{r = 0}^{k}
						p^{-\frac{r}{2}}						
						F_{r, f}^{Y}(\sqrt{p} Z_Y, \sqrt{p} Z'_Y) 
						\kappa_{\phi}^{Y}(Z_Y)^{-\frac{1}{2}}
						\kappa_{\phi}^{Y}(Z'_Y)^{-\frac{1}{2}}
					\bigg)
				\bigg|_{\ccal^{l'}}
				\\
				\leq
				C p^{-\frac{k + 1 - l}{2}}
				\Big(1 + \sqrt{p}|Z_Y| + \sqrt{p} |Z'_Y| \Big)^{Q} \exp(- c \sqrt{p} |Z_Y - Z'_Y|),
				\end{multline}
				\vspace*{-1cm}
				\begin{multline}\label{eq_tpy_defn_exp_tay2}
				\bigg| 
					\frac{\partial^{|\alpha|+|\alpha'|}}{\partial Z^{\alpha} \partial Z'_Y{}^{\alpha'}}
					\bigg(
						\frac{1}{p^m} T_{\llangle g \rrangle, p}^{X|Y} \big(\psi_{y_0}^{X|Y}(Z), \phi_{y_0}^{Y}(Z'_Y) \big)
						\\
						-
						\sum_{r = 0}^{k}
						p^{-\frac{r}{2}}						
						F_{r, g}^{E}(\sqrt{p} Z, \sqrt{p} Z'_Y) 
						\kappa_{\psi}^{X|Y}(Z)^{-\frac{1}{2}}
						\kappa_{\phi}^{Y}(Z'_Y)^{-\frac{1}{2}}
					\bigg)
				\bigg|_{\ccal^{l'}}
				\\
				\leq
				C p^{-\frac{k + 1 - l}{2}}
				\Big(1 + \sqrt{p}|Z| + \sqrt{p} |Z'_Y| \Big)^{Q} \exp(- c \sqrt{p}( |Z_N| + |Z_Y - Z'_Y| ) ),
				\end{multline}
				\vspace*{-1cm}
				\begin{multline}\label{eq_tpy_defn_exp_tay3}
				\bigg| 
					\frac{\partial^{|\alpha|+|\alpha'|}}{\partial Z_Y^{\alpha} \partial Z'{}^{\alpha'}}
					\bigg(
						\frac{1}{p^n} T_{\llangle g \rrangle, p}^{Y|X} \big(\phi_{y_0}^{Y}(Z_Y), \psi_{y_0}^{X|Y}(Z') \big)
						\\
						-
						\sum_{r = 0}^{k}
						p^{-\frac{r}{2}}						
						F_{r, g}^{R}(\sqrt{p} Z_Y, \sqrt{p} Z') 
						\kappa_{\phi}^{Y}(Z_Y)^{-\frac{1}{2}}
						\kappa_{\psi}^{X|Y}(Z')^{-\frac{1}{2}}
					\bigg)
				\bigg|_{\ccal^{l'}}
				\\
				\leq
				C p^{-\frac{k + 1 - l}{2}}
				\Big(1 + \sqrt{p}|Z_Y| + \sqrt{p} |Z'| \Big)^{Q} \exp(- c \sqrt{p}( |Z'_N| + |Z_Y - Z'_Y| ) ),
			\end{multline}
			where the $\ccal^{l'}$-norm is taken with respect to $y_0$.
		\par 
		Moreover, we have $J_{0, f}(Z_Y, Z'_Y) = f(y_0)$ and in the notations of (\ref{eq_brack_defn}), for $g^h \in \ccal^{\infty}_{b}(Y, {\rm{Sym}}^k (N^{X|Y})^{(1, 0)*} \otimes \enmr{\iota^* F})$ (resp. $g^a \in \ccal^{\infty}_{b}(Y, {\rm{Sym}}^k (N^{X|Y})^{(0, 1)*} \otimes \enmr{\iota^* F})$), $k \in \nat$, the polynomial $J_{0, g^h}^E(Z, Z'_Y)$ (resp. $J_{0, g^a}^R(Z_Y, Z')$) depends only on $z_N$ (resp. $\overline{z}_N$), it has degree $k$, and as a section of ${\rm{Sym}}^{k} (N^{X|Y})^{(1, 0)*} \otimes \enmr{\iota^* F}$ (resp. ${\rm{Sym}}^{k} (N^{X|Y})^{(0, 1)*} \otimes \enmr{\iota^* F}$) over $Y$, it coincides with $g^h \cdot \kappa_N^{X|Y}(y_0)^{\frac{1}{2}}$ (resp. $g^a \cdot \kappa_N^{X|Y}(y_0)^{-\frac{1}{2}}$) for $k \geq 1$ and equal to $0$ for $k = 0$.
	\end{lem}
	\end{sloppypar}
	\begin{proof}
		We establish each of the three statements one by one.
		Remark first that Ma-Marinescu in \cite[Lemma 4.6]{MaMarToepl} established the first part of this result for compact manifolds.
		We will now describe why essentially the same proof holds for the first statement of Lemma \ref{lem_toepl_tay_type} for functions from $\ccal^{\infty}_{b}(Y, \enmr{\iota^* F})$ on manifolds of bounded geometry $(Y, g^{TY})$. 
		Trivially, we have
		\begin{equation}\label{eq_b_f}
			T_{f, p}^{Y}(y_1, y_2) = \int_{Y} B_p^Y(y_1, y_3) \cdot f(y_3) \cdot B_p^Y(y_3, y_2) dv_Y (y_3).
		\end{equation}
		Now, let $\epsilon > 0$ be as in Theorem \ref{thm_berg_off_diag}.
		We put $\epsilon_0 := \frac{\epsilon}{2}$.
		Let $y_0 \in Y$ and $y_1, y_2 \in B_{y_0}^{Y}(\epsilon_0)$.
		We decompose the integral in (\ref{eq_b_f}) into two parts: the first one over $B_{y_0}^{Y}(\epsilon)$, and the second one is over its complement, which we denote by $Q$.
		Clearly, for $y_3 \in Q$, we get 
		\begin{equation}\label{eq_triangle}
			\dist(y_1, y_3) + \dist(y_3, y_2) \geq \epsilon.
		\end{equation}
		Hence, from Theorem \ref{thm_bk_off_diag}, Proposition \ref{prop_exp_bound_int} and (\ref{eq_vol_comp_unif}), we see that the contribution from the integration over $Q$ is smaller than $\exp(- c \sqrt{p}(1 + \dist(y_1, y_2)))$ for some constant $c > 0$.
		Consequently, only the integration over $B_{y_0}^{Y}(\epsilon)$ is non-negligible.
		To evaluate it, we apply Theorem \ref{thm_bk_off_diag}.
		We calculate the integral over the pull-back with respect to the exponential map of our differential forms.
		We use the notations introduced before Theorem \ref{thm_berg_off_diag}.
		After the change of variables $Z \mapsto \sqrt{p} Z$, an estimate, similar to the one which bounded the integral over $Q$, and the first part of Lemma \ref{lem_comp_poly}, applied for $n := m$, we see that (\ref{eq_tpy_defn_exp_tay}) holds for
		\begin{equation}\label{eq_jrf_expr}
			J_{r, f}^{Y}
			:=
			\sum_{a + b + |\alpha| = r}
			\sum_{\alpha}
			\mathcal{K}_{m, m} \Big[ J_a^{Y|Y}, \frac{\partial^{\alpha} f(\phi_{y_0}^Y(Z_Y))}{\partial Z_Y^{\alpha}}(0) \cdot \frac{Z_Y^{\alpha}}{\alpha!} \cdot J_b^{Y|Y} \Big],
		\end{equation}
		where $\alpha$ runs through the set of multiindices $\nat^{2m}$.
		From (\ref{eq_jo_expl_form}) and (\ref{eq_jrf_expr}), we deduce the statement about $J_{0, f}^{Y}$.
		From (\ref{eq_jrf_expr}), our bounded geometry assumption, the fact that $f \in \ccal^{\infty}_b(Y, \enmr{\iota^* F})$ and the fact that the coefficients of $J_r^{Y|Y}$ are polynomials in $\omega$, $R^{TY}$, $R^F$, $(dv_Y / dv_{g^{TY}})^{\pm \frac{1}{2m}}$, and their derivatives of order $\leq 2r$, we deduce that the coefficients of $J_{r, f}^{Y}$ lie in $\ccal^{\infty}_{b}(Y, \enmr{\iota^* F})$.
		The statement about their parity follows from Lemma \ref{lem_comp_poly} and (\ref{eq_jrf_expr}).
		\par 
		Now, let us establish the second part.
		The proof is absolutely analogous to the proof of the first part.
		One only has to in addition to Theorems \ref{thm_bk_off_diag}, \ref{thm_berg_off_diag} use Theorems \ref{thm_ext_exp_dc}, \ref{thm_ext_as_exp}, \ref{thm_berg_perp_off_diag}.
		Again we use the notations introduced before Theorem \ref{thm_berg_off_diag}.
		We view the section $\{g\}$, constructed before (\ref{eq_brack_defn}), as a function with values in $\enmr{F_{y_0}}$.
		Clearly, for the function $\{g\}$, constructed before (\ref{eq_brack_defn}), $\{g\}(\psi_{y_0}^{X|Y}(Z_Y, Z_N))$ is polynomial in vertical directions; in other words
		\begin{equation}\label{eq_fig_g_dec}
			\{g\} \big(\psi_{y_0}^{X|Y}(Z_Y, Z_N) \big) = \sum_{\beta} g_{\beta}(\phi_{y_0}^{Y}(Z_Y)) \frac{Z_N^{\beta}}{\beta!},
		\end{equation}
		where $\beta \in \nat^{2(n - m)}$ runs through all the multiindices and the sum is finite. 
		From the first and the second equations of (\ref{eq_comp_ext_pmn1}) and the reasoning, similar to the one before (\ref{eq_jrf_expr}), we obtain that the polynomials
		\begin{multline}\label{eq_jrg_e_form}
			J_{r, g}^{E}(Z, Z'_Y)
			:=
			\sum_{\alpha}
			\sum_{\beta}
			\sum_{a + b + |\alpha| = r}
			\Big(
				\mathcal{K}_{n, n} \Big[ J_a^{X|Y}, \frac{\partial^{\alpha} g_{\beta}(\phi_{y_0}^{Y}(Z_Y))}{\partial Z_Y^{\alpha}}(0) \cdot \frac{Z_N^{\beta}}{\beta!} \cdot \frac{Z_Y^{\alpha}}{\alpha!} \cdot J_b^{X|Y, E} \Big]
				\\
				-
				\mathcal{K}_{n, m} \Big[ J_a^{X|Y, \perp}, \frac{\partial^{\alpha} g_{\beta}(\phi_{y_0}^{Y}(Z_Y))}{\partial Z_Y^{\alpha}}(0) \cdot \frac{Z_N^{\beta}}{\beta!} \cdot \frac{Z_Y^{\alpha}}{\alpha!} \cdot J_b^{X|Y, E}  \Big]
			\Big)(Z, Z'_Y),
		\end{multline}
		where $\alpha$ runs through the multiindices $\nat^{2m}$, satisfy the second equation from (\ref{eq_tpy_defn_exp_tay}).
		The statement about the parity of $J_{r, g}^{E}$ follows from (\ref{eq_jrg_e_form}) similarly as it was done in (\ref{eq_jrf_expr}).
		The calculation of $J_{0, g}^{E}$ follows from (\ref{eq_kmn_poly}), (\ref{eq_jo_expl_form}), (\ref{eq_je0_exp}), (\ref{eq_jopep_0}) and (\ref{eq_jrg_e_form}).
		\par 
		Now, let us establish the third part.
		The proof is absolutely analogous to the proof of the first part.
		One only has to use Theorems \ref{thm_ext_exp_dc}, \ref{thm_berg_perp_off_diag} in addition to Theorems \ref{thm_bk_off_diag}, \ref{thm_berg_off_diag}.
		More precisely, in the notations (\ref{eq_fig_g_dec}),  from (\ref{eq_lem_comp_poly_1}) and (\ref{eq_lem_comp_poly_2}), we obtain that the third equation from (\ref{eq_tpy_defn_exp_tay}) holds for
		\begin{multline}\label{eq_jrg_e_form_rest}
			J_{r, g}^{R}(Z_Y, Z')
			:=
			\sum_{\alpha}
			\sum_{\beta}
			\sum_{a + b + |\alpha| = r}
			\Big(
				\mathcal{K}_{n, n} \Big[ J_a^{X|Y}, \frac{\partial^{\alpha} g_{\beta}(\phi_{y_0}^{Y}(Z_Y))}{\partial Z_Y^{\alpha}}(0) \cdot \frac{Z_N^{\beta}}{\beta!} \cdot \frac{Z_Y^{\alpha}}{\alpha!} \cdot J_b^{X|Y} \Big]
				\\
				-
				\mathcal{K}'_{n, m} \Big[ J_a^{X|Y}, \frac{\partial^{\alpha} g_{\beta}(\phi_{y_0}^{Y}(Z_Y))}{\partial Z_Y^{\alpha}}(0) \cdot \frac{Z_N^{\beta}}{\beta!} \cdot \frac{Z_Y^{\alpha}}{\alpha!} \cdot J_b^{X|Y, \perp}  \Big]
			\Big) (Z_Y, Z').
		\end{multline}
		The statement about the parity of $J_{r, g}^{R}$ follows from (\ref{eq_jrg_e_form_rest}) similarly as it was done in (\ref{eq_jrf_expr}).
		The calculation of $J_{0, g}^{R}$ follows from (\ref{eq_kmn_poly}), (\ref{eq_jo_expl_form}), (\ref{eq_je0_exp}), (\ref{eq_jopep_0}) and (\ref{eq_jrg_e_form_rest}).
	\end{proof}
	From Lemma \ref{lem_toepl_tay_type}, we obtain directly the following corollary.
	\begin{sloppypar}
	\begin{cor}\label{cor_brack_well_def1}
		Assume that for $f_1, f_2 \in \ccal^{\infty}_{b}(Y, \enmr{\iota^* F})$, $g_1^h, g_2^h \in \oplus_{k = 1}^{\infty} \ccal^{\infty}_{b}(Y, {\rm{Sym}}^{k} (N^{X|Y})^{(1, 0)*} \otimes \enmr{\iota^* F})$, $g_1^a, g_2^a \in \oplus_{k = 1}^{\infty} \ccal^{\infty}_{b}(Y, {\rm{Sym}}^{k} (N^{X|Y})^{(0, 1)*} \otimes \enmr{\iota^* F})$, there are $C > 0$, $p_1 \in \nat^*$, such that for any $p \geq p_1$, in the notations from Lemma \ref{lem_exp_dec_toepl},
		\begin{equation}
		\begin{aligned}
			&
			\Big|  
				T_{f_1, p}^Y(y_1, y_2)  
				- 
				T_{f_2, p}^Y(y_1, y_2)  
			\Big|
			\leq 
			C p^{m - \frac{1}{2}},
			\\
			&
			\Big|  
				T_{\llangle g_1^h \rrangle, p}^{X|Y}(x, y_1)  
				- 
				T_{\llangle g_2^h \rrangle, p}^{X|Y}(x, y_1)  
			\Big|
			\leq 
			C p^{m - \frac{1}{2}},
			\\
			&
			\Big|  
				T_{\llangle g_1^a \rrangle, p}^{Y|X}(y_1, x)  
				- 
				T_{\llangle g_2^a \rrangle, p}^{Y|X}(y_1, x)  
			\Big|
			\leq 
			C p^{n - \frac{1}{2}},
		\end{aligned}
		\end{equation}
		for any $x \in X$, $y_1, y_2 \in Y$.
		Then we have $f_1 = f_2$, $g_1^h = g_2^h$, $g_1^a = g_2^a$.
		In particular, the notation $[\cdot]_i$ from Definition \ref{defn_ttype} is well-defined.
	\end{cor}
	\end{sloppypar}
	\begin{proof}
		It follows directly from (\ref{eq_tpy_defn_exp_tay}), (\ref{eq_tpy_defn_exp_tay2}), (\ref{eq_tpy_defn_exp_tay3}) and the precise descriptions of $J_{0, f_i}^{Y}$, $J_{0, g_i^h}^{E}$, $J_{0, g_i^a}^{R}$ for $i = 1, 2$, given in the end of Lemma \ref{lem_toepl_tay_type}.
	\end{proof}
	\par 
	Let us now describe the proof of Theorem \ref{thm_ttype_as}.
	The idea of the proof is simple: we need to verify that the first terms of the asymptotic expansions of the Toeplitz operators coincide with the first terms of the asymptotic expansions of the corresponding operators (\ref{eq_ext0_op}), (\ref{eq_ext1_op}) and that the Schwartz kernels of all the operators have exponential decay.
	\par 
	\begin{sloppypar}
	We conserve the notation from Section \ref{sect_intro}.
	We fix arbitrary  $g \in \oplus_{k = 0}^{\infty} \ccal^{\infty}_{b}(Y, {\rm{Sym}}^k (N^{X|Y})^{*} \otimes \enmr{\iota^* F})$,  $g^h \in \oplus_{k = 1}^{\infty} \ccal^{\infty}_{b}(Y, {\rm{Sym}}^k (N^{X|Y})^{(1, 0)*} \otimes \enmr{\iota^* F})$, $g^a \in \oplus_{k = 1}^{\infty} \ccal^{\infty}_{b}(Y, {\rm{Sym}}^k (N^{X|Y})^{(0, 1)*} \otimes \enmr{\iota^* F})$.
	For $x \in X$, $y \in Y$, let $M_{g, p}^{X|Y}(x, y)$ (resp. $M_{g, p}^{Y|X, \dagger}(y, x)$) be the Schwartz kernel of $M_{g, p}^{X|Y}$ (resp. $M_{g, p}^{Y|X, \dagger}$), evaluated with respect to $dv_Y$ (resp. $dv_X$).
	\end{sloppypar}
	\begin{lem}\label{lem_mgo_exp_dc}
		There are $c, C > 0$, $p_1 \in \nat^*$, such that for any $p \geq p_1$, $x \in X$, $y \in Y$, the following estimates hold
		\begin{equation}\label{eq_mgo_exp_dc}
		\begin{aligned}
			&
			\big|  M_{g, p}^{X|Y}(x, y) \big| \leq C p^m \exp \big(- c \sqrt{p} \cdot \dist(x, y) \big),
			\\
			&
			\big|  M_{g, p}^{Y|X, \dagger}(y, x) \big| \leq C p^n \exp \big(- c \sqrt{p} \cdot \dist(x, y) \big),
		\end{aligned}
		\end{equation}
	\end{lem}
	\begin{proof}
		It follows trivially from Theorem \ref{thm_bk_off_diag}, (\ref{eq_ext0_op}) and the fact that the function $u^k \exp(- u)$ is bounded for $u \in \real_+$ for any $k \in \nat$.
	\end{proof}
	\par 
	\begin{proof}[Proof of Theorem \ref{thm_ttype_as}.]
		We consider the Schwartz kernel $M_{g, p}^{X|Y}(x, y)$ (resp. $M_{g, 0, p}^{Y|X, \dagger}(y, x)$) of $M_{g, p}^{X|Y}$, (resp. $M_{g, p}^{Y|X, \dagger}$, viewed as an operator acting on the sections with support in a $r_{\perp}$-tubular neighborhood of $Y$)  evaluated with respect to the volume form $dv_Y$ (resp. $dv_Y \wedge dv_{N^{X|Y}}$).
		We use the notational convention introduced before Theorem \ref{thm_berg_off_diag}.
		From Theorem \ref{thm_berg_off_diag} and (\ref{eq_ext0_op}), (\ref{eq_berg_k_expl}), (\ref{eq_pperp_defn_fun}), we conclude that there are $\epsilon, c, C, Q > 0$, $p_1 \in \nat^*$, such that for any $y_0 \in Y$, $p \geq p_1$, $Z, Z' \in \real^{2n}$, $Z = (Z_Y, Z_N)$, $Z' = (Z'_Y, Z'_N)$, $|Z|, |Z'| \leq \epsilon$, $Z_Y, Z'_Y \in \real^{2m}$, the following bounds hold
		\begin{multline}\label{eq_ext_as_exp_mg_operator}
			\bigg| 
					\frac{1}{p^m} M_{g, p}^{X|Y}\big(\psi_{y_0}^{X|Y}(Z), \phi_{y_0}^{Y}(Z'_Y) \big)
					-
					\{g\}(y_0, \sqrt{p} Z_N)
					\cdot
					\mathscr{E}_{n, m}(\sqrt{p} Z, \sqrt{p} Z'_Y) 
					\kappa_{\phi}^{Y}(Z_Y)^{-\frac{1}{2}}
					\kappa_{\phi}^{Y}(Z'_Y)^{-\frac{1}{2}}
			\bigg|
			\\
			\leq
			C p^{- \frac{1}{2}}
			\Big(1 + \sqrt{p}|Z| + \sqrt{p} |Z'_Y| \Big)^{Q}
			\exp\Big(- c \sqrt{p} \big( |Z_Y - Z'_Y| + |Z_N| \big) \Big),
		\end{multline}
		\vspace*{-1cm}
		\begin{multline}\label{eq_ext_as_exp_mg_operator2a}
			\bigg| 
					\frac{1}{p^n} M_{g, 0, p}^{Y|X, \dagger}\big(\phi_{y_0}^{Y}(Z_Y), \psi_{y_0}^{X|Y}(Z') \big)
					-
					\{g\}(y_0, \sqrt{p} Z'_N)
					\cdot
					\mathscr{R}_{n, m}(\sqrt{p} Z_Y, \sqrt{p} Z') 
					\kappa_{\phi}^{Y}(Z_Y)^{-\frac{1}{2}}
					\kappa_{\phi}^{Y}(Z'_Y)^{-\frac{1}{2}}
			\bigg|
			\\
			\leq
			C p^{- \frac{1}{2}}
			\Big(1 + \sqrt{p}|Z_Y| + \sqrt{p} |Z'| \Big)^{Q}
			\exp\Big(- c \sqrt{p} \big( |Z_Y - Z'_Y| + |Z'_N| \big) \Big).
		\end{multline}
		We denote now by $M_{g, p}^{Y|X, \dagger}(y, x)$ the Schwartz kernel of $M_{g, p}^{Y|X, \dagger}$, evaluated with respect to $dv_X$.
		From Remark \ref{rem_bndg_tripl}b), (\ref{eq_kappa_relation}), (\ref{eq_ext_as_exp_mg_operator}) and (\ref{eq_ext_as_exp_mg_operator2a}), for $Q_0 := \max \{ Q, \deg g, \deg g^h, \deg g^a\}$, we then obtain that in the same notations (but, probably, for a different choice of $C$), we have
		\begin{multline}\label{eq_ext_as_exp_mg_operator2}
			\bigg| 
					\frac{1}{p^m} M_{g, p}^{X|Y}\big(\psi_{y_0}^{X|Y}(Z), \phi_{y_0}^{Y}(Z'_Y) \big)
					\\
					-
					\kappa_N^{X|Y}(y_0)^{\frac{1}{2}}
					\cdot
					\{g\}(y_0, \sqrt{p} Z_N)
					\cdot
					\mathscr{E}_{n, m}(\sqrt{p} Z, \sqrt{p} Z'_Y) 
					\kappa_{\psi}^{X|Y}(Z)^{-\frac{1}{2}}
					\kappa_{\phi}^{Y}(Z'_Y)^{-\frac{1}{2}}
			\bigg|
			\\
			\leq
			C p^{- \frac{1}{2}}
			\Big(1 + \sqrt{p}|Z| + \sqrt{p} |Z'_Y| \Big)^{Q_0}
			\exp\Big(- c \sqrt{p} \big( |Z_Y - Z'_Y| + |Z_N| \big) \Big),
		\end{multline}
		\vspace*{-1cm}
		\begin{multline}\label{eq_ext_as_exp_mg_operator2abb}
			\bigg| 
					\frac{1}{p^n} M_{g, p}^{Y|X, \dagger}\big(\phi_{y_0}^{Y}(Z_Y), \psi_{y_0}^{X|Y}(Z') \big)
					\\
					-
					\kappa_N^{X|Y}(y_0)^{-\frac{1}{2}}
					\cdot
					\{g\}(y_0, \sqrt{p} Z'_N)
					\cdot
					\mathscr{R}_{n, m}(\sqrt{p} Z_Y, \sqrt{p} Z') 
					\kappa_{\phi}^{Y}(Z_Y)^{-\frac{1}{2}}
					\kappa_{\psi}^{X|Y}(Z')^{-\frac{1}{2}}
			\bigg|
			\\
			\leq
			C p^{- \frac{1}{2}}
			\Big(1 + \sqrt{p}|Z_Y| + \sqrt{p} |Z'| \Big)^{Q_0}
			\exp\Big(- c \sqrt{p} \big( |Z_Y - Z'_Y| + |Z'_N| \big) \Big).
		\end{multline}
		By comparing (\ref{eq_ext_as_exp_mg_operator2}) with the expansions from Lemma \ref{lem_toepl_tay_type}, there is $Q_1 \geq 0$, such that
		\begin{multline}\label{eq_ext_as_exp_mg_operator3}
			\bigg| 
					M_{g^h, p}^{X|Y}\big(\psi_{y_0}^{X|Y}(Z), \phi_{y_0}^{Y}(Z'_Y) \big)
					-
					T_{\llangle g^h \rrangle, p}^{X|Y}\big(\psi_{y_0}^{X|Y}(Z), \phi_{y_0}^{Y}(Z'_Y) \big)
			\bigg|
			\\
			\leq
			C p^{m - \frac{1}{2}}
			\Big(1 + \sqrt{p}|Z| + \sqrt{p} |Z'_Y| \Big)^{Q_1}
			\exp\Big(- c \sqrt{p} \big( |Z_Y - Z'_Y| + |Z_N| \big) \Big),
		\end{multline}
		\vspace*{-1cm}
		\begin{multline}\label{eq_ext_as_exp_mg_operator3a}
			\bigg| 
					M_{g^a, p}^{Y|X, \dagger}\big(\phi_{y_0}^{Y}(Z_Y), \psi_{y_0}^{X|Y}(Z') \big)
					-
					T_{\llangle g^a \rrangle, p}^{Y|X}(\phi_{y_0}^{Y}(Z_Y), \psi_{y_0}^{X|Y}(Z') \big)
			\bigg|
			\\
			\leq
			C p^{n - \frac{1}{2}}
			\Big(1 + \sqrt{p}|Z_Y| + \sqrt{p} |Z'| \Big)^{Q_1}
			\exp\Big(- c \sqrt{p} \big( |Z_Y - Z'_Y| + |Z'_N| \big) \Big).
		\end{multline}
		From Proposition \ref{prop_norm_bnd_distk_expbnd}, Lemma \ref{lem_mgo_exp_dc}, (\ref{eq_ext_as_exp_mg_operator3}) and (\ref{eq_ext_as_exp_mg_operator3a}), we finally deduce Theorem \ref{thm_ttype_as}.
	\end{proof}
	\par 
	Let us now briefly describe how to calculate the asymptotics of the norms of operators $M_{g, p}^{X|Y}$ and $M_{g, p}^{Y|X, \dagger}$.
	For this, the following lemma will be of crucial importance.
	\begin{lem}\label{lem_norm_toepl}
		Let $f \in \ccal^{\infty}_{b}(X, \enmr{F})$ be non-zero.
		Then the following asymptotics holds
		\begin{equation}\label{eq_norm_toepl_calc}
			\| T_{f, p}^{X} \| \sim \sup_{x \in X} \| f(x) \|.
		\end{equation}
		Moreover, we have 
		\begin{equation}\label{eq_top_final_112}
			\big\| T_{f, p}^{X} \big\| 
			\leq 
			\sup_{x \in X} \| f(x) \|,
		\end{equation}
		and if the above supremum is achieved in $X$, then there is $C > 0$, such that for $p$ big enough
		\begin{equation}\label{eq_top_final_2}
			\sup_{x \in X} \| f(x) \| - \frac{C}{\sqrt{p}}
			\leq
			\big\| T_{f, p}^{X} \big\|.
		\end{equation}
	\end{lem}
	\begin{proof}
		\par Recall that for compact manifolds $X$, Bordemann-Meinrenken-Schlichenmaier \cite[Theorem 4.1]{BordMeinSchl} (for $(F, h^F)$ trivial) and Ma-Marinescu \cite[Theorem 3.19, (3.91)]{MaMarToepl} (for any $(F, h^F)$) established (\ref{eq_top_final_112}) and (\ref{eq_top_final_2}).
		Essentially the same proof works in our more general situation.
		The details are given in the end of the proof of \cite[Theorem 1.1]{FinOTAs}.
	\end{proof}
	\par
	\begin{sloppypar}
	Now, let us introduce some further notations. 
	We fix $y_0 \in Y$ and an orthogonal basis $w_1, \ldots, w_{n - m}$ of $(N^{X|Y}_{y_0})^{(1, 0)}$, verifying $\norm{w_i} = \frac{1}{\sqrt{2}}$, $i = 1, \ldots, n - m$.
	For example, take $w_j = \frac{\partial}{\partial z_{j + m}}$, $j = 1, \ldots, n - m$, where $z_i$ are the complex coordinates induced by Fermi coordinates at $y_0$.
	For any $i, j \in \nat$, we define the operator
	\begin{equation}\label{eq_lambda_eq_defined}
		\Lambda_{\omega, =} : {\rm{Sym}}^{i} ((N^{X|Y})^{(1, 0)*}) \otimes  {\rm{Sym}}^{j}((N^{X|Y})^{(0, 1)*}) \to \comp,
	\end{equation}
	for multiindices $\alpha, \beta \in \nat^{n - m}$, as follows
	\begin{equation}
		\Lambda_{\omega, =} \big[ w^{\alpha} \otimes \overline{w}^{\beta} \big]
		=
		\begin{cases}
			\frac{1}{\pi^{|\beta|}} \beta!,  & \text{if } \alpha = \beta,
			\\
			0,  & \text{otherwise}.
		\end{cases}
	\end{equation}
	Clearly, this operator does not depend on the choice of the basis $w_1, \ldots, w_{n - m}$. 
	By linearity, we extend $\Lambda_{\omega, =}[\cdot]$ to ${\rm{Sym}}^{i} ((N^{X|Y})^{*})  \otimes \enmr{\iota^* F} \otimes \comp $.
	\end{sloppypar}
	\begin{prop}\label{prop_c_1c_2_calc}
		The constants $C_1, C_2 > 0$ from (\ref{eq_mh_norm_operators}) are given by
		\begin{equation}
		\begin{aligned}
			& 
			C_1 = \sup_{y \in Y} \Big( \kappa_N^{X|Y}(y)^{\frac{1}{2}} \cdot \big\| \Lambda_{\omega, =} \big[ g_y^* \otimes g_y \big] \big\|^{\frac{1}{2}} \Big),
			\\
			&
			C_2 = \sup_{y \in Y} \Big( \kappa_N^{X|Y}(y)^{-\frac{1}{2}} \cdot \big\| \Lambda_{\omega, =} \big[ g_y \otimes g_y^* \big] \big\|^{\frac{1}{2}} \Big).
		\end{aligned}
		\end{equation}		
	\end{prop}
	\begin{proof}
		The main idea of the proof is to reduce the calculation of the norms from (\ref{eq_mh_norm_operators}) to the calculation of the norms of some Toeplitz operators.
		\par 
		An easy calculation using (\ref{eq_kappan}) shows that for any $f \in L^2(Y, \iota^*(L^p \otimes F))$, we have
		\begin{equation}\label{eq_norm_1stterm_er}
			\big\| M_{g, p}^{X|Y} f \big\|_{L^2(dv_X)}
			=
			\big\| h_{g, p}(y) \cdot B_p^Y f \big\|_{L^2(dv_Y)},
		\end{equation}
		where the section $h_{g, p} \in \ccal^{\infty}(Y, \enmr{\iota^* F})$ satisfies
		\begin{multline}\label{eq_func_f_defn}
			h_{g, p}(y)^2 
			:= 
			\int_{\real^{2(n - m)}} 
			\kappa_N^{X|Y}(y, \sqrt{p} Z_N) 
			\cdot  
			\exp (- p \pi |Z_N|^2 ) 
			\cdot
			\\
			\cdot
			\{ g \} (y, \sqrt{p} Z_N)^*
			\cdot
			\{ g \} (y, \sqrt{p} Z_N)
			\cdot
			\rho \Big(  \frac{|Z_N|}{r_{\perp}} \Big)^2 
			dZ_{2m + 1} \wedge \cdots \wedge dZ_{2n}.
		\end{multline}
		By an easy calculation using (\ref{eq_gauss_integral}), there is $c > 0$, such that, as $p \to \infty$, we have
		\begin{equation}\label{eq_calc_fp_ex}
			h_{g, p}(y)^2 = \frac{\kappa_N^{X|Y}(y) \cdot \Lambda_{\omega, =} [ g(y)^* \otimes g(y) ] }{p^{n - m}}  + O \Big( \frac{1}{p^{n - m + \frac{1}{2}}} \Big).
		\end{equation}
		\par 
		Now, let $h \in \ccal^{\infty}(Y, \enmr{\iota^* F})$ verifies $h(y)^2 = \kappa_N^{X|Y}(y) \cdot \Lambda_{\omega, =} [ g(y)^* \otimes g(y) ]$.
		Clearly, bounded geometry condition and assumption $g \in \oplus_{k = 0}^{\infty} \ccal^{\infty}_{b}(Y, {\rm{Sym}}^k (N^{X|Y})^{*} \otimes \enmr{\iota^* F})$ imply that $h(y) \in \ccal^{\infty}_{b}(Y, \enmr{\iota^* F})$.
		Trivially, Toeplitz operator $T_{h, p}^{Y}$ satisfies
		\begin{equation}\label{eq_torpeldfdsf}
			\big\langle T_{h, p}^{Y} f,  f  \big\rangle_{L^2(dv_Y)}
			=
			\big\| h \cdot B_p^Y f \big\|_{L^2(dv_Y)}.
		\end{equation}
		Thus, by (\ref{eq_norm_1stterm_er}), (\ref{eq_calc_fp_ex}) and (\ref{eq_torpeldfdsf}), we have 
		\begin{equation}\label{eq_top_final_1}
			\big\| M_{g, p}^{X|Y} \big\|
			=
			\frac{1}{p^{\frac{n - m}{2}}} 
			\big\| T_{h, p}^{Y} \big\| + O \Big( \frac{1}{p^{\frac{n -m + 1}{2}}} \Big).
		\end{equation}
		We deduce the first part of Proposition \ref{prop_c_1c_2_calc} by Lemma \ref{lem_norm_toepl} and (\ref{eq_top_final_1}).
		\par 
		Now, to get the second part, let us first remark that the following formula holds
		\begin{equation}
			(M_{g, p}^{Y|X, \dagger})^*
			=
			(\kappa_N^{X|Y})^{-1}
			\cdot
			M_{g^*, p}^{X|Y}	.	
		\end{equation}
		The proof now proceeds in the same way as the proof for the first part.
	\end{proof}

\subsection{Asymptotic criteria for Toeplitz type operators}\label{sect_as_crit}
	As we approach the study of Toeplitz operators with exponential decay through the asymptotic expansions of their Schwartz kernels, it is fundamental to find characterizations of the latter operators in terms of the former asymptotic expansions. This is the main goal of this section.
	\par 
	To state and prove those characterizations in the generality we need, we will introduce a notion of Toeplitz operators, which is weaker than the one from Definition \ref{defn_ttype}.
	\begin{defn}\label{defn_ttype_weak}
		A sequence of operators $T_p^Y$, (resp. $T_{p}^{Y|X}$, $T_{p}^{X|Y}$), $p \in \nat$, as in Definition \ref{defn_ttype} is called a \textit{Toeplitz operator with \underline{weak} exponential decay} (resp. \textit{of type} $Y|X$, $X|Y$) \textit{with respect to $Z$} if there is a sequence $f_i$ (resp. $g_i^{h}$, $g_i^{a}$), $i \in \nat$, as in Definition \ref{defn_ttype}, such that the estimate (\ref{eq_toepl_off_diag}) holds, where on the right-hand side instead of $\dist_Y(y_1, y_2)$ (resp. $\dist_X(x, y_1)$, $\dist_X(y_1, x)$) we put $\dist_Z(y_1, y_2)$ (resp. $\dist_Z(x, y_1)$, $\dist_Z(y_1, x)$) for some Riemannian manifold $(Z, g^{TZ})$ and an embedding $\iota': X \to Z$, such that $(\iota')^* g^{TZ} = g^{TX}$, and such that the triple $(Z, X, g^{TZ})$ is of bounded geometry.
		To shorten, we sometimes omit the reference to $Z$. The coefficients of the asymptotic expansions will still be denoted by $[T_p^Y]_i$, $[T_{p}^{Y|X}]_i$, $[T_{p}^{X|Y}]_i$, $i \in \nat$.
	\end{defn}
	\begin{rem}\label{rem_weak_vers_ttype}
		a) The analogue of Remark \ref{rem_defn_tpl}a) holds for this weaker notion of Toeplitz operators due to Propositions \ref{prop_exp_bound_int}, \ref{prop_norm_bnd_distk_expbnd}.
		\par 
		b) From Proposition \ref{prop_qisom_bgeom}, we conclude that if $\iota'$ is quasi-isometry, the notions of Toeplitz operators with weak exponential decay of types $Y|X$, $X|Y$ with respect to $Z$ coincide with the corresponding non-weak notions.
	\end{rem}
	\begin{thm}\label{thm_ma_mar_crit_exp_dec}
		Let $(Y, g^{TY})$ be of bounded geometry. Then a family of linear operators $T_p^Y \in \enmr{ L^2(Y, \iota^*( L^p \otimes F)) }$, $p \in \nat$, forms a Toeplitz operator with (resp. weak) exponential decay if and only if the following conditions hold
		\begin{enumerate}
			\item For any $p \in \nat$,  $T_p^Y = B_p^Y \circ T_p^Y \circ B_p^Y$.
			\item There is $p_1 \in \nat$, such that for any $l \in \nat$, there is $C > 0$, such that for any $p \geq p_1$, the Schwartz kernel $T_p^Y(y_1, y_2)$; $y_1, y_2 \in Y$, of $T_p^Y$, evaluated with respect to $dv_Y$, satisfies
			\begin{equation}\label{eq_exp_est_ass_1}
				\Big|  
					T_p^Y (y_1, y_2) 
				\Big|_{\ccal^l} 
				\leq 
				C p^{m + \frac{l}{2}} 
				\cdot 
				\exp \big(- c \sqrt{p} \cdot \dist_Y(y_1, y_2) \big),
			\end{equation}
			\begin{equation}\label{eq_exp_est_ass_1_wk_vers}				
				\Big( \text{resp. }				
				\Big|  
					T_p^Y (y_1, y_2) 
				\Big|_{\ccal^l} 
				\leq 
				C p^{m + \frac{l}{2}} 
				\cdot 
				\exp \big(- c \sqrt{p} \cdot \dist_Z(y_1, y_2) \big)
				\Big),
			\end{equation}
			where in the last equation we used the notations from Definition \ref{defn_ttype_weak}.
			\item
			For any $y_0 \in Y$, $r \in \nat$, there are $I_r^Y(Z_Y, Z'_Y) \in \enmr{F_{y_0}}$ polynomials in $Z_Y, Z'_Y \in \real^{2m}$ of the same parity as $r$, such that the coefficients of $I_r^Y$ lie in $\ccal^{\infty}_{b}(Y, \enmr{\iota^* F})$, and for $F_r := I_r^Y \cdot \mathscr{P}_{m}$, the following holds.
			There are $\epsilon, c > 0$, $p_1 \in \nat^*$, such that for any $k, l, l' \in \nat$, there are $C, Q > 0$, such that for any $y_0 \in Y$, $p \geq p_1$, $Z_Y, Z'_Y \in \real^{2m}$, $|Z_Y|, |Z'_Y| \leq \epsilon$, $\alpha, \alpha' \in \nat^{2m}$, $|\alpha|+|\alpha'| \leq l$, the following bound holds
			\begin{multline}\label{eq_tpy_defn_exp_tay12as}
				\bigg| 
					\frac{\partial^{|\alpha|+|\alpha'|}}{\partial Z_Y^{\alpha} \partial Z'_Y{}^{\alpha'}}
					\bigg(
						\frac{1}{p^m} T_p^Y \big(\phi_{y_0}^{Y}(Z_Y), \phi_{y_0}^{Y}(Z'_Y) \big)
						\\
						-
						\sum_{r = 0}^{k}
						p^{-\frac{r}{2}}						
						F_r(\sqrt{p} Z_Y, \sqrt{p} Z'_Y) 
						\kappa_{\phi}^{Y}(Z_Y)^{-\frac{1}{2}}
						\kappa_{\phi}^{Y}(Z'_Y)^{-\frac{1}{2}}
					\bigg)
				\bigg|_{\ccal^{l'}}
				\\
				\leq
				C p^{-\frac{k + 1 - l}{2}}
				\Big(1 + \sqrt{p}|Z_Y| + \sqrt{p} |Z'_Y| \Big)^{Q} \exp(- c \sqrt{p} |Z_Y - Z'_Y|),
			\end{multline}
			where the $\ccal^{l'}$-norm is taken with respect to $y_0$.
		\end{enumerate}		
		Moreover, (\ref{eq_tpy_defn_exp_tay12as}) is related to the expansion from Definition \ref{defn_ttype} by $I_0^Y(0, 0) = [T_p^Y]_0$.
	\end{thm}
	\begin{proof}
		The proof for Toeplitz operators with weak exponential decay is analogous to the proof for Toeplitz operators with exponential decay, so we only concentrate on the proof of the former one.
		\par 
		First of all, let us assume that the sequence of operators $T_p^Y$, $p \in \nat$, forms a Toeplitz operator with exponential decay. 
		Then the first condition of Theorem \ref{thm_ma_mar_crit_exp_dec} holds by definition.
		The second holds due to Lemma  \ref{lem_exp_dec_toepl}.
		The third holds due to Lemma \ref{lem_toepl_tay_type}. 
		The identity $I_0^Y(0, 0) = [T_p^Y]_0$ follows from Lemma \ref{lem_toepl_tay_type}.
		Overall, we obtain one direction of Theorem \ref{thm_ma_mar_crit_exp_dec}.
		\par 
		Let us now prove the opposite direction.
		Our proof is based on \cite[Theorem 4.9]{MaMar08a}, where Ma-Marinescu proved the analogous theorem for compact manifolds and Toeplitz operators in the sense of \cite[\S 7]{MaHol}, see Remark \ref{rem_defn_tpl}a).
		The first step of their proof shows that the polynomial $I_0^Y(Z_Y, Z'_Y)$ from Theorem \ref{thm_ma_mar_crit_exp_dec} is constant, and hence equal to $I_0^Y(0, 0)$.
		Their argument (which doesn't use the assumption on parity of $I_r^Y(Z_Y, Z'_Y)$, $r \in \nat$) adapts line by line in our non-compact setting, except that the estimate  \cite[(4.47)]{MaMar08a} has to be replaced by Corollary \ref{cor_norm_bnd_oper}.
		\par 
		Then Ma and Marinescu define a section $f_0 \in \ccal^{\infty}(Y, \enmr{\iota^* F})$ as $f_0(y_0) := I_0^Y(0, 0)$.
		Our assumption on the coefficients of $I_r^Y$ implies that $f_0 \in \ccal^{\infty}_{b}(Y, \enmr{\iota^* F})$.
		From Lemma \ref{lem_toepl_tay_type}, most notably the fact that $I_{0, f_0}^Y(Z_Y, Z'_Y) = f_0(y_0)$, the fact that $I_0^Y(Z_Y, Z'_Y)$ is constant and the choice of $f_0$, we see that all the assumptions of Theorem \ref{thm_ma_mar_crit_exp_dec} are satisfied for the sequence of operators $\sqrt{p} \cdot (T_p^Y - T_{f_0, p}^{Y})$, $p \in \nat$ (except for the parity of $I_r^Y$, which is now opposite to $r$).
		\par 
		By repeating this argument for $\sqrt{p} \cdot (T_p^Y - T_{f_0, p}^{Y})$ instead of $T_p^Y$, we conclude that the first polynomial in Taylor-type expansion of $\sqrt{p} \cdot (T_p^Y - T_{f_0, p}^{Y})$, as in (\ref{eq_tpy_defn_exp_tay12as}), is constant. It is, however, of odd parity due to the parenthesized remark above. 
		Hence, the first coefficient is equal to $0$.
		Due to this, we see that all the assumptions of Theorem \ref{thm_ma_mar_crit_exp_dec} (now, even for the parity of $I_r^Y$) are satisfied for the sequence of operators $p \cdot (T_p^Y - T_{f_0, p}^{Y})$, $p \in \nat$.
		In particular, the first equation of (\ref{eq_toepl_off_diag}) holds for $k = 1$ by (\ref{eq_exp_est_ass_1}), applied for $p \cdot (T_p^Y - T_{f_0, p}^{Y})$.
		\par 
		Now, to deduce the first equation of (\ref{eq_toepl_off_diag}) for $k \geq 2$, we need to repeat the same procedure for the sequence of operators $p \cdot (T_p^Y - T_{f_0, p}^{Y})$, $p \in \nat$.
		By induction, we get a sequence $f_i \in \ccal^{\infty}_{b}(Y, \enmr{\iota^* F})$, $i \in \nat$, which satisfies the first equation of (\ref{eq_toepl_off_diag}) for any $k \in \nat$.
	\end{proof}
	We will now describe the analogue of Theorem \ref{thm_ma_mar_crit_exp_dec} for Toeplitz operators with exponential decay of type $X|Y$.
	\begin{sloppypar}
	\begin{thm}\label{thm_ma_mar_crit_exp_dec2}
		Let $(X, Y, g^{TX})$ be a triple of bounded geometry.
		Then a family $T_p^{X|Y} : L^2(Y, \iota^*( L^p \otimes F)) \to L^2(X, L^p \otimes F)$, $p \in \nat$, of linear operators forms a Toeplitz operator with (resp. weak) exponential decay of type $X|Y$ if and only if the following conditions are satisfied:
		\begin{enumerate}
			\item For any $p \in \nat$,  $T_p^{X|Y} = (B_p^X - B_p^{X|Y \perp})  \circ T_p^{X|Y} \circ B_p^Y$.
			\item There is $p_1 \in \nat^*$, such that for any $l \in \nat$, there is $C > 0$, such that for any $p \geq p_1$, the Schwartz kernel $T_p^{X|Y}(x, y)$; $x \in X$, $y \in Y$, of $T_p^{X|Y}$, evaluated with respect to $dv_Y$, satisfies
			\begin{equation}\label{eq_thm_tpxy_off_d_char}
				\Big|  
					T_p^{X|Y} (x, y) 
				\Big|_{\ccal^l} 
				\leq 
				C p^{m + \frac{l}{2}} 
				\cdot 
				\exp \big(- c \sqrt{p} \cdot \dist(x, y) \big),
			\end{equation}
			\begin{equation}\label{eq_thm_tpxy_off_d_char_weak}
				\Big( \text{resp. } 
				\Big|  
					T_p^{X|Y} (x, y) 
				\Big|_{\ccal^l} 
				\leq 
				C p^{m + \frac{l}{2}} 
				\cdot 
				\exp \big(- c \sqrt{p} \cdot \dist_Z(x, y) \big)
				\Big),
			\end{equation}
			where in the last equation we used notations from Definition \ref{defn_ttype_weak}.
			\item
			For any $y_0 \in Y$, $r \in \nat$, there are $I_r^{E}(Z, Z'_Y) \in \enmr{F_{y_0}}$ polynomials in $Z \in \real^{2n}$, $Z'_Y \in \real^{2m}$ of the same parity as $r$, such that the coefficients of $I_r^{E}$ lie in $\ccal^{\infty}_{b}(Y, \enmr{\iota^* F})$, and for $F_r^{E} := I_r^{E} \cdot \mathscr{E}_{n, m}$, the following holds.
			There are $\epsilon, c > 0$, $p_1 \in \nat^*$, such that for any $k, l, l' \in \nat$, there are $C, Q > 0$, such that for any $y_0 \in Y$, $p \geq p_1$, $Z \in \real^{2n}$, $Z'_Y \in \real^{2m}$, $|Z|, |Z'_Y| \leq \epsilon$, $\alpha \in \nat^{2n}$, $\alpha' \in \nat^{2m}$, $|\alpha|+|\alpha'| \leq l$, we have
			\begin{multline}\label{eq_berg_off_diag2}
				\bigg| 
					\frac{\partial^{|\alpha|+|\alpha'|}}{\partial Z^{\alpha} \partial Z'_Y{}^{\alpha'}}
					\bigg(
						\frac{1}{p^m} T_p^{X|Y} \big(\psi_{y_0}^{X|Y}(Z), \phi_{y_0}^{Y}(Z'_Y) \big)
						\\
						-
						\sum_{r = 0}^{k}
						p^{-\frac{r}{2}}						
						F_r^{E}(\sqrt{p} Z, \sqrt{p} Z'_Y) 
						\kappa_{\psi}^{X|Y}(Z)^{-\frac{1}{2}}
						\kappa_{\phi}^{Y}(Z'_Y)^{-\frac{1}{2}}
					\bigg)
				\bigg|_{\ccal^{l'}}
				\\
				\leq
				C p^{-\frac{k + 1 - l}{2}}
				\Big(1 + \sqrt{p}|Z| + \sqrt{p} |Z'_Y| \Big)^{Q} \exp(- c \sqrt{p}( |Z_N| + |Z_Y - Z'_Y| ) ).
			\end{multline}
		\end{enumerate}	 
		Moreover, in the notations of (\ref{eq_brack_defn}), for any $y_0 \in Y$, the polynomial $I_0^{E}(Z, Z'_Y)$ depends only on $z_N$, and, as a section of $\oplus_{k = 1}^{\infty} {\rm{Sym}}^{k} (N^{X|Y})^{(1, 0)*} \otimes \enmr{\iota^* F}$ over $Y$, it coincides with $[T_p^{X|Y}]_0 \cdot \kappa_N^{X|Y}(y_0)^{\frac{1}{2}}$.
	\end{thm}
	\end{sloppypar}
	\begin{sloppypar}
	\begin{proof}
		As the proof for the weak version of Toeplitz operators is completely analogous to the proof of non-weak version, we only concentrate on the former one.
		\par 
		First of all, let us assume that the sequence of operators $T_p^{X|Y}$, $p \in \nat$, forms a Toeplitz operator with exponential decay of type $X|Y$. 
		Then the first condition of Theorem \ref{thm_ma_mar_crit_exp_dec2} holds by definition.
		The second holds due to Lemma \ref{lem_exp_dec_toepl}.
		The third holds due to Lemma \ref{lem_toepl_tay_type}. 
		The relation between $I_0^E$ and $[T_p^{X|Y}]_0$ follows from Lemma \ref{lem_toepl_tay_type}.
		Overall, we obtain one direction of Theorem \ref{thm_ma_mar_crit_exp_dec2}.
		\par 
		Let us now prove the opposite direction.
		From the first condition of Theorem \ref{thm_ma_mar_crit_exp_dec2} and Corollary \ref{cor_poly_incomp}, we first deduce that in the notations of (\ref{eq_berg_off_diag2}), $I_0^{E}(Z, Z'_Y)$ is a polynomial in $z, \overline{z}'_Y$.
		Let us show that it only depends on $z_N$.
		\par 
		Consider first the operator $T_p^{Y} := \res_Y \circ T_p^{X|Y}$.
		Of course $T_p^{Y} = 0$ due to our first assumption on $T_p^{X|Y}$.
		From Lemma \ref{lem_toepl_tay_type}, we obtain as a consequence that $I_0^{E}(Z_Y, Z'_Y) = 0$.
		\par 
		Now, let $U$ be a smooth vector field, such that $U|_Y \in N^{X|Y}$, $U(y_0) = \frac{\partial}{\partial z_j}$, $j = m+1, \ldots, n$, where $z_1, \ldots, z_n$ are the complex coordinates associated to Fermi coordinates.
		We consider the sequence of operators $T_{p, 1}^{Y} := \frac{1}{\sqrt{p}} \res_Y \circ B_p^{X} \circ \nabla_U T_{p}^{X|Y}$.
		This sequence of operators is well-defined due to (\ref{eq_defn_res_map}).
		It follows from (\ref{eq_defn_res_map}) that $T_{p, 1}^{Y}$, $p \in \nat^*$, satisfies the first condition from Theorem \ref{thm_ma_mar_crit_exp_dec}.
		The second condition associated to the weak notion also holds for $Z := X$ in the notations of Definition \ref{defn_ttype_weak} (remark that the strong version holds only under additional quasi-isometry assumption, see Proposition \ref{prop_qisom_bgeom}. This technical caveat is one of the reasons why we need to consider weak version of Toeplitz type operators).
		\par 
		An easy verification using (\ref{eq_lem_comp_poly_2}) shows that $T_{p, 1}^{Y}$ satisfies the third assumption of Theorem \ref{thm_ma_mar_crit_exp_dec} for $I_0^Y(Z_Y, Z'_Y) :=  (\frac{\partial }{\partial z_j} I_0^E)(Z_Y, Z'_Y)$.
		Hence, by the results of Theorem \ref{thm_ma_mar_crit_exp_dec} and its proof, we conclude that the sequence of operators $T_{p, 1}^{Y}$, $p \in \nat$, forms a Toeplitz operator with weak exponential decay associated to $X$, and $(\frac{\partial}{\partial z_j} I_0^{E})(Z_Y, Z'_Y)$ a constant.
		We construct $g'_1 \in \ccal^{\infty}_{b}(Y,  (N^{X|Y})^{(1, 0)*} \otimes \enmr{\iota^* F})$, so that for any $n \in (N^{X|Y}_{y_0})^{(1, 0)}$, we have $g'_1 \cdot n = \sum_{j = m + 1}^{n} (\frac{\partial}{\partial z_j} I_0^{E})(0, 0) \cdot n_j$, where $n_j$ are the coordinates of $n$ in the basis $\frac{\partial}{\partial z_j}$.
		We define $T_{p, 1}^{X|Y} := T_{p}^{X|Y} - (B_p^X - B_p^{X|Y \perp}) \circ ( \llangle g'_1 \rrangle \cdot \ext_p^{X|Y} )$.
		By Lemma \ref{lem_toepl_tay_type}, (\ref{eq_berg_off_diag2}) and the remark after it, we deduce that the asymptotic expansion (\ref{eq_tpy_defn_exp_tay12as}) holds for $T_p^{X|Y} := T_{p, 1}^{X|Y}$ and $I_0^{E} := P_2 I_0^{E}$, where $P_i$, $i \in \nat$, is the projection onto the vector space of polynomials of degree $\geq i$.
		\par 
		We then repeat the procedure for the pair of smooth vector fields $U$, $V$, verifying similar assumptions as above, and the sequence of operators $T_{p, 2}^{Y} := \frac{1}{p} \res_Y \circ B_p^{X} \circ \nabla_U \nabla_V T_{p, 1}^{X|Y}$ to construct $g'_2 \in \ccal^{\infty}_{b}(Y, {\rm{Sym}}^2 (N^{X|Y})^{(1, 0)*} \otimes \enmr{\iota^* F})$.
		Then, as before, we form the sequence of operators $T_{p, 2}^{X|Y} := T_{p, 1}^{X|Y} - (B_p^X - B_p^{X|Y \perp}) \circ ( \llangle g'_2 \rrangle \cdot \ext_p^{X|Y} )$.
		By continuing in the same fashion, we construct the sequence of elements $g'_k \in \ccal^{\infty}_{b}(Y, {\rm{Sym}}^k (N^{X|Y})^{(1, 0)*} \otimes \enmr{\iota^* F})$, $k \in \nat^*$, and operators $T_{p, k}^{X|Y}$, $k \in \nat$, such that the asymptotic expansion (\ref{eq_tpy_defn_exp_tay12as}) holds for $T_p^{X|Y} := T_{p, k}^{X|Y}$ and $I_0^{E} := P_{k + 1} I_0^{E}$ .
		\par
		Of course, since $I_0^{E}(Z, Z'_Y)$ is a polynomial, only a finite number of $g'_k$, $k \in \nat^*$, is non-zero.
		We put $g_0 := \sum_{i = 1}^{\infty} g'_i$.
		Clearly, $g_0$ has the same parity as $I_0^{E}$.
		By the above, we see that the asymptotic expansion (\ref{eq_berg_off_diag2}) holds for $T_p^{X|Y} := T_p^{X|Y} - (B_p^{X} - B_p^{X|Y \perp}) \circ ( \llangle g_0 \rrangle \cdot \ext_p^{X|Y} )$ and $I_0^{E} := 0$.
		Hence, the same asymptotic expansion (\ref{eq_berg_off_diag2}) holds for $\sqrt{p} \big(T_p^{X|Y} - (B_p^{X} - B_p^{X|Y \perp}) \circ ( \llangle g_0 \rrangle \cdot \ext_p^{X|Y} ) \big)$.
		We repeat the same procedure for the new sequence of operators and construct an element $g_1$. 
		Clearly, by the assumptions on the parity of $I_r^E$, the parity of $g_1$ is different from $g_0$. 
		By induction, we get a sequence of elements $g_i$, $i \in \nat$, which satisfy the second equation from (\ref{eq_toepl_off_diag}), and the parities of which are as we need.
	\end{proof}
	\end{sloppypar}
	We are finally ready to treat the last type of Toeplitz operators with exponential decay.
	\begin{thm}\label{thm_ma_mar_crit_exp_dec3}
		A family $T_p^{Y|X} : L^2(X, L^p \otimes F) \to L^2(Y, \iota^*( L^p \otimes F))$, $p \in \nat$, of linear operators forms a Toeplitz operator with exponential decay of type $Y|X$ if and only if the following three conditions are satisfied:
		\begin{enumerate}
			\item For any $p \in \nat$,  $T_p^{Y|X} = B_p^Y  \circ T_p^{Y|X} \circ (B_p^X - B_p^{X|Y \perp})$.
			\item There is $p_1 \in \nat^*$, such that for any $l \in \nat$, there is $C > 0$, such that for any $p \geq p_1$, the Schwartz kernel $T_p^{Y|X}(y, x)$; $x \in X$, $y \in Y$, of $T_p^{Y|X}$, evaluated with respect to $dv_X$, satisfies
			\begin{equation}\label{eq_thm_tpxy_off_d_char2}
				\Big|  
					T_p^{Y|X} (y, x) 
				\Big|_{\ccal^l} 
				\leq 
				C p^{n + \frac{l}{2}} 
				\cdot 
				\exp \big(- c \sqrt{p} \cdot \dist(x, y) \big),
			\end{equation}
			\begin{equation}\label{eq_thm_tpxy_off_d_char2_weak}
				\Big( \text{resp. } 
				\Big|  
					T_p^{Y|X} (y, x) 
				\Big|_{\ccal^l} 
				\leq 
				C p^{n + \frac{l}{2}} 
				\cdot 
				\exp \big(- c \sqrt{p} \cdot \dist_Z(x, y) \big)
				\Big),
			\end{equation}
			where in the last equation we used the notations from Definition \ref{defn_ttype_weak}.
			\item
			For any $y_0 \in Y$, $r \in \nat$, there are $I_r^{R}(Z_Y, Z') \in \enmr{F_{y_0}}$ polynomials in $Z_Y \in \real^{2m}$, $Z' \in \real^{2n}$ of the same parity as $r$, such that the coefficients of $I_r^{R}$ lie in $\ccal^{\infty}_{b}(Y, \enmr{\iota^* F})$, and for $F_r^{R} := I_r^{R} \cdot \mathscr{Res}_{n, m}$, the following holds.
			There are $\epsilon, c > 0$, $p_1 \in \nat^*$, such that for any $k, l, l' \in \nat$, there are $C, Q > 0$, such that for any $y_0 \in Y$, $p \geq p_1$, $Z_Y \in \real^{2m}$, $Z' \in \real^{2n}$, $|Z_Y|, |Z'| \leq \epsilon$, $\alpha \in \nat^{2m}$, $\alpha' \in \nat^{2n}$, $|\alpha|+|\alpha'| \leq l$, we have
			\begin{multline}\label{eq_berg_off_diag3}
				\bigg| 
					\frac{\partial^{|\alpha|+|\alpha'|}}{\partial Z_Y^{\alpha} \partial Z'{}^{\alpha'}}
					\bigg(
						\frac{1}{p^n} T_p^{Y|X} \big(\phi_{y_0}^{Y}(Z_Y), \psi_{y_0}^{X|Y}(Z') \big)
						\\
						-
						\sum_{r = 0}^{k}
						p^{-\frac{r}{2}}						
						F_r^{R}(\sqrt{p} Z_Y, \sqrt{p} Z') 
						\kappa_{\phi}^{Y}(Z_Y)^{-\frac{1}{2}}
						\kappa_{\psi}^{X|Y}(Z')^{-\frac{1}{2}}
					\bigg)
				\bigg|_{\ccal^{l'}}
				\\
				\leq
				C p^{- \frac{k + 1 - l}{2}}
				\Big(1 + \sqrt{p}|Z_Y| + \sqrt{p} |Z'| \Big)^{Q} \exp(- c \sqrt{p}( |Z'_N| + |Z_Y - Z'_Y| ) ).
			\end{multline}
		\end{enumerate}		 
		Moreover, in the notations of (\ref{eq_brack_defn}) and (\ref{eq_berg_off_diag2}), the polynomial $I_0^{R}(Z_Y, Z')$ depends only on $\overline{z}'_N$, and as a section of $\oplus_{k = 1}^{\infty} {\rm{Sym}}^{k} (N^{X|Y})^{(0, 1)*} \otimes \enmr{\iota^* F}$ over $Y$, it coincides with $[T_p^{Y|X}]_0 \cdot \kappa_N^{X|Y}(y_0)^{-\frac{1}{2}}$.
	\end{thm}
	\par 
	Clearly, Lemmas \ref{lem_exp_dec_toepl}, \ref{lem_toepl_tay_type} imply the first implication of Theorem \ref{thm_ma_mar_crit_exp_dec3}.
	The proof of the second implication will be given in Section \ref{sect_adj_t_oper}, where we study adjoints of Toeplitz type operators.

\section{Toeplitz type operators: algebraic properties and examples}\label{sect_alg_prop}
	The main goal of this section is to study algebraic properties of the set of Toeplitz type operators and to construct some examples of those operators.
	More precisely, in Section \ref{sect_prod_ext_res}, we show that the set of Toeplitz type operators is closed under taking restrictions, extensions and some products.
	In Section \ref{sect_adj_t_oper}, we prove the analogous statement for the adjoints of Toeplitz type operators.
	To do this and out of independent interest, we introduce a sequence of operators, so-called multiplicative defect, which plays a crucial role in our approach to the main statements of this article. 
	We also prove that the multiplicative defect is itself a Toeplitz type operator with weak exponential decay.
	Finally, in Section \ref{sect_basic_exmpl}, we provide several examples of Toeplitz type operators.

\subsection{Products, extensions and restrictions of Toeplitz type operators}\label{sect_prod_ext_res}

	The main goal of this section is to show that the set of Toeplitz type operators with (weak) exponential decay is closed under taking restrictions, extensions and some products.
	\par 
	\begin{sloppypar}
	To describe our main result, we fix some notations first.
	We have a natural isomoprhism
	\begin{equation}
		N^{X|Y} \to \iota_1^* N^{X|W} \oplus N^{W|Y}.
	\end{equation}
	We extend  the induced projection onto the $N^{W|Y}$ component to an operator on ${\rm{Sym}}^k (N^{X|Y})^{(1, 0)*}$, and denote it by an abuse of notation $P_N^{W|Y}$. 
	Recall that $\Lambda_{\omega, =} [ \cdot ]$ was defined in (\ref{eq_lambda_eq_defined}).
	\end{sloppypar}
	\begin{sloppypar}
	For $k, k' \in \nat^*$, we fix $g_1 \in \ccal^{\infty}_{b}(Y, {\rm{Sym}}^k (N^{X|Y})^{(1, 0)*} \otimes \enmr{\iota^* F})$, $g'_1 \in \ccal^{\infty}_{b}(Y, {\rm{Sym}}^{k'} (N^{X|Y})^{(0, 1)*} \otimes \enmr{\iota^* F})$, $g_2 \in \ccal^{\infty}_{b}(Y, {\rm{Sym}}^k (N^{W|Y})^{(1, 0)*} \otimes \enmr{\iota_1^* F})$, $g_3 \in \ccal^{\infty}_{b}(W, {\rm{Sym}}^{k'} (N^{X|W})^{(1, 0)*} \otimes \enmr{\iota_2^* F})$.
	The main result of this section goes as follows.
	\end{sloppypar}
	\begin{thm}\label{thm_ttype_closure2}
		The sequences of operators
		\begin{tasks}[style=enumerate](2)
			\task $\res_W \circ T_{\llangle g_1 \rrangle, p}^{X|Y}$,
			\task $\res_W \circ \ext_p^{X|Y} - \ext_p^{W|Y},$
		\end{tasks}
		for $p \in \nat$, form a Toeplitz operator with weak exponential decay of type $W|Y$ with respect to $X$.
		The sequence of operators
		\begin{tasks}[style=enumerate, resume]
			\task $T_{\llangle g'_1 \rrangle, p}^{Y|X} \circ T_{\llangle g_1 \rrangle, p}^{X|Y}$,
		\end{tasks}
		for $p \in \nat$, forms a Toeplitz operator with weak exponential decay on $Y$ with respect to $X$.
		Finally, the sequences of operators
        \begin{tasks}[style=enumerate, resume](3)
			\task $T_{\llangle g_3 \rrangle, p}^{X|W} \circ T_{\llangle g_2 \rrangle, p}^{W|Y},$
			\task $T_{\llangle g_3 \rrangle, p}^{X|W} \circ \ext_p^{W|Y},$
			\task $\ext_p^{X|W} \circ T_{\llangle g_2 \rrangle, p}^{W|Y},$
			\task $T_{\llangle g_1 \rrangle, p}^{X|Y} \circ T_{f, p}^{Y},$
		\end{tasks}
   		for $p \in \nat$, form Toeplitz operators with exponential decay of type $X|Y$.
   		Moreover, we have
   		\begin{tasks}[style=enumerate](2)
   			\task $\big[ \res_W \circ T_{\llangle g_1 \rrangle, p}^{X|Y} \big]_0 = P_N^{W|Y}(g_1),$
   			\task $[\res_W \circ \ext_p^{X|Y} - \ext_p^{W|Y}]_0 = 0,$
   			\task $[T_{\llangle g'_1 \rrangle, p}^{Y|X} \circ T_{\llangle g_1 \rrangle, p}^{X|Y}]_0 = \Lambda_{\omega, =} \big[ g'_1 \cdot g_1 \big],$
			\task $[T_{\llangle g_3 \rrangle, p}^{X|W} \circ T_{\llangle g_2 \rrangle, p}^{W|Y}]_0 = \iota_1^*(g_3) \cdot g_2,$
			\task $[T_{\llangle g_3 \rrangle, p}^{X|W} \circ \ext_p^{W|Y}]_0 = \iota_1^*(g_3),$
			\task $[\ext_p^{X|W} \circ T_{\llangle g_2 \rrangle, p}^{W|Y}]_0 = g_2,$
			\task $[T_{\llangle g_1 \rrangle, p}^{X|Y} \circ T_{f, p}^{Y}]_0 = g_1 \cdot f.$
		\end{tasks}
	\end{thm}
	\begin{rem}\label{rem_thm_ttype_closure2}
		In particular, from Proposition \ref{prop_qisom_bgeom}, if the embedding $\iota_2 : W \to X$ is quasi-isometry then, in points 1 and 2, the related sequences of operators form Toeplitz type operator with exponential decay.
		The same holds for the point 3 if the embedding $\iota : Y \to X$ is quasi-isometry.
	\end{rem}
	\begin{proof}
		The proofs of all those statements proceed by the verification that relevant operators satisfy the assumptions of Theorems \ref{thm_ma_mar_crit_exp_dec} and \ref{thm_ma_mar_crit_exp_dec2}.
		\par 
		For statements 1 and 2, the validity of the first condition from Theorem \ref{thm_ma_mar_crit_exp_dec2} follows from (\ref{eq_defn_res_map}).
		For statements 3, 4, 5, 6 and 7, the validity of the first condition from Theorems \ref{thm_ma_mar_crit_exp_dec}, \ref{thm_ma_mar_crit_exp_dec2} is direct.
		\par 
		The weak version of the second condition for $Z := X$ for statement 1 (resp. 2) follows trivially from Lemma \ref{lem_exp_dec_toepl} (resp. Theorem \ref{thm_ext_exp_dc}).
		For statements 4, 5, 6 and 7, (resp. 3) the validity of the second condition (resp. weak version of the second condition for $Z := X$) from Theorem \ref{thm_ma_mar_crit_exp_dec2} follows from Corollary \ref{cor_comp_exp_bound} and Lemma \ref{lem_exp_dec_toepl}.
		\par 
		Hence, it is only left to verify the third statement for each of the operators.
		This is slightly more delicate, and will be done separately for each of the statements.
		By doing so, and employing the relationship between the polynomials from the third condition of Theorems \ref{thm_ma_mar_crit_exp_dec} and \ref{thm_ma_mar_crit_exp_dec2} and the asymptotic expansions (\ref{eq_toepl_off_diag}), we also establish the second part of Theorem \ref{thm_ttype_closure2}.
		\par 
		Let us introduce the following notations first.
		For a function $f : \real^k \to \real$ and $s > 0$, we denote by $f_s(Z)$, the function given by $Z \mapsto \frac{1}{s} f(sZ)$. 
		For functions $P(Z, Z')$, $R(Z, Z')$, $Z \in \real^r$, $Z' \in \real^e$; $r, e \in \nat^*$, verifying $R(0, 0) \neq 0$, and two functions $f: \real^r \to \real^r$, $g: \real^e \to \real^e$, verifying $f(Z) = Z + O(|Z|^2)$, $g(Z') = Z' + O(|Z'|^2)$, we decompose $P(f_s(Z), g_s(Z'))$ as follows
		\begin{equation}
		\begin{aligned}
			&
			P(f_s(Z), g_s(Z'))= \sum_{i = 0}^{k}  P(f, g)_{[i]}(Z, Z') s^i + O(s^{k + 1}), 
			\\
			&
			\frac{R(f_s(Z), g_s(Z'))}{R(Z, Z')}= \sum_{i = 0}^{k}  R'(f, g)_{[i]}(Z, Z') s^i + O(s^{k + 1}), 
		\end{aligned}
		\end{equation}
		where $P(f, g)_{[i]}, R'(f, g)_{[i]}$ are functions, which do not depend on $s$.
		Clearly, $P(f, g)_{[i]}, R'(f, g)_{[i]}$ are polynomials if $P$ is a polynomial and $R$ is the exponential of a polynomial.
		When $f$ (resp. $g$) is the identity map, we write $P(f, Z')_{[i]}$ (resp. $P(Z, g)_{[i]}$) for $P(f, g)_{[i]}$.
		When $P$ or $R$ depend only on $Z$ or $Z'$, we write $P(f)_{[i]}(Z)$, $P(g)_{[i]}(Z')$ and $R'(f)_{[i]}(Z)$, $R'(g)_{[i]}(Z')$ for the above polynomials.
		\par 
		We use the notations introduced before Theorem \ref{thm_berg_off_diag}.
		From (\ref{eq_defn_sigma}) and (\ref{eq_frame_tilde2}), we deduce that 
		\begin{multline}\label{eq_resw_conn_g1}
			\res_W \circ T_{\llangle g_1 \rrangle, p}^{X|Y} \big(\psi^{W|Y}_{y_0}(Z_W), \phi^{Y}_{y_0}(Z'_Y) \big)
			=
			T_{\llangle g_1 \rrangle, p}^{X|Y} \big(\psi^{X|Y}_{y_0}(\sigma(Z_W)), \phi^{Y}_{y_0}(Z'_Y) \big)
			\cdot
			\\
			\cdot
			\exp(-p \tau_L - \tau_F) \big(\psi^{W|Y}_{y_0}(Z_W) \big)
			.
		\end{multline}
		\par For $k \in \nat$, we decompose $\exp(-\tau_F)$,  into power series expansion
		\begin{equation}\label{eq_tau_exp}
			\exp(-\tau_F)(Z_W) = \sum_{i = 0}^{k} \exp(-\tau_F)_{[i]}(Z_W) + O(|Z_W|^{k + 1}), 
		\end{equation}
		where $\exp(-\tau_F)_{[i]}$ are homogeneous polynomials in $Z_W$ of degree $i$.
		Using (\ref{eq_phi_fun_exp12}), we see that we can decompose $\exp(-p \tau_L)$ as follows
		\begin{equation}\label{eq_ptau_exp}
			\exp(-p \tau_L)(Z_W) = \sum_{i = 0}^{2 k} \sum_{j = 0}^{\lfloor \frac{i}{2} \rfloor} \sqrt{p}^{j} \exp(-p \tau_L)_{[i, j]}(Z_W) + O(\sqrt{p}^{k + 1} |Z_W|^{2k + 1}), 
		\end{equation}
		where $\exp(-p \tau_L)_{[i, j]}$ are homogeneous polynomials in $Z_W$ of degree $i$, independent of $p$.
		\par 
		Recall that $\kappa$-functions were defined in (\ref{eq_kappan}), (\ref{eq_defn_kappaxy1})  and (\ref{eq_defn_kappaxy2}).
		Clearly, from (\ref{eq_kappa_relation}), we have
		\begin{equation}\label{eq_kappa_corr_1_def}
			\kappa_{\psi}^{X|Y}(\sigma(Z_W))
			=
			\kappa_{\psi}^{W|Y}(Z_W)
			\cdot
			\kappa_{N}^{X|Y}(\psi^{X|Y}(\sigma(Z_W)))
			\cdot
			\kappa_{N}^{W|Y}(\psi^{W|Y}(Z_W))^{-1}
			\cdot
			\frac{\kappa_{\phi}^{Y}(\sigma_Y(Z_W))}{\kappa_{\phi}^{Y}(Z_Y)},
		\end{equation}
		where $\sigma_Y$ is the horizontal component of $\sigma$.
		\par 
		For $k \in \nat$, let us expand in a neighborhood of $y_0$
		\begin{equation}\label{eq_kappa_N_tay_exp}
		\begin{aligned}
			&
			\kappa_{N}^{X|W}(\psi^{X|Y}(\sigma(Z_W)))^{- \frac{1}{2}} = \sum_{i = 0}^{k} \kappa_{N, [i]}^{X|W}(\sigma)^{-  \frac{1}{2}}(Z_W) + O(|Z_W|^{k + 1}),
			\\
			&
			\kappa_{N}^{W|Y}(\psi^{W|Y}(Z_W))^{\frac{1}{2}} = \sum_{i = 0}^{k} (\kappa_{N, [i]}^{W|Y})^{\frac{1}{2}}(Z_W) + O(|Z_W|^{k + 1}),
		\end{aligned}
		\end{equation}
		where $\kappa_{N, [i]}^{X|W}(\sigma)^{-  \frac{1}{2}}(Z_W)$, $(\kappa_{N, [i]}^{W|Y})^{\frac{1}{2}}(Z_W)$ are homogeneous polynomials of degree $i$.
		\par 
		We also decompose
		\begin{equation}
			\Big(
				\frac{\kappa_{\phi}^{Y}(\sigma_Y(Z_W))}{\kappa_{\phi}^{Y}(Z_Y)}
			\Big)^{- \frac{1}{2}}
			=
			\sum_{i = 0}^{k} \Big(
				\frac{\kappa_{\phi}^{Y}(\sigma_Y)}{\kappa_{\phi}^{Y}}
			\Big)_{[i]}^{- \frac{1}{2}}(Z_W) + O(|Z_W|^{k + 1}),
		\end{equation}
		where $( \frac{\kappa_{\phi}^{Y}(\sigma_Y)}{\kappa_{\phi}^{Y}} )_{[i]}^{- \frac{1}{2}}(Z_W)$ are homogeneous polynomials of degree $i$.
		As $\sigma(Z_Y, 0) = Z_Y$, for any $i \in \nat^*$, the polynomials $( \frac{\kappa_{\phi}^{Y}(\sigma_Y)}{\kappa_{\phi}^{Y}} )_{[i]}^{- \frac{1}{2}}(Z_W)$ divide $Z_{N^{W|Y}}$, where $Z_W = (Z_Y, Z_{N^{W|Y}})$, and we have $( \frac{\kappa_{\phi}^{Y}(\sigma_Y)}{\kappa_{\phi}^{Y}} )_{[0]}^{- \frac{1}{2}} = 1$.
		For $r \in \nat$, we now introduce
		\begin{equation}\label{eq_kappa_cor_1}
			\kappa_{{\rm{cor}}, [r]}^{1}(Z_W)
			:=
			\sum_{a + b + c = r}
			\kappa_{N, [a]}^{X|Y}(\sigma)^{-  \frac{1}{2}}(Z_W)
			\cdot
			(\kappa_{N, [b]}^{W|Y})^{ \frac{1}{2}}(Z_W)
			\cdot
			\Big(
				\frac{\kappa_{\phi}^{Y}(\sigma_Y)}{\kappa_{\phi}^{Y}}
			\Big)_{[c]}^{- \frac{1}{2}}(Z_W).
		\end{equation}
		\par 
		From (\ref{eq_resw_conn_g1}), (\ref{eq_tau_exp}), (\ref{eq_ptau_exp}), (\ref{eq_kappa_corr_1_def}), (\ref{eq_kappa_N_tay_exp}) and (\ref{eq_kappa_cor_1}), we deduce that the asymptotic expansion (\ref{eq_berg_off_diag2}) holds for $X := W$ the operator $T_p^{X|Y} := \res_W \circ T_{\llangle g_1 \rrangle, p}^{X|Y}$ and the polynomials
		\begin{multline}\label{eq_jre_rest_op}
			I_r^{E}(Z_W, Z'_Y) := \sum_{a + b + c + d + e + f = r} (\res_l \circ J_{a, g_1}^{E})(\sigma, Z'_Y)_{[b]} 
			\cdot
			(\res_l \circ (\mathscr{E}'_{n, m}(\sigma, Z'_Y)_{[c]})) 
			\cdot
			\\
			\cdot 
			\exp(-\tau_F)_{[d]}(Z_W)
			\cdot
			\sum_{i - j = e} 
			\exp(-p \tau_L)_{[i, j]}(Z_W) \cdot \kappa_{{\rm{cor}}, [f]}^{1}(Z_W).
		\end{multline}
		From  (\ref{eq_ptau_exp}), we see that the second sum in (\ref{eq_jre_rest_op}) is finite.
		From this, we see that the first part of the first statement of Theorem \ref{thm_ttype_closure2} follows from Theorem \ref{thm_ma_mar_crit_exp_dec2}.
		The fact that the coefficients of $I_r^{E}$ are bounded with all their derivatives follows from Propositions \ref{prop_sigma_texp}, \ref{prop_tau} and the corresponding statement for the polynomials $J_{r, g_1}^{E}$, $r \in \nat$, from Lemma \ref{lem_toepl_tay_type}.
		The statement about the parity of $I_r^{E}$ follows from the analogous statements for $J_{b, g_1}^{E}$ from Lemma \ref{lem_toepl_tay_type} and the fact that $\exp(-p \tau_L)_{[i, j]}$ are non-zero only for even $j$.
		Now, from Propositions \ref{prop_sigma_texp}, \ref{prop_tau}, the expression for $J_{0, g_1}^{E}$ from Lemma \ref{lem_toepl_tay_type} and (\ref{eq_jre_rest_op}), we deduce that for any $Z_W = (Z_Y, Z_{N^{W|Y}}) \in \real^{2l}$; $Z_Y, Z'_Y \in \real^{2m}$
		\begin{equation}
			I_0^{E}(Z_W, Z'_Y) = g_1(y_0) \cdot Z_{N^{W|Y}}^{\otimes k} \cdot \kappa_{N}^{W|Y}(y_0)^{\frac{1}{2}}.
		\end{equation}
		From this, we deduce by the last remark from Theorem \ref{thm_ma_mar_crit_exp_dec2} the first statement of the second part of Theorem \ref{thm_ttype_closure2}.
		\par 
		The proof for the second statement of the first part of Theorem \ref{thm_ttype_closure2} is completely analogous to the proof for the first statement. 
		One only has to realize that the asymptotic expansion (\ref{eq_berg_off_diag2}) now holds for the operators $T_p^{X|Y} := \res_W \circ E_{p}^{X|Y} - E_{p}^{W|Y}$ and the polynomials
		\begin{multline}\label{eq_jre_rest_op112}
			I_r^{E}(Z_W, Z'_Y) := \sum_{a + b + c + d + e + f = r} 
			(\res_l \circ J_{a}^{X|Y, E})(\sigma, Z'_Y)_{[b]} 
			\cdot
			(\res_l \circ (\mathscr{E}'_{n, m}(\sigma, Z'_Y)_{[c]})) 
			\cdot
			\\
			\cdot 
			\exp(-\tau_F)_{[d]}(Z_W)
			\cdot
			\sum_{i - j = e} 
			\exp(-p \tau_L)_{[i, j]}(Z_W) 
			\cdot 
			\kappa_{{\rm{cor}}, [f]}^{1}(Z_W)
			-
			J_{r}^{W|Y, E}(Z_W, Z'_Y).
		\end{multline}
		From (\ref{eq_kmn_poly}), (\ref{eq_jre_rest_op112}) and the expressions for $J_{0}^{X|Y, E}$, $J_{0}^{W|Y, E}$ from Theorem \ref{thm_ext_as_exp}, we conclude that for $Z = (Z_Y, Z_N)$, we have
		\begin{equation}\label{eq_jre_rest_op112a}
			I_0^{E}(Z_W, Z'_Y)
			:=
			0.
		\end{equation}
		This establishes the second statement from the second part of Theorem \ref{thm_ttype_closure2} by the remark in the end of Theorem \ref{thm_ma_mar_crit_exp_dec2}.
		\par 
		Let us now treat the third statement from the first part of Theorem \ref{thm_ttype_closure2}.
		Directly from Lemma \ref{prop_exp_bound_int}, the last part of (\ref{eq_comp_ext_pmn1}) and the analysis, similar to the one before (\ref{eq_jrf_expr}), we conclude that the expansion (\ref{eq_tpy_defn_exp_tay12as}) holds for $T_p^Y := T_{\llangle g'_1 \rrangle, p}^{Y|X} \circ T_{\llangle g_1 \rrangle, p}^{X|Y}$ and
		\begin{equation}\label{eq_sec_stat_j_r}
			I_r^Y 
			:=
			\sum_{a + b = r}
			\res_m 
			\circ
			\mathcal{K}_{n, m} \big[ J_{a, g'_1}^{R}, J_{b, g_1}^{E} \big]
			\circ
			\res_m.
		\end{equation}
		From (\ref{eq_kmn_poly23}), (\ref{eq_sec_stat_j_r}) and the expression for $J_{0, g'_1}^{R}$, $J_{0, g_1}^{E}$ from Lemma \ref{lem_toepl_tay_type}, we conclude that for any $Z_Y, Z'_Y \in \real^{2m}$, we have
		\begin{equation}
			J_0(Z_Y, Z'_Y)
			:=
			\Lambda_{\omega, =} \big[ g'_1 \cdot g_1 \big].
		\end{equation}
		The statement about the parity of $I_r^Y$ follows from the corresponding statements for $J_{a, g'_1}^{R}$ and $ J_{b, g_1}^{E}$ from Lemma \ref{lem_toepl_tay_type} and the parity statement from Lemma \ref{lem_comp_poly}.
		This establishes the third statement from the first and the second parts of Theorem \ref{thm_ttype_closure2} by the last part of Theorem \ref{thm_ma_mar_crit_exp_dec}.
		\par 
		Let us now treat the fourth statement from the first part of Theorem \ref{thm_ttype_closure2}. 
		From (\ref{eq_defn_sigma}) and (\ref{eq_frame_tilde2}), for $Z \in \real^{2n}$, $Z'_W \in \real^{2l}$, we deduce that 
		\begin{multline}\label{eq_g3_expr_coord_change}
			T_{\llangle g_3 \rrangle, p}^{X|W} \big(\psi^{X|Y}_{y_0}(Z), \psi^{W|Y}_{y_0}(Z'_W) \big)
			=
			\exp \big(-p (\xi_L^{W|Y})^{*} - (\xi_F^{W|Y})^{*} \big)( \psi^{W|Y}_{y_0}(Z'_W))
			\cdot
			\\
			\cdot
			T_{\llangle g_3 \rrangle, p}^{X|W} \big(\psi^{X|W}_{y_0}(\upsilon(Z)), \phi^{W}_{y_0}(h^{W|Y}(Z'_W)) \big)
			\cdot
			\exp(- p \chi_L - \chi_F)(\psi^{X|Y}_{y_0}(Z))
			.
		\end{multline}
		\par 
		From (\ref{eq_h_defn_tr_m}), (\ref{eq_defn_kappaxy1}) and  (\ref{eq_defn_kappaxy2}), we deduce that 
		\begin{equation}\label{eq_kappa_det_jac}
			\kappa_{\psi}^{W|Y}(Z_W)
			=
			\kappa_{\phi}^{W}(h^{W|Y}(Z_W))
			\cdot
			(\det {\rm{Jac}} (h^{W|Y}))(Z_W).
		\end{equation}
		From (\ref{eq_h_defn_tr_m}) and (\ref{eq_defn_ups}), we deduce that
		\begin{equation}
			h^{X|W}(\upsilon(Z))
			=
			h^{X|Y}(Z).
		\end{equation}
		Hence by (\ref{eq_kappa_det_jac}), we obtain
		\begin{equation}
			\kappa_{\psi}^{X|W}(\upsilon(Z))
			=
			\kappa_{\psi}^{X|Y}(Z)
			\cdot
			\frac{(\det {\rm{Jac}} (h^{X|W}))(\upsilon(Z))}{(\det {\rm{Jac}} (h^{X|Y}))(Z)}.
		\end{equation}
		\par 
		\begin{sloppypar}
			For $r \in \nat$, we denote by $(\det {\rm{Jac}} (h^{X|W}))(\upsilon)_{[r]}^{\frac{1}{2}}$, $(\det {\rm{Jac}} (h^{X|Y}))^{-\frac{1}{2}}_{[r]}$, $(\det {\rm{Jac}} (h^{W|Y}))_{[r]}^{\frac{1}{2}}$ the homogeneous polynomials of degree $r$, defined as in (\ref{eq_kappa_N_tay_exp}) from Taylor expansions of $(\det {\rm{Jac}} (h^{X|W}))(\upsilon)^{\frac{1}{2}}$, $(\det {\rm{Jac}} (h^{X|Y}))^{-\frac{1}{2}}$, $(\det {\rm{Jac}} (h^{W|Y}))^{\frac{1}{2}}$.
			For $r \in \nat$, we now introduce
			\begin{multline}\label{eq_kappa_cor_2}
				\kappa_{{\rm{cor}}, [r]}^{2}(Z, Z'_W)
				:=
				\sum_{a + b + c = r}
				(\det {\rm{Jac}} (h^{X|W}))(\upsilon)_{[a]}^{\frac{1}{2}}(Z)
				\cdot
				(\det {\rm{Jac}} (h^{X|Y}))^{-\frac{1}{2}}_{[b]}(Z)
				\cdot
				\\
				\cdot
				(\det {\rm{Jac}} (h^{W|Y}))_{[c]}^{\frac{1}{2}}(Z'_W).
			\end{multline}
		\end{sloppypar}
		\par 
		We use notations, similar to (\ref{eq_tau_exp}) and (\ref{eq_ptau_exp}) for $\exp(- (\xi_F^{W|Y})^{*} )$, $\exp(- \chi_F)$ and $\exp(-p (\xi_L^{W|Y})^{*})$, $\exp(- p \chi_L)$.
		From  the analysis, similar to the one before (\ref{eq_jrf_expr}), (\ref{eq_resw_conn_g1}), (\ref{eq_tau_exp}) and (\ref{eq_g3_expr_coord_change}), we deduce that the asymptotic expansion (\ref{eq_berg_off_diag2}) holds for the operator $T_p^{X|Y} := T_{\llangle g_3 \rrangle, p}^{X|W} \circ T_{\llangle g_2 \rrangle, p}^{W|Y}$ and the polynomials
		\begin{multline}\label{eq_jre_rest_op2}
			I_r^{E} := \sum_{a + b + c + d + e + f + g + h + k = r}
			\mathcal{K}_{n, m}^{E}
			\Big[
				\exp(- (\xi_F^{W|Y})^{*} )_{[a]}(Z'_W)
				\cdot 
				\sum_{i - j = b} 
				\exp(-p (\xi_L^{W|Y})^{*})_{[i, j]}(Z'_W)
				\cdot
				\\
				\cdot
				J_{c, g_3}^{E}(\upsilon, h^{W|Y})_{[d]}
				\cdot
				\mathscr{E}'_{n, l}(\upsilon, h^{W|Y})_{[e]}
				\cdot
				\exp(- \chi_F)_{[f]}(Z)
				\cdot
				\\
				\cdot
				\sum_{i - j = g} 
				\exp(-p \chi_L)_{[i, j]}(Z)
				\cdot
				\kappa_{{\rm{cor}}, [h]}^{2}(Z, Z'_W)
				, 
				J_{k, g_2}^{E}
			\Big],
		\end{multline}
		where both sums run over a subset of natural numbers.
		From (\ref{eq_phi_fun_exp12}), similarly to the remark after (\ref{eq_ptau_exp}), we see that the second and the third sums in (\ref{eq_jre_rest_op2}) are actually finite.
		The statement about the parity of $I_r^{E}$ follows from the analogous statements for $J_{d, g_3}^{E}$, $J_{h, g_2}^{E}$ from Lemma \ref{lem_toepl_tay_type} and the parity statement from Lemma \ref{lem_comp_poly}.
		From this, we see that the fourth part of the first statement of Theorem \ref{thm_ma_mar_crit_exp_dec2} follows from Theorem \ref{thm_ma_mar_crit_exp_dec2}.
		From (\ref{eq_knme_form}), (\ref{eq_kmn_poly}), (\ref{eq_jre_rest_op2}) and the expressions for $J_{r, g_3}^{E}$, $J_{r, g_2}^{E}$, $r \in \nat$, from Lemma \ref{lem_toepl_tay_type}, we conclude that for $Z = (Z_Y, Z_N)$, $Z_Y \in \real^{2m}$, we have
		\begin{equation}
			I_0^{E}(Z, Z'_Y)
			:=
			\big( \iota_1^*(g_3) \cdot g_2 \big)(y_0) \cdot Z_N^{\otimes (k + k')}  \cdot \kappa_{N}^{X|Y}(y_0)^{\frac{1}{2}}, 
		\end{equation}
		which establishes the fourth statement from the second part of Theorem \ref{thm_ttype_closure2} by the last part of Theorem \ref{thm_ma_mar_crit_exp_dec2}.
		\par 
		The proofs of the fifth and sixth statements are completely analogous to the proof of the fourth one.
		The only difference is that the asymptotic expansion (\ref{eq_berg_off_diag2}) holds for the operators $T_p^{X|Y} := T_{\llangle g_3 \rrangle, p}^{X|W} \circ E_{p}^{W|Y}$, $T_p^{X|Y} := E_{p}^{X|W} \circ T_{\llangle g_2 \rrangle, p}^{W|Y}$ and the polynomials
		\begin{equation}\label{eq_jre_rest_op3}
		\begin{aligned}
			&
			I_r^{E} := \sum_{a + b + c + d + e + f + g + h + k = r}		
			\mathcal{K}_{n, m}^{E}
			\Big[
				\exp(- (\xi_F^{W|Y})^{*} )_{[a]}(Z'_W)
				\cdot
				\sum_{i - j = b} 
				\exp(-p (\xi_L^{W|Y})^{*})_{[i, j]}(Z'_W)
				\cdot 
				\\
				&
				\qquad  \qquad  \qquad  
				\cdot
				J_{c, g_3}^{E}(\upsilon, h^{W|Y})_{[d]}
				\cdot
				\mathscr{E}'_{n, l}(\upsilon, h^{W|Y})_{[e]}
				\cdot
				\exp(- \chi_F)_{[f]}(Z)
				\\
				&
				\qquad \qquad \qquad \qquad \qquad \qquad 
				\cdot
				\sum_{i - j = g} 
				\exp(-p \chi_L)_{[i, j]}(Z)
				\cdot
				\kappa_{{\rm{cor}}, [h]}^{2}(Z, Z'_W)
				, 
				J_{k}^{W|Y, E}
			\Big],
			\\
			&
			I_r^{E} := \sum_{a + b + c + d + e + f + g + h + k = r}		
			\mathcal{K}_{n, m}^{E}
			\Big[
				\exp(- (\xi_F^{W|Y})^{*} )_{[a]}(Z'_W) 
				\cdot
				\sum_{i - j = b} 
				\exp(-p (\xi_L^{W|Y})^{*})_{[i, j]}(Z'_W)
				\cdot 
				\\
				&
				\qquad  \qquad  \qquad  
				\cdot
				J_{c}^{X|W, E}(\upsilon, h^{W|Y})_{[d]}
				\cdot
				\mathscr{E}'_{n, l}(\upsilon, h^{W|Y})_{[e]}
				\cdot
				\exp(- \chi_F)_{[f]}(Z)
				\\
				&
				\qquad \qquad \qquad \qquad \qquad \qquad 
				\cdot
				\sum_{i - j = g} 
				\exp(-p \chi_L)_{[i, j]}(Z)
				\cdot
				\kappa_{{\rm{cor}}, [h]}^{2}(Z, Z'_W)
				, 
				J_{k, g_2}^{W|Y, E}
			\Big],
		\end{aligned}
		\end{equation}
		respectively.
		The proofs of the parity statements are analogous.
		From (\ref{eq_knme_form}), (\ref{eq_kmn_poly}), (\ref{eq_jre_rest_op2}) and the expressions for $J_{0}^{X|W, E}$, $J_{0}^{W|Y, E}$, $J_{0, g_3}^{E}$, $J_{0, g_2}^{E}$, from Theorem \ref{thm_ext_as_exp} and Lemma \ref{lem_toepl_tay_type}, we conclude that
		\begin{equation}
		\begin{aligned}
			&
			I_0^{E}(Z, Z'_Y)
			:=
			\iota_1^*(g_3)(y_0) \cdot Z_N^{\otimes k} \cdot \kappa_{N}^{X|Y}(y_0)^{\frac{1}{2}}, 
			\\
			&
			I_0^{E}(Z, Z'_Y)
			:=
			 g_2(y_0) \cdot Z_N^{\otimes  k'} \cdot \kappa_{N}^{X|Y}(y_0)^{\frac{1}{2}}, 
		\end{aligned}
		\end{equation}
		respectively.
		This establishes fifth and sixth statements from the second part of Theorem \ref{thm_ttype_closure2} by the remark in the end of Theorem \ref{thm_ma_mar_crit_exp_dec2}.
		\par 
		Let us now treat the seventh statement.
		Directly from Lemma \ref{prop_exp_bound_int}, (\ref{eq_comp_ext_pmn222}), the analysis, similar to the one before (\ref{eq_jrf_expr}), we conclude that (\ref{eq_berg_off_diag2}) holds for $T_p^{X|Y} := T_{\llangle g_1 \rrangle, p}^{X|Y} \circ T_{f, p}^{Y}$ and
		\begin{equation}\label{eq_sec_stat_j_r22}
			I_r^E
			:=
			\sum_{a + b = r}
			\mathcal{K}_{n, m}^{EP} \big[ J_{a, g_1}^{E}, J_{b, f}^{Y} \big].
		\end{equation}
		The parity statement $I_r^E$ holds by the same reasons as before.
		From (\ref{eq_kmn_poly}), (\ref{eq_sec_stat_j_r22}) and the expression for $J_{0, g_1}^{E}$, $J_{0, f}^{Y}$ from Lemma \ref{lem_toepl_tay_type}, we conclude that for $Z = (Z_Y, Z_N)$, we have
		\begin{equation}
			I_0^E(Z, Z'_Y)
			:=
			g_1(y_0) \cdot Z_N^{\otimes k} \cdot f \cdot \kappa_{N}^{X|Y}(y_0)^{\frac{1}{2}}, 
		\end{equation}
		which establishes the seventh statement from the second part of Theorem \ref{thm_ttype_closure2} by the last part of Theorem \ref{thm_ma_mar_crit_exp_dec}.
	\end{proof}

\subsection{Multiplicative defect and adjoints of Toeplitz type operators}\label{sect_adj_t_oper}
	The main goal of this section is to study the adjoints of Toeplitz type operators.
	For this, we introduce the so-called multiplicative defect operator and study some of its properties.
	The operator itself will be of fundamental importance to our calculations of the first significant term of the asymptotic expansion of the transitivity defect, $D_p$, from Theorem \ref{thm_trans}.
	\begin{thm}\label{thm_ex_ap}
		Assume that $(X, Y, g^{TX})$ is of bounded geometry.
		Then there is $p_1 \in \nat^*$, such that for any $p \geq p_1$, there is a unique operator $A_p^{X|Y} \in \enmr{H^{0}_{(2)}(Y, \iota^* ( L^p \otimes F))}$, verifying
		\begin{equation}\label{eq_resp_ap_lem}
			(\res_Y \circ B_p^X)^{*} = \ext_p^{X|Y} \circ A_p^{X|Y}.
		\end{equation}
		Moreover, the sequence of operators $\frac{1}{p^{n - m}} A_p^{X|Y}$, $p \geq p_1$, forms a Toeplitz operator with weak exponential decay with respect to $X$, and we have $[\frac{1}{p^{n - m}} A_p^{X|Y}]_0 = \kappa_N^{X|Y}|_Y^{-1}$, where $\kappa_N^{X|Y}$ was defined in (\ref{eq_kappan}). 
	\end{thm}
	\begin{rem}
		a) The sequence of operators $\frac{1}{p^{n - m}} A_p^{X|Y}$, will be later called “multiplicative defect".
		\par 
		b) This theorem can be used to give an alternative proof of the main results from \cite{FinOTAs} bypassing some of the technical difficulties, contained in \cite[\S 2.5, \S 4]{FinOTAs}.
	\end{rem}
	\begin{proof}
		First of all, let us establish the existence and uniqueness of $A_p^{X|Y}$ for $p$ big enough.
		Clearly, it suffices to prove that the kernels and the images of the operators $(\res_Y \circ B_p^X)^*$ and $\ext_p^{X|Y}$ coincide for $p$ big enough.
		First of all, we have
		\begin{equation}\label{eq_eq_ker_im_dual}
			\ker (\res_Y \circ B_p^X)^{*} = ( \Im (\res_Y \circ B_p^X) )^{\perp}.
		\end{equation}
		Now, in \cite[(4.1)]{FinOTAs}, we established that $\res_Y \circ B_p^X$ has its image inside of $H^0_{(2)}(Y, \iota^*(L^p \otimes F))$.
		In \cite[Theorem 4.4]{FinOTAs}, by following the proof of Ohsawa-Takegoshi extension theorem, we proved that there is $p_1 \in \nat$, such that for any $p \geq p_1$, the image of $\res_Y \circ B_p^X$ coincides exactly with $H^0_{(2)}(Y, \iota^*(L^p \otimes F))$.
		From this, and (\ref{eq_eq_ker_im_dual}), we see that the kernels of $(\res_Y \circ B_p^X)^{*}$ and $\ext_p^{X|Y}$ coincide.
		Similar reasoning shows that the images of those operators coincide as well.
		In particular, for $p \geq p_1$, there is a unique sequence of operators $A_p^{X|Y}$ as in (\ref{eq_resp_ap_lem}).
		\par 
		Now, let us establish that the sequence of operators $\frac{1}{p^{n - m}} A_p^{X|Y}$, $p \geq p_1$, forms a Toeplitz operator with weak exponential decay.
		We do so by applying Theorem \ref{thm_ma_mar_crit_exp_dec}.
		\par 
		In fact, from (\ref{eq_resp_ap_lem}) and the trivial fact $\res_Y \circ \ext_p^{X|Y} = B_p^Y$, we obtain the explicit formula
		\begin{equation}\label{eq_ap_form}
			A_p^{X|Y} = \res_Y \circ (\res_Y \circ B_p^X)^{*}.
		\end{equation}
		Clearly, the first property from Theorem \ref{thm_ma_mar_crit_exp_dec} follows from (\ref{eq_defn_res_map}) and (\ref{eq_ap_form}). 
		The weak version of the second property with respect to $X$ follows from Theorem \ref{thm_bk_off_diag} and (\ref{eq_ap_form}).
		\par 
		We will now show that the third property is a direct consequence of Theorem \ref{thm_berg_off_diag}.
		For the Taylor expansions of the $\kappa$-functions, we will use the same notation as in (\ref{eq_kappa_N_tay_exp}).
		From the fact that $\res_Y( \tilde{f}_1^{X|Y}, \ldots, \tilde{f}_r^{X|Y}) = \tilde{f}'_1{}^{Y}, \ldots, \tilde{f}'_r{}^{Y}$ and (\ref{eq_kappa_relation}), we see directly that the expansion (\ref{eq_tpy_defn_exp_tay12as}) holds for $T_p^Y := \frac{1}{p^{n - m}} A_p^{X|Y}$ and for the polynomials $I_r^Y(Z_Y, Z'_Y)$, $Z_Y, Z'_Y \in \real^{2m}$, defined as follows 
		\begin{equation}\label{eq_defn_jra}
			I_r^Y(Z_Y, Z'_Y)
			:=
			\sum_{a + b + c = r} J_{a}^{X|X}(Z_Y, Z'_Y) \cdot \kappa_{N, [b]}^{X|Y}(Z_Y)^{-\frac{1}{2}} \cdot \kappa_{N, [c]}^{X|Y}(Z'_Y)^{-\frac{1}{2}}.
		\end{equation}
		From the parity properties of $J_{a}^{X|X}$ from Theorem \ref{thm_berg_off_diag} and the bounded geometry assumption, we see that the coefficients of $I_r^Y$ are bounded with all their derivatives, and the parity of $I_r^Y$ coincides with $r$.
		Hence, by Theorem \ref{thm_ma_mar_crit_exp_dec}, the sequence of operators $\frac{1}{p^{n - m}} A_p^{X|Y}$ forms a Toeplitz operator with weak exponential decay with respect to $X$.
		Moreover, from (\ref{eq_jo_expl_form}) and (\ref{eq_defn_jra}), we deduce 
		\begin{equation}\label{eq_jo_expl_form12121}
			I_0^Y(Z_Y, Z'_Y) = \kappa_N^{X|Y}|_Y^{-1}.
		\end{equation}
		From the last statement of Theorem \ref{thm_ma_mar_crit_exp_dec}, we deduce that $[\frac{1}{p^{n - m}} A_p^{X|Y}]_0 = \kappa_N^{X|Y}|_Y^{-1}$.
	\end{proof}
	\par 
	For technical reasons, we will later need to consider the inverse of $\frac{1}{p^{n - m}} A_p^{X|Y}$. The following result gives a sufficient condition for inverting Toeplitz operators with weak exponential decay.
	\begin{lem}\label{lem_inverse_toepl}
		Assume that a sequence of operators $G_p$, $p \in \nat$, forms a Toeplitz operator with weak exponential decay with respect to a manifold $Z$ in the notations from Definition \ref{defn_ttype_weak}.
		Assume that for $f := [G_p]_0$, we have $f \neq 0$ everywhere and $f^{-1} \in \ccal^{\infty}_{b}(Y, \enmr{\iota^* F})$.
		Then there is $p_1 \in \nat$, such that for $p \geq p_1$, the operators $G_p$ are invertible.
		Moreover, the sequence of operators $G_p^{-1}$, $p \geq p_1$, forms a Toeplitz operator with weak exponential decay with respect to the same manifold $Z$ and we have $[(G_p)^{-1}]_0 = f^{-1}$.
	\end{lem}
	To prove this result, the following statement will be of utmost importance.
	\begin{lem}\label{lem_prod_exp_dec}
		For any $f_1, f_2 \in \ccal^{\infty}_{b}(Y, \enmr{\iota^* F})$, the sequence of operators $T_{f_1, p}^{Y} \circ T_{f_2, p}^{Y}$, $p \in \nat$, forms a Toeplitz operator with exponential decay.
		Moreover, we have $[T_{f_1, p}^{Y} \circ T_{f_2, p}^{Y}]_0 = f_1 \cdot f_2$.
		In particular, a product of two Toeplitz type operators with weak exponential decay forms a Toeplitz type operator with weak exponential decay.
	\end{lem}
	\begin{sloppypar}
	\begin{proof}
		The first part of this result for compact manifolds, in realms of Toeplitz operators in the sense of \cite[\S 7]{MaHol} (which is a slightly weaker notion), has first appeared in Bordemann–Meinrenken–Schlichenmaier \cite{BordMeinSchl}, cf. \cite[Theorem 3.1]{SchliBer}, for trivial $F$ and when the volume form $dv_Y$ coincides with the Riemannian volume form $dv_{g^{TY}}$.
		For nontrivial $F$ and other volume forms $dv_Y$, and in a more general setting of compact symplectic manifolds, this result was proved by Ma-Marinescu in \cite{MaMarToepl}, \cite[Theorem 7.4.1]{MaHol} by using the asymptotic characterization of Toeplitz operators as in Theorem \ref{thm_ma_mar_crit_exp_dec}.
		Since according to Theorem \ref{thm_ma_mar_crit_exp_dec}, the analogous characterization holds in our more refined setting of Toeplitz type operators with weak exponential decay, the same proof would give us the needed result.
	\end{proof}
	\end{sloppypar}
	\begin{proof}[Proof of Lemma \ref{lem_inverse_toepl}.]
		First of all, let us consider a sequence of operators $K_p := G_p \circ T_{f^{-1}, p}^{Y}$, $p \in \nat$.
		According to Lemma \ref{lem_prod_exp_dec}, $K_p$, $p \in \nat$, form a Toeplitz operator with weak exponential decay with respect to $Z$ and we can represent it in the form
		\begin{equation}\label{eq_dp_ep}
			K_p = 1 + \frac{Q_p}{p}, 
		\end{equation}
		where $Q_p$, $p \in \nat$, is a Toeplitz operator with weak exponential decay with respect to $Z$.
		In particular, by Corollary \ref{cor_norm_bnd_oper}, there are $C > 0$, $p_1 \in \nat^*$, such that for any $p \geq p_1$, we have
		\begin{equation}\label{eq_ep_norm22}
			\| Q_p \| \leq C.
		\end{equation}
		From (\ref{eq_dp_ep}) and (\ref{eq_ep_norm22}), we deduce that there is $p_1 \in \nat^*$, such that $K_p$ is invertible for $p \geq p_1$, and 
		\begin{equation}\label{eq_fp_in_exp}
			K_p^{-1} = \sum_{r = 0}^{\infty} (-1)^r \frac{Q_p^r}{p^r}. 
		\end{equation}
		However, by Corollary \ref{cor_comp_exp_bound}, we infer that there are $C > 0$, $p_1 \in \nat$, such that for any $p \geq p_1$, $r \in \nat^*$, we have
		\begin{equation}\label{eq_bnd_prod_ep}
			\big| Q_p^r(y_1, y_2) \big|_{\ccal^k} 
			\leq C^r  p^{m + \frac{k}{2}} \cdot \exp \big(- c \sqrt{p} \cdot \dist_X(y_1, y_2) \big).
		\end{equation}
		We conclude by Lemma \ref{lem_prod_exp_dec} and (\ref{eq_bnd_prod_ep}) that the sequence of operators $K_p^{-1}$, $p \geq p_1$, forms a Toeplitz type operator.
		But then again by Lemma \ref{lem_prod_exp_dec}, we obtain that the sequence of operators $T_{f^{-1}, p}^{Y} \circ K_p^{-1}$, $p \geq p_1$, forms a Toeplitz operator with exponential decay.
		But trivially, we have $G_p \circ T_{f^{-1}, p}^{Y} \circ K_p^{-1} = T_{f^{-1}, p}^{Y} \circ K_p^{-1} \circ G_p = {\rm{Id}}$.
		Hence, $G_p$ is invertible and $(G_p)^{-1} = T_{f^{-1}, p}^{Y} \circ K_p^{-1}$, which finishes the proof.
	\end{proof}
	As an important consequence of our considerations, we obtain the following result.
	\begin{thm}\label{thm_dual_ttype}
		A family $T_p^{Y|X} : L^2(X, L^p \otimes F) \to L^2(Y, \iota^*( L^p \otimes F))$, $p \in \nat$, of linear operators forms a Toeplitz operator with exponential decay of type $Y|X$ if and only if the family of linear operators $\frac{1}{p^{n - m}} (T_p^{Y|X})^* : L^2(Y, \iota^*( L^p \otimes F)) \to L^2(X, L^p \otimes F)$, $p \in \nat$, forms a Toeplitz operator with exponential decay of type $X|Y$.
		Moreover, we have 
		\begin{equation}\label{eq_rel_tp_xy_and_dual_zeroterm}
			[T_p^{Y|X}]_0 = \Big( \Big[\frac{1}{p^{n - m}} (T_p^{Y|X})^* \Big]_0 \Big)^{*} \cdot \kappa_N^{X|Y}|_Y.
		\end{equation}
	\end{thm}
	\begin{proof}
		Let us first assume that a sequence of operators $T_p^{Y|X}$, $p \in \nat$, forms a Toeplitz operator with exponential decay of type $Y|X$.
		Clearly, it is enough to prove that for any $k \in \nat$, $j = \{1, 2\}$, $g^a_j \in \ccal^{\infty}_{b}(Y, {\rm{Sym}}^{2k + j} N^{(0, 1)*} \otimes \enmr{\iota^* F})$, for $T_p^{Y|X} = T_{\llangle g^a_j \rrangle, p}^{Y|X}$, $j = 2$, and for $T_p^{Y|X} = \frac{1}{\sqrt{p}} T_{\llangle g^a_j \rrangle, p}^{Y|X}$, $j = 1$, the sequence of operators $\frac{1}{p^{n - m}} (T_p^{Y|X})^*$, $p \in \nat$, forms a Toeplitz operator with exponential decay of type $X|Y$.
		\par 
		According to Theorem \ref{thm_ex_ap} and (\ref{eq_toepl_fund_defned}), for $p \geq p_1$, where $p_1 \in \nat^*$ is as in Theorem \ref{thm_ex_ap}, we have
		\begin{equation}\label{eq_tgh_adjoint}
			(T_{\llangle g^a_j \rrangle, p}^{Y|X})^*
			=
			T_{\llangle (g^a_j)^* \rrangle, p}^{X|Y} \circ A_p^{X|Y}.
		\end{equation}
		Hence, according to Theorems \ref{thm_ttype_closure2}.6 and \ref{thm_ex_ap}, we see that $\frac{1}{p^{n - m}} (T_{\llangle g^a_j \rrangle, p}^{Y|X})^*$ forms a Toeplitz operator with exponential decay of type $X|Y$.
		The relation (\ref{eq_rel_tp_xy_and_dual_zeroterm}) follows from Theorems \ref{thm_ttype_closure2}.6 and \ref{thm_ex_ap}.
		This proves the first direction of Theorem \ref{thm_dual_ttype}.
		The proof of the opposite direction is completely analogous and is left to the interested reader.
	\end{proof}
	
	\begin{proof}[Proof of Theorem \ref{thm_ma_mar_crit_exp_dec3}.]
		The proof of one implication of Theorem \ref{thm_ma_mar_crit_exp_dec3} was described in the end of Section \ref{sect_as_crit}.
		The inverse implication is a direct consequence of Theorems \ref{thm_ma_mar_crit_exp_dec2} and \ref{thm_dual_ttype}.
	\end{proof}

\subsection{Some examples of Toeplitz type operators}\label{sect_basic_exmpl}
	
	The main goal of this section is to give some examples of Toeplitz type operators.
	To state our results in this direction, we need to fix some notation first.
	\par
	We fix $y_0 \in Y$, choose an orthogonal basis $w_1, \ldots, w_{n - m}$ of $(N_{y_0}^{X|Y})^{(1, 0)}$ as in (\ref{eq_lambda_eq_defined}).
	We define 
	\begin{equation}
	\begin{aligned}
		&
		\Lambda_{\omega, h} : {\rm{Sym}}^{i} ((N^{X|Y})^{(1, 0)*}) \otimes  {\rm{Sym}}^{j}((N^{X|Y})^{(0, 1)*}) \to {\rm{Sym}}^{\max\{i - j, 0 \}}((N^{X|Y})^{(1, 0)*}),
		\\
		&
		\Lambda_{\omega, a} : {\rm{Sym}}^{i} ((N^{X|Y})^{(1, 0)*}) \otimes  {\rm{Sym}}^{j}((N^{X|Y})^{(0, 1)*}) \to {\rm{Sym}}^{\max\{j - i, 0 \}}((N^{X|Y})^{(0, 1)*}),
	\end{aligned}
	\end{equation}
	for multiindices $\alpha, \beta \in \nat^{n - m}$, as follows
	\begin{equation}
	\begin{aligned}
		&
		\Lambda_{\omega, h} ( w^{\alpha} \otimes \overline{w}^{\beta} )
		=
		\begin{cases}
			\frac{1}{\pi^{|\alpha|}} \frac{\alpha!}{(\alpha - \beta)!} w^{\alpha - \beta},  & \text{if } \alpha \geq \beta, \alpha \neq \beta,
			\\
			0,  & \text{otherwise},
		\end{cases}
		\\
		&
		\Lambda_{\omega, a} ( w^{\alpha} \otimes \overline{w}^{\beta} )
		=
		\begin{cases}
			\frac{1}{\pi^{|\beta|}} \frac{\beta!}{(\beta - \alpha)!} \overline{w}^{\beta - \alpha},  & \text{if } \alpha \geq \beta, \alpha \neq \beta,
			\\
			0,  & \text{otherwise}.
		\end{cases}
	\end{aligned}
	\end{equation}
	Clearly, those operators do not depend on the choice of the basis.
	We extend $\Lambda_{\omega, h}[ \cdot ]$ and $\Lambda_{\omega, a}[ \cdot ]$ to ${\rm{Sym}}^{k} (N^{X|Y})^{*} \otimes \comp$ linearly.
	For the next result, we will use the following notation.
	For $f \in \ccal^{\infty}_{b}(X, \enmr{F})$, we let $T_{f, p}^{Y|Y} := \res_Y \circ T_{f, p}^X \circ \ext_p^{X|Y}$.
	\begin{prop}\label{prop_toepl_type_poly_suff}
		For any $f \in \ccal^{\infty}_{b}(X, \enmr{F})$, $g \in \oplus_{k = 0}^{\infty} \ccal^{\infty}_{b}(Y, {\rm{Sym}}^{k} (N^{X|Y})^{*} \otimes \enmr{\iota^* F})$, $g_e \in \oplus_{k = 0}^{\infty} \ccal^{\infty}_{b}(Y, {\rm{Sym}}^{2k} (N^{X|Y})^{*} \otimes \enmr{\iota^* F})$, $g_o \in \oplus_{k = 0}^{\infty} \ccal^{\infty}_{b}(Y, {\rm{Sym}}^{2k + 1} (N^{X|Y})^{*} \otimes \enmr{\iota^* F})$, the sequences of operators 
		\begin{tasks}[style=enumerate](4)
			\task $T_{f, p}^{Y|Y}$, 
			\task $T_{\llangle g \rrangle, p}^{Y|Y}$,
		\end{tasks}
		$p \in \nat$, form Toeplitz operators with exponential decay.
		Also, the sequences of operators
		\begin{tasks}[style=enumerate, resume](4)
			\task $T_{f, p}^{X|Y}$, 
			\task $T_{\llangle g_e \rrangle, p}^{X|Y}$, 
			\task $\frac{1}{\sqrt{p}} T_{\llangle g_o \rrangle, p}^{X|Y}$,
		\end{tasks}
		$p \in \nat$, form Toeplitz operators with exponential decay of type $X|Y$.
		The sequences of operators
        \begin{tasks}[style=enumerate, resume](4)
			\task $T_{f, p}^{Y|X}$, 
			\task $T_{\llangle g_e \rrangle, p}^{Y|X}$, 
			\task $\frac{1}{\sqrt{p}} T_{\llangle g_o \rrangle, p}^{Y|X}$,
		\end{tasks}
		$p \in \nat$, form Toeplitz operators with exponential decay of type $Y|X$.
		Moreover, we have 
        \begin{tasks}[style=enumerate](3)
			\task $[T_{f, p}^{Y|Y}]_0 = f$, 
			\task $[T_{f, p}^{X|Y}]_0 = 0$, 
			\task $[T_{\llangle g \rrangle, p}^{Y|Y}]_0 = \Lambda_{\omega, =} [ g ]$,
			\task $[T_{\llangle g_e \rrangle, p}^{X|Y}]_0 = \Lambda_{\omega, h} [ g_e ]$,
			\task $[\sqrt{p} T_{\llangle g_o \rrangle, p}^{X|Y}]_0 = 0$,
			\task $[T_{\llangle g_o \rrangle, p}^{X|Y}]_1 = \Lambda_{\omega, h} [ g_o ]$,
			\task $[T_{f, p}^{Y|X}]_0 = 0$,
			\task $[T_{\llangle g_e \rrangle, p}^{Y|X}]_0 = \Lambda_{\omega, a} [ g_e ]$,
			\task $[\frac{1}{\sqrt{p}} T_{\llangle g_o \rrangle, p}^{Y|X}]_0 = 0$,
			\task $[\frac{1}{\sqrt{p}} T_{\llangle g_o \rrangle, p}^{Y|X}]_1 = \Lambda_{\omega, a} [ g_o ]$.
		\end{tasks}
	\end{prop}
	\begin{proof}
		The proofs of all the statements from the first part of Proposition \ref{prop_toepl_type_poly_suff} are very similar to the proofs from Theorem \ref{thm_ttype_closure2}: they all proceed by the verification that the relevant operators satisfy the assumptions of Theorems \ref{thm_ma_mar_crit_exp_dec}, \ref{thm_ma_mar_crit_exp_dec2} and  \ref{thm_ma_mar_crit_exp_dec3}.
		The proofs of the second part are also analogous: we only need to calculate the first term of the asymptotic expansions as in (\ref{eq_tpy_defn_exp_tay12as}), (\ref{eq_berg_off_diag2}), (\ref{eq_berg_off_diag3}) and apply the last part of Theorems \ref{thm_ma_mar_crit_exp_dec}, \ref{thm_ma_mar_crit_exp_dec2}, \ref{thm_ma_mar_crit_exp_dec3}. 
		For brevity, we only present the proof for the fourth statement, which is slightly more complicated than the rest.
		\par 
		The validity of the first condition from Theorem \ref{thm_ma_mar_crit_exp_dec2} for $T_{\llangle g_e \rrangle, p}^{X|Y}$ is direct.
		The second and the third conditions are proved in Lemma \ref{lem_toepl_tay_type}.
		Hence, by Theorem \ref{thm_ma_mar_crit_exp_dec2}, the sequence of operators $T_{\llangle g_e \rrangle, p}^{X|Y}$, $p \in \nat$, form a Toeplitz operator with exponential decay of type $X|Y$.
		We now only need to calculate the first term of the asymptotic expansion of this sequence of operators to establish the second part of the theorem.
		\par 
		We decompose $g_e$ as follows
		\begin{equation}\label{eq_fig_g_dec12}
			g_e = \sum_{i, j} g_{e, i j},
		\end{equation}
		where $i, j \in \nat$ and $g_{e, i j} \in {\rm{Sym}}^{i} (N^{X|Y})^{(1, 0)*} \otimes {\rm{Sym}}^{j} (N^{X|Y})^{(0, 1)*}$.
		From (\ref{eq_kmn_poly}), (\ref{eq_kmn_poly23}) and (\ref{eq_jrg_e_form}), we deduce that for any $Z = (Z_Y, Z_N)$, $Z_N \in \real^{2 (n - m)}$; $Z_Y, Z'_Y \in \real^{2 m}$, we have
		\begin{equation}
			J_{0, g_e}^{E}(Z, Z'_Y)
			=
			\sum_{i, j} \Lambda_{\omega, h} \big[ g_{e, i j} \big] 
			\cdot 
			Z_N^{\otimes (i - j)}
			\cdot
			\kappa_N^{X|Y}(y_0)^{\frac{1}{2}}
			.
		\end{equation}
		From this and the last part of Theorem \ref{thm_ma_mar_crit_exp_dec2}, we conclude that the second part of Proposition \ref{prop_toepl_type_poly_suff} for the fourth point holds.
	\end{proof}

\section{Complex embeddings and associated Toeplitz type operators}\label{sect_asymp_trans}

	The main goal of this section is to establish Theorems \ref{thm_as_trans}, \ref{thm_trans} and \ref{thm_as_ext_res}.
	More precisely, in Section \ref{sect_ct_asymp}, we calculate the second term of the asymptotic expansion of the multiplicative defect introduced in Section \ref{sect_adj_t_oper} and, as a consequence, we prove Theorem \ref{thm_as_ext_res}.
	In Section \ref{sec_as_tr_pf_mn_thm}, we establish Theorems \ref{thm_as_trans}, \ref{thm_trans} and the extension of Theorem \ref{thm_trans} to towers of submanifolds of arbitrary length.

\subsection{Optimal Ohsawa-Takegoshi theorem, multiplicative defect asymptotics}\label{sect_ct_asymp}
	The main goal of this section is to calculate the second term of the asymptotic expansion of the multiplicative defect and to calculate the asymptotics of the optimal constant in Ohsawa-Takegoshi theorem, i.e. to establish Theorem \ref{thm_as_ext_res}.

	\begin{thm}\label{thm_ap_exp_two_terms}
 		In the notations of Theorems \ref{thm_as_trans}, \ref{thm_ex_ap}, under assumption (\ref{eq_comp_vol_omeg}), we have
 		\begin{equation}
 			\Big[ \frac{1}{p^{n - m}} A_p^{X|Y} \Big]_0 = 1, 
 			\qquad
 			\Big[ \frac{1}{p^{n - m}} A_p^{X|Y} \Big]_1 = \frac{1}{8 \pi} \Big( 
 			\textbf{r}^{X} - \textbf{r}^{Y} \Big)
			-
			\frac{1}{2 \pi \imun} \Big( 
				\Lambda_{\omega} [ R^F ]
				-
				\Lambda_{\iota^* \omega} [ R^F ]
			\Big).
 		\end{equation}
 	\end{thm}

	The proof of Theorem \ref{thm_ap_exp_two_terms} will be based on the following result.
	
	\begin{thm}\label{thm_berg_dailiuma_sec}
		In the notations of Theorem \ref{thm_berg_off_diag}, under assumptions (\ref{eq_comp_vol_omeg}), we have
		\begin{equation}
			J_2^{X|X}(0, 0) =  
			\frac{1}{8 \pi} \textbf{r}_{x_0}^{X}
			-
			\frac{1}{2 \pi \imun} \Lambda_{\omega} [ R^F_{x_0} ].
		\end{equation}
	\end{thm}
	\begin{proof}
		The proof is due to Lu \cite{LuBergman} (for trivial $(F, h^F)$) and Wang \cite{WangBergmKern} (for general $(F, h^F)$), cf. also Dai-Liu-Ma \cite[Theorem 1.3]{MaHol}.
	\end{proof}
	Recall that in Lemma \ref{lem_toepl_tay_type}, for any $f \in \ccal^{\infty}_{b}(X, \enmr{F})$, $x_0 \in X$, $r \in \nat$, we defined the polynomials $J_{r, f}^X(Z, Z') \in \enmr{F_{x_0}}$, $Z, Z' \in \real^{2n}$.
	\begin{cor}\label{cor_first_term_toepl}
		Under the assumptions (\ref{eq_comp_vol_omeg}), we have
		\begin{equation}
			J_{1, f}^{X}(Z, Z')
			=
			\nabla^{\enmr{F}}_{\frac{\partial}{\partial z}} f + \frac{\partial f}{\partial \overline{z}'}.
		\end{equation}
	\end{cor}
	\begin{proof}
		It follows directly from (\ref{eq_kmn_poly}), (\ref{eq_jo_expl_form}), (\ref{eq_jrf_expr}), (\ref{eq_j1_expl_form}) and the fact, following from Proposition \ref{prop_phi_fun_exp}, that a derivative of a sections of a vector bundle, written in the trivialization, considered in Theorem \ref{thm_berg_off_diag}, correspond to covariant derivatives.
	\end{proof}
	
	\begin{proof}[Proof of Theorem \ref{thm_ap_exp_two_terms}]
 		The first identity is a direct consequence of  Theorem \ref{thm_ex_ap} and our assumption, see the remark before (\ref{eq_comp_vol_omeg}).
 		\par 
 		To establish the second identity, remark that from the first part and Theorem \ref{thm_ex_ap}, the sequence of operators $p (\frac{1}{p^{n - m}} A_p^{X|Y} - B_p^Y)$, $p \geq p_1$, forms a Toeplitz operator with weak exponential decay with respect to $X$.
 		Moreover, from (\ref{eq_j1_expl_form}) and (\ref{eq_defn_jra}), we see that the expansion (\ref{eq_tpy_defn_exp_tay12as}) holds for $T_p^Y := p (\frac{1}{p^{n - m}} A_p^{X|Y} - B_p^Y)$ and for polynomials $I_r^Y(Z_Y, Z'_Y)$, $Z_Y, Z'_Y \in \real^{2m}$, verifying
 		\begin{equation}
 			I_0^Y(0, 0)
 			=
 			J_2^{X|X}(0, 0)
 			-
 			J_2^{Y|Y}(0, 0).
 		\end{equation}
 		From Theorem \ref{thm_berg_dailiuma_sec}, we obtain that
 		\begin{equation}\label{eq_j_2a_need_bnd}
 			I_0^Y(0, 0) = 
 			\frac{1}{8 \pi} \Big( \textbf{r}_{y_0}^{X} - \textbf{r}_{y_0}^{Y} \Big)
			-
			\frac{1}{2\pi \imun} \Big( 
				\Lambda_{\omega} [R^F_{y_0}]
				-
				\Lambda_{\iota^* \omega} [R^F_{y_0}]
			\Big).
 		\end{equation}
		From the last part of Theorem \ref{thm_ma_mar_crit_exp_dec} and (\ref{eq_j_2a_need_bnd}), we obtain the needed result.
 	\end{proof}
	Let us now give the first application of those calculations.
	\begin{proof}[Proof of Theorem \ref{thm_as_ext_res}]
		From (\ref{eq_resp_ap_lem}), remark that the following identities hold
		\begin{equation}\label{eq_comp_ext_oper}
			(\ext_p^{X|Y})^* \circ \ext_p^{X|Y} = \big( (A_p^{X|Y})^* \big)^{-1},
			\qquad
			\res_p^{Y|X} \circ (\res_p^{Y|X})^{*} = A_p^{X|Y}.
		\end{equation}
		Clearly, we have $ \| (\ext_p^{X|Y})^* \circ \ext_p^{X|Y} \| = \| \ext_p^{X|Y} \|^2$ and $ \| \res_p^{Y|X} \circ (\res_p^{Y|X})^{*} \| = \|  \res_p^{Y|X} \|^2$.
		The result now follows from this observation, Theorem \ref{thm_ap_exp_two_terms}, Lemma \ref{lem_norm_toepl} and (\ref{eq_comp_ext_oper}).
	\end{proof}
	\begin{rem}\label{rem_ap_self_adj}
		From (\ref{eq_comp_ext_oper}), we see that $A_p^{X|Y}$ is a self-adjoint operator.
	\end{rem}

 \subsection{Transitivity defect, proofs of Theorems \ref{thm_as_trans}, \ref{thm_trans}}\label{sec_as_tr_pf_mn_thm}

	The main goal of this section is to study the asymptotic transitivity of the optimal holomorphic extension operator and to prove Theorems \ref{thm_as_trans} and \ref{thm_trans}.
	One way of proceeding would be to directly use the formula (\ref{eq_jre_rest_op112}) to calculate the asymptotics of the sequence of operators 
	\begin{equation}\label{eq_defn_final_seq_of_operar}
		T_p^{W|Y} := \res_W \circ \ext_p^{X|Y} - \ext_p^{W|Y},
	\end{equation}
	$p \geq p_1$, where $p_1 \in \nat$ is as in (\ref{eq_ext_op}) and study the first non vanishing term of this asymptotics.
	Then, we will get the needed result by the use of the basic formula, cf. (\ref{eq_res_ext_rel_dp_oper}),
	\begin{equation}\label{eq_dp_tp_relat}
		\ext_p^{X|W} \circ T_p^{W|Y} = D_p = \ext_p^{X|Y} - \ext_p^{X|W} \circ \ext_p^{W|Y},
	\end{equation}
	and the subsequent use of the formula (\ref{eq_jre_rest_op3}).
	This method is, although straightforward, computationally complicated.
	In fact, to calculate $[\ext_p^{X|Y} - \ext_p^{X|W} \circ \ext_p^{W|Y}]_3$ (which is the first significant term of our asymptotic expansion according to Theorem \ref{thm_trans}), we will need to calculate $J_3^{X|X}$, and the third terms of the Taylor expansions of $\tau_E$, $p \tau_L$, $\sigma$, $\upsilon$, etc. 
	Working directly with the asymptotic expansion of the sequence of operators $D_p$ is also possible due to the results from Section \ref{sect_model_calc}, but computationally is even more difficult.
	\par 
	Our approach is different.
	We will still, however, base our consideration on the study of the sequence of operators $T_p^{W|Y}$, $p \in \nat$, instead of $D_p$.
	But differently from the above approach, instead of using right away the explicit formula for the asymptotic expansion, we will first find an alternative expression for $T_p^{W|Y}$ in terms of the operators $A_p^{X|Y}$, $A_p^{X|W}$ and $A_p^{W|Y}$.
	Then the calculation of the asymptotic expansion for $T_p^{W|Y}$ will be essentially encapsulated in the calculations of the asymptotic expansions of $A_p^{X|Y}$, $A_p^{X|W}$ and $A_p^{W|Y}$.
	More precisely, our first result goes as follows.
	\begin{lem}\label{lem_def_alt_expr}
		There is $p_1 \in \nat$, such that for any $p \geq p_1$, the following expression for $T_p^{W|Y}$ holds
		\begin{equation}\label{eq_main_form_defect}
			\begin{aligned}
			T_p^{W|Y}
			=
			&\,
			\res_W \circ \ext_p^{X|Y} 
			\circ
			\Big[ B_p^{Y} -  \Big( \frac{1}{p^{n - m}} A_p^{X|Y} \Big) \circ \Big(  \frac{1}{p^{l - m}} A_p^{W|Y} \Big)^{-1} \Big]
			\\
			&+
			\Big[ B_p^{W} - \Big( \frac{1}{p^{n - l}} (A_p^{X|W})^* \Big)^{-1} \Big]
			\circ
			\res_W \circ \ext_p^{X|Y} 
			\\
			&
			-
			\Big[ B_p^{W} - \Big( \frac{1}{p^{n - l}} (A_p^{X|W})^* \Big)^{-1} \Big]
			\circ
			\res_W \circ \ext_p^{X|Y} 
			\circ
			\\
			&
			\qquad \qquad \qquad \qquad \qquad
			\circ
			\Big[ B_p^{Y} -  \Big( \frac{1}{p^{n - m}} A_p^{X|Y} \Big) \circ \Big(  \frac{1}{p^{l - m}} A_p^{W|Y} \Big)^{-1} \Big]
			.
			\end{aligned}
		\end{equation}
	\end{lem}
	\begin{proof}
		First of all, recall that by Theorem \ref{thm_ex_ap} and Lemma \ref{lem_inverse_toepl}, there is $p_1 \in \nat$, such that for $p \geq p_1$, the operators $\frac{1}{p^{n - l}} A_p^{X|W}$, $\frac{1}{p^{n - m}} A_p^{X|Y}$, $\frac{1}{p^{l - m}} A_p^{W|Y}$ are invertible.
		In what follows, we work with such $p$ with no further notice.
		From (\ref{eq_resp_ap_lem}), we have
		\begin{equation}\label{eq_aux_fin_lem1}
			\ext_p^{W|Y} =  (\res_Y \circ B_p^W)^{*} \circ (A_p^{W|Y})^{-1}.
		\end{equation}
		Now, from the trivial fact that $B_p^W = \res_W \circ \ext_p^{X|W}$ and (\ref{eq_aux_fin_lem1}) for $W := X$, $Y := W$, we obtain
		\begin{equation}\label{eq_aux_fin_lem2}
			B_p^W = \res_W \circ (\res_W \circ B_p^X)^{*} \circ (A_p^{X|W})^{-1}.
		\end{equation}
		From (\ref{eq_aux_fin_lem1}), (\ref{eq_aux_fin_lem2}) and the trivial fact that $\res_Y \circ \res_W = \res_Y$, we obtain
		\begin{equation}\label{eq_aux_fin_lem3}
			\ext_p^{W|Y} = \big( \res_Y \circ (\res_W \circ B_p^X)^{*} \circ (A_p^{X|W})^{-1} \big)^{*} \circ (A_p^{W|Y})^{-1}.
		\end{equation}
		We replace $\res_Y$ in (\ref{eq_aux_fin_lem3}) by $\res_Y \circ B_p^X$, open the brackets in (\ref{eq_aux_fin_lem3}), use once again (\ref{eq_resp_ap_lem}) to give an alternative expression for $(\res_Y \circ B_p^X)^*$ and use the trivial fact $B_p^X \circ \ext_p^{X|Y} =  \ext_p^{X|Y}$, to obtain
		\begin{equation}\label{eq_aux_fin_lem4}
			\ext_p^{W|Y} 
			=
			\Big( \frac{1}{p^{n - l}} (A_p^{X|W})^* \Big)^{-1}
			\circ
			\res_W 
			\circ 
			\ext_p^{X|Y}
			\circ
			\Big(
			\frac{1}{p^{n - m}}
			A_p^{X|Y} 
			\Big)
			\circ 
			\Big(
			\frac{1}{p^{l - m}} A_p^{W|Y}
			\Big)^{-1}.
		\end{equation}
		The formula (\ref{eq_main_form_defect}) is then a formal consequence of (\ref{eq_aux_fin_lem4}) and the fact that $\res_W \circ \ext_p^{X|Y} = B_p^W \circ \res_W \circ \ext_p^{X|Y} \circ B_p^Y$, following from (\ref{eq_defn_res_map}).
	\end{proof}
	
	To establish Theorem \ref{thm_trans}, we need two additional lemmas.
	To state the first, let us fix a function $f \in \ccal^{\infty}_{b}(X, \enmr{F})$ and consider the Toeplitz operator $T_{f, p}^{X}$, $p \in \nat$.
	We will use below the notational conventions introduced before Theorem \ref{thm_berg_off_diag}.
	\begin{lem}\label{lem_toepl_tay_typexynew}
		There are polynomials $J_{0, f}^{X|Y}(Z, Z')$, $J_{1, f}^{X|Y}(Z, Z')$ in $Z, Z' \in \real^{2n}$, such that for $F_{r, f}^{X|Y} := J_{r, f}^{X|Y} \cdot \mathscr{P}_{n}$, $r = 0, 1$, the following holds.
		There are $\epsilon, c, C, Q > 0$, $p_1 \in \nat^*$, such that for any $y_0 \in Y$, $p \geq p_1$, $|Z|, |Z'| \leq \epsilon$, the Schwartz kernels of $T_{f, p}^{X}$, evaluated with respect to the volume form $dv_X$, satisfies
			\begin{multline}\label{eq_tpy_defn_exp_tayxynew}
				\bigg| 
						\frac{1}{p^n} T_{f, p}^{X} \big(\psi^{X|Y}_{y_0}(Z), \psi^{X|Y}_{y_0}(Z') \big)
						\\
						-
						\sum_{r = 0}^{1}
						p^{-\frac{r}{2}}						
						F_{r, f}^{X|Y}(\sqrt{p} Z, \sqrt{p} Z') 
						\kappa_{\psi}^{X|Y}(Z)^{-\frac{1}{2}}
						\kappa_{\psi}^{X|Y}(Z')^{-\frac{1}{2}}
				\bigg|
				\\
				\leq
				C p^{- 1}
				\Big(1 + \sqrt{p}|Z| + \sqrt{p} |Z'| \Big)^{Q} \exp(- c \sqrt{p} |Z - Z'|).
				\end{multline}
		Moreover, we have $J_{0, f}^{X|Y}(Z_Y, Z'_Y) = f(y_0)$ and for $Z = (0, Z_N)$, $Z_N \in \real^{2(n - m)}$, we have
		\begin{equation}
			J_{1, f}^{X|Y}(Z, 0)
			=
			\nabla^{\enmr{E}}_{\frac{\partial}{\partial z}} f.
		\end{equation}
	\end{lem}
	\begin{proof}
		First of all, recall that the diffeomorphism $h^{X|Y}$ was defined in (\ref{eq_h_defn_tr_m}), and the functions $\xi_L^{X|Y}$, $\xi_F^{X|Y}$ were defined in (\ref{eq_frame_tilde}).
		Directly from the definitions, we obtain the following relation between the Schwartz kernels
		\begin{multline}\label{eq_rel_kernels_typexynew}
			T_{f, p}^{X} \big(\psi^{X|Y}_{y_0}(Z), \psi^{X|Y}_{y_0}(Z') \big)
			=
			\exp \big(-p (\xi_L^{X|Y})^{*} - (\xi_F^{X|Y})^{*} \big)(\psi_{y_0}^{X|Y}(Z'))
			\cdot
			\\
			\cdot 
			T_{f, p}^{X} \big(\phi^{X}_{y_0}(h^{X|Y}(Z)), \phi^{X}_{y_0}(h^{X|Y}(Z')) \big)
			\cdot
			\exp \big(-p \xi_L^{X|Y} - \xi_F^{X|Y} \big)(\psi_{y_0}^{X|Y}(Z)).
		\end{multline}
		Remark also that in the notations of (\ref{eq_mn_curv_d}), (\ref{eq_defn_kappaxy1}), (\ref{eq_defn_kappaxy2}), by \cite[(3.26)]{MaMarBTKah} and \cite[(5.35)]{FinOTAs}, we have
		\begin{equation}\label{eq_kappa_exp_first}
			\kappa_{\phi, y_0}^{X}(Z) = 1 + O(|Z|^2), 
			\qquad 
			\kappa_{\psi, y_0}^{X|Y}(Z) = 1 - g_{y_0}^{TX} (\nu^{X|Y}, Z) + O(|Z|^2).
		\end{equation}
		By Proposition \ref{prop_prop_sfndform}.5 and (\ref{eq_kappa_exp_first}), we deduce that 
		\begin{equation}\label{eq_kappa_exp_sec}
			\frac{\kappa_{\phi, y_0}^{X}}{\kappa_{\psi, y_0}^{X|Y}} = 1 + O(|Z|^2).
		\end{equation}
		From  Lemma \ref{lem_toepl_tay_type}, Corollary \ref{cor_first_term_toepl}, (\ref{eq_rel_kernels_typexynew}), (\ref{eq_kappa_exp_sec}) and the trivial fact that for $Z = (0, Z_N)$, $Z_N \in \real^{2(n - m)}$, we have $\xi_L^{X|Y}(Z) = \xi_E^{X|Y}(Z) = 0$, $h^{X|Y}(Z) = Z$, we deduce the result.
	\end{proof}
	\begin{lem}\label{eq_ex_res_epx_2terms}
		There are polynomials $J_{0, Res}^{W|Y}(Z_W, Z'_Y)$, $J_{1, Res}^{W|Y}(Z_W, Z'_Y)$ in $Z_W \in \real^{2l}$, $Z'_Y \in \real^{2m}$, such that for $F_{r, Res}^{W|Y} := J_{r, Res}^{W|Y} \cdot \mathscr{E}_{l, m}$, $r = 0, 1$, the following holds.
		There are $\epsilon, c, C, Q > 0$, $p_1 \in \nat^*$, such that for any $y_0 \in Y$, $p \geq p_1$, $Z_W = (Z_Y, Z_{N^{W|Y}})$, $Z_Y \in \real^{2m}$, $|Z_W|, |Z'_Y| \leq \epsilon$, the Schwartz kernel of $\res_W \circ \ext_p^{X|Y}$, evaluated with respect to $dv_Y$, satisfies the following bound
			\begin{multline}\label{eq_tpy_defn_exp_tayxynew11212}
				\bigg| 
						\frac{1}{p^m} \res_W \circ \ext_p^{X|Y}  \big(\psi^{W|Y}_{y_0}(Z_W), \phi^{Y}_{y_0}(Z'_Y) \big)
						\\
						-
						\sum_{r = 0}^{1}
						p^{-\frac{r}{2}}						
						F_{r, Res}^{W|Y}(\sqrt{p} Z_W, \sqrt{p} Z'_Y) 
						\kappa_{\psi}^{W|Y}(Z_W)^{-\frac{1}{2}}
						\kappa_{\phi}^{W}(Z'_Y)^{-\frac{1}{2}}
				\bigg|
				\\
				\leq
				C p^{- 1}
				\Big(1 + \sqrt{p}|Z_W| + \sqrt{p} |Z'_Y| \Big)^{Q} \exp \big( - c \sqrt{p} ( |Z_Y - Z'_Y| + |Z_{N^{W|Y}}| ) \big).
				\end{multline}
		Moreover, we have $J_{0, Res}^{W|Y}(Z_W, Z'_Y) = 1$, and for $Z_W = (0, Z_{N^{W|Y}})$, $Z_{N^{W|Y}} \in \real^{2(l - m)}$, in the notations of Lemma \ref{lem_comp_poly}, we have
		\begin{equation}\label{eq_jre_rest_op112b}
			\mathcal{K}_{n, n}[1, J_{1, Res}^{W|Y}](Z_W, 0) 
			=
			J_{1, Res}^{W|Y}(Z_W, 0),
			\qquad
			\mathcal{K}_{l, m}^{EP}[ J_{1, Res}^{W|Y}, 1]
			=
			J_{1, Res}^{W|Y}.
		\end{equation}
	\end{lem}
	\begin{proof}
		The existence of polynomials was proved in (\ref{eq_jre_rest_op112}). The calculation of $J_{0, Res}^{W|Y}$ was included in (\ref{eq_jre_rest_op112a}).
		Now, to prove (\ref{eq_jre_rest_op112b}), we first remark that 
		\begin{equation}\label{eq_fin_triv_imp}
			B_p^W \circ \res_W \circ \ext_p^{X|Y} = \res_W \circ \ext_p^{X|Y}, \qquad \res_W \circ \ext_p^{X|Y} = \res_W \circ \ext_p^{X|Y} \circ B_p^Y.
		\end{equation}
		Comparing the first order asymptotics of each side of (\ref{eq_fin_triv_imp}), using Lemma \ref{lem_comp_poly} and the analysis, similar to the one before (\ref{eq_jrf_expr}), gives
		\begin{equation}
		\begin{aligned}
			&
			\mathcal{K}_{n, n}[1, J_{1, Res}^{W|Y}](Z_W, 0)
			+
			\mathcal{K}_{n, n}[J_1^{W|Y}, 1](Z_W, 0)
			=
			 J_{1, Res}^{W|Y}(Z_W, 0),
			 \\
			 &
			  \mathcal{K}_{l, m}^{EP}[ J_{1, Res}^{W|Y}, 1]
			 +
			 \mathcal{K}_{l, m}^{EP}[1, J_1^{Y|Y}]
			 =
			J_{1, Res}^{W|Y}.
		\end{aligned}
		\end{equation}
		However, an easy calculation, using (\ref{eq_j1_expl_form}), shows that 
		\begin{equation}
			\mathcal{K}_{l, m}^{EP}[J_1^{W|Y}, 1](Z_W, 0)
			=
			0,
			\qquad
			\mathcal{K}_{l, m}^{EP}[1, J_1^{Y|Y}]
			=
			0,
		\end{equation}
		which obviously finishes the proof.
	\end{proof}
	\begin{proof}[Proof of Theorem \ref{thm_trans}]
		First of all, remark that in Theorem \ref{thm_ttype_closure2}.2, we already established that the sequence of operators $T_p^{W|Y}$, $p \geq p_1$, from (\ref{eq_defn_final_seq_of_operar}), forms a Toeplitz operator with weak exponential decay with respect to $X$ of type $W|Y$, and the identity $[T_p^{W|Y}]_0 = 0$ holds.
		From this, Theorem \ref{thm_ttype_closure2}.6 and (\ref{eq_dp_tp_relat}), we obtain that the sequence of operators $D_p$, $p \in \nat$, form a Toeplitz operator with exponential decay of type $X|Y$, and the identity $[D_p]_0 = 0$ holds.
		Similarly, we see that it suffices to prove that under the assumption (\ref{eq_comp_vol_omeg}) and $dv_W = dv_{g^{TW}}$, we have $[T_p^{W|Y}]_1 = 0$, $[T_p^{W|Y}]_2 = 0$, the polynomial $[T_p^{W|Y}]_3$ has degree $1$, and for any $n \in (N^{W|Y})^{(1, 0)}$, we have 
		\begin{equation}\label{eq_needed_for_the_proof}
			[T_p^{W|Y}]_3 \cdot n = \frac{1}{8 \pi}  \frac{\partial}{\partial n} \big( \textbf{r}^X - \textbf{r}^W \big) \cdot {\rm{Id}}_{F} - \frac{1}{2 \pi \imun} \nabla^{\enmr{F}}_{n} \big( \Lambda_{\omega} [R^F] - \Lambda_{\iota_2^* \omega} [ R^F ] \big).
		\end{equation}
		Let us now establish all those statements. 
		We assume in what follows (\ref{eq_comp_vol_omeg}) and $dv_W = dv_{g^{TW}}$.
		\par 
		To simplify further presentation, we define $f \in \ccal^{\infty}(W)$, $g \in \ccal^{\infty}(Y)$, as follows  
		\begin{equation}
			f := \Big[ \frac{1}{p^{n - l}} A_p^{X|W} \Big]_{1}, \qquad 
			g := \Big[ \frac{1}{p^{l - m}} A_p^{W|Y} \Big]_1 - \Big[ \frac{1}{p^{n - m}} A_p^{X|Y} \Big]_1. 
		\end{equation}
		Remark that both $f$ and $g$ take real values due to Theorem \ref{thm_ap_exp_two_terms}, cf. also Remark \ref{rem_ap_self_adj}.
		From Theorem \ref{thm_ap_exp_two_terms}, we obtain the following identities
		\begin{equation}\label{eq_f_min_g_ident}
			f = - g = \frac{1}{8 \pi} \Big( \textbf{r}_{y_0}^{X} - \textbf{r}_{y_0}^{W} \Big)
			-
			\frac{1}{2 \pi \imun} \Big( 
				\Lambda_{\omega} [ R^F ]
				-
				\Lambda_{\iota_2^* \omega} [ R^F ]
			\Big).
		\end{equation}
		Clearly, from Lemmas \ref{lem_inverse_toepl} and \ref{lem_def_alt_expr}, the sequences of operators 
		$T_{p, 1}^W := B_p^{W} - ( \frac{1}{p^{n - l}} (A_p^{X|W})^*)^{-1}$, 
		$T_{p, 2}^Y := B_p^{Y} -  ( \frac{1}{p^{n - m}} A_p^{X|Y} ) \circ (  \frac{1}{p^{l - m}} A_p^{W|Y} )^{-1}$, $p \in \nat$, form Toeplitz operators with exponential decay, and we have
		\begin{equation}\label{eq_toepl_last_clac_terms}
			[T_{p, 1}^W]_0 = 0, \qquad [T_{p, 1}^W]_0 = f, \qquad [T_{p, 2}^Y]_0 = 0, \qquad [T_{p, 2}^Y]_0 = g.
		\end{equation}
		We now denote
		\begin{equation}\label{eq_defn_tp0wy}
			T_{p, 0}^{W|Y}
			:=
			T_{f, p}^W
			\circ
			\res_W \circ \ext_p^{X|Y} 
			+
			\res_W \circ \ext_p^{X|Y} 
			\circ
			T_{g, p}^Y.
		\end{equation}
		From Corollary \ref{cor_comp_exp_bound}, Lemma \ref{lem_def_alt_expr} and (\ref{eq_toepl_last_clac_terms}), we deduce that the Schwartz kernels $T_{p}^{W|Y}(x, y)$, $T_{p, 0}^{W|Y}(x, y)$; $x \in W$, $y \in Y$, of  $T_{p}^{W|Y}$, $T_{p, 0}^{W|Y}$, evaluated with respect to $dv_Y$, are related by 
		\begin{equation}\label{eq_final_eq_12221}
			\Big| 
				T_{p}^{W|Y}(x, y)
				-
				\frac{1}{p}
				T_{p, 0}^{W|Y}(x, y)
			\Big|
			\leq 
			C p^{m - 2} 
			\cdot 
			\exp \big(- c \sqrt{p} \cdot \dist_X(x, y) \big).
		\end{equation}
		\par 
		From Lemmas \ref{lem_comp_poly}, \ref{lem_toepl_tay_typexynew}, \ref{eq_ex_res_epx_2terms} and (\ref{eq_defn_tp0wy}), we see that there are polynomials $J_{0, 0}^{W|Y}(Z_W, Z'_Y)$, $J_{0, 1}^{X|Y}(Z_W, Z'_Y)$, $Z_W = (Z_Y, Z_{N^{W|Y}})$, $Z_Y, Z'_Y \in  \real^{2m}$, verifying
		\begin{equation}\label{eq_j00_j11_final}
		\begin{aligned}
			&
			J_{0, 0}^{W|Y}
			:= 
			J_{0, f}^{W|Y} \cdot J_{0, Res}^{W|Y}
			+
			J_{0, Res}^{W|Y} \cdot J_{0, g}^{Y|Y},
			\\
			&
			J_{0, 1}^{W|Y}(Z_W, 0)
			:=
			\sum_{i = 0}^{1} \mathcal{K}_{l, l}[J_{i, f}^{W|Y} , J_{1 - i, Res}^{W|Y}](Z_W, 0)
			+
			\sum_{i = 0}^{1} \mathcal{K}_{l, m}^{EP}[J_{i, Res}^{W|Y}, J_{1-i, g}^{Y|Y}](Z_W, 0),
		\end{aligned}
		\end{equation}
		such that for $F_{0, r}^{W|Y} := J_{0, r}^{W|Y} \cdot \mathscr{E}_{m, l}$, $r = 0, 1$, the following holds.
		There are $\epsilon, c, C, Q > 0$, $p_1 \in \nat^*$, such that for any $y_0 \in Y$, $p \geq p_1$,  $Z_W \in \real^{2l}$, $|Z_W|, |Z'_Y| \leq \epsilon$, the following bound holds
			\begin{multline}\label{eq_tpy_defn_exp_tayxynewaaaaaba}
				\bigg| 
						\frac{1}{p^m} T_{p, 0}^{W|Y} \big(\psi^{W|Y}_{y_0}(Z_W), \phi^{Y}_{y_0}(Z'_Y) \big)
						\\
						-
						\sum_{r = 0}^{1}
						p^{-\frac{r}{2}}						
						F_{0, r}^{W|Y}(\sqrt{p} Z_W, \sqrt{p} Z'_Y) 
						\kappa_{\psi}^{W|Y}(Z_W)^{-\frac{1}{2}}
						\kappa_{\phi}^{Y}(Z'_Y)^{-\frac{1}{2}}
				\bigg|
				\\
				\leq
				C p^{- 1}
				\Big(1 + \sqrt{p}|Z_W| + \sqrt{p} |Z'_Y| \Big)^{Q} \exp \big(- c \sqrt{p} ( |Z_Y - Z'_Y| + |Z_{N^{W|Y}}| )\big).
		\end{multline}
		From Lemmas \ref{lem_toepl_tay_typexynew}, \ref{eq_ex_res_epx_2terms}, (\ref{eq_j1_expl_form}) and (\ref{eq_j00_j11_final}), we deduce that for $Z_W = (0, Z_{N^{W|Y}})$, $Z_{N^{W|Y}} \in \real^{2(l - m)}$,
		\begin{equation}\label{eq_j00finaifinal}
			J_{0, 0}^{W|Y}
			:= 
			f + g,
			\qquad
			J_{0, 1}^{W|Y}(Z_W, 0)
			:=
			\nabla^{\enmr{F}}_{\frac{\partial}{\partial z_W}} f.
		\end{equation}
		\par 
		Now, from (\ref{eq_final_eq_12221}) and (\ref{eq_tpy_defn_exp_tayxynewaaaaaba}), we deduce that $[T_{p}^{W|Y}]_1 = 0$.
		Remark now that by (\ref{eq_f_min_g_ident}), we have $f + g = 0$.
		From this, (\ref{eq_final_eq_12221}), (\ref{eq_tpy_defn_exp_tayxynewaaaaaba}) and (\ref{eq_j00finaifinal}), we deduce that $[T_{p}^{W|Y}]_2 = 0$.
		Now, finally, from (\ref{eq_final_eq_12221}), (\ref{eq_tpy_defn_exp_tayxynewaaaaaba}), (\ref{eq_j00finaifinal}) and the last part of Lemma \ref{lem_toepl_tay_type}, we deduce (\ref{eq_needed_for_the_proof}), which finishes the proof as we explained before (\ref{eq_needed_for_the_proof}).
	\end{proof}
	\par 
	Let us now generalize Theorem \ref{thm_trans} to the tower of embeddings of an arbitrary length.
	We fix a tower of embeddings $Y \xhookrightarrow[]{\iota_1} W_1 \xhookrightarrow[]{\iota_2} \cdots \xhookrightarrow[]{\iota_r} W_r \xhookrightarrow[]{\iota_{r+1}} X$, $\iota := \iota_{r+1} \circ \cdots \circ \iota_1$, and volume forms $dv_{W_i}$ on $W_i$, $i = 1, \ldots, r$, verifying assumptions, similar to (\ref{eq_vol_comp_unif}) with respect to the metric $g^{TW_i}$ induced by $g^{TX}$.
	We assume that the triples $(X, W_r, g^{TX})$, $\cdots$, $(W_{i+1}, W_i, g^{TW_{i + 1}})$, $(W_1, Y, g^{TW_1})$, $i = 1, \ldots, r - 1$, are of bounded geometry in the sense of Definition \ref{defn_bnd_subm}.
	\begin{cor}\label{thm_trans_get}
		The sequence of operators 
		\begin{equation}
			D_{p, r} := \ext_p^{X|Y} - \ext_p^{X|W_r} \circ  \ext_p^{W_r|W_{r - 1}}  \circ \cdots \circ  \ext_p^{W_2|W_1} \circ \ext_p^{W_1|Y}, \quad p \in \nat,
		\end{equation}
		forms a Toeplitz operator with exponential decay of type $X|Y$.
		Moreover, we have $[D_{p, r}]_0 = 0$. 
		Also, under assumptions (\ref{eq_comp_vol_omeg}) and $dv_{W_i} = dv_{g^{TW_i}}$, $i = 1, \ldots r$,  we have $[D_{p, r}]_1 = 0$, $[D_{p, r}]_2 = 0$ and $[D_{p, r}]_3 \in \ccal^{\infty}_{b}(Y, (N^{X|Y})^{(1, 0)*} \otimes \enmr{\iota^* F})$ for $n \in (N^{X|Y})^{(1, 0)}$, we have
		\begin{multline}
			[D_{p, r}]_3 \cdot n = \sum_{i = 1}^{r} \bigg\{ \frac{1}{8 \pi} \frac{\partial}{ \partial n_i} \cdot \big( \textbf{r}^{W_{i+1}} - \textbf{r}^{W_{i}}  \big) 
			\\
			+
			 \frac{\imun}{2 \pi }
			 \nabla^{\enmr{F}}_{n_i} \Big( \Lambda_{(\iota^{i+1})^* \omega} [R^F] - \Lambda_{(\iota^{i})^* \omega} [ R^F ] \Big) \bigg\},
		\end{multline}
		where we denoted $W_{r+1} := X, W_{0} := Y$; $\iota^i : W_i \to X$ is defined as $\iota^i := \iota_{r} \circ \cdots \circ \iota_{i + 1}$, and $n_i := P_N^{W_i | W_{i - 1}} n$, $i = 1, \ldots, r$.
	\end{cor}
	\begin{proof}
		Let us rewrite $D_{p, r}$ in the following way
		\begin{equation}
			D_{p, r} := \ext_p^{X|Y} - \ext_p^{X|W_r} \circ \ext_p^{W_r|Y} 
			+
			\ext_p^{X|W_r} \circ \Big( \ext_p^{W_r|Y} -  \ext_p^{W_r|W_{r - 1}}  \circ \cdots \circ  \ext_p^{W_2|W_1} \circ \ext_p^{W_1|Y} \Big).
		\end{equation}
		Now, the result follows directly from Theorems \ref{thm_trans} and \ref{thm_ttype_closure2}.6 by induction.
	\end{proof}
	\begin{proof}[Proof of Theorem \ref{thm_as_trans}]
		It follows directly from Theorems \ref{thm_trans}, \ref{thm_ttype_as}, Remark \ref{rem_ttype_as}a) and Proposition \ref{prop_c_1c_2_calc}.
	\end{proof}

\bibliography{bibliography}

		\bibliographystyle{abbrv}

\Addresses

\end{document}

%% file: packages.tex
\usepackage{times}

\usepackage{enumitem}
\usepackage[utf8]{inputenc}
\usepackage{commath}
\usepackage{amsthm, amscd}
\usepackage{amsmath}
\usepackage{amssymb}
\usepackage{cite}
\usepackage{setspace}
\usepackage{tasks}

\usepackage{bigints}
\usepackage{comment}
\usepackage{mathtools}
\usepackage{mathrsfs}
\usepackage{fancyhdr}

\allowdisplaybreaks[4]

\numberwithin{equation}{section}
\usepackage{etoolbox}
\patchcmd{\thebibliography}
  {\settowidth}
  {\setlength{\itemsep}{0pt plus -10pt}\settowidth}
  {}{}
\apptocmd{\thebibliography}
  {
  }
  {}{}

%% file: theorem.tex
\makeatletter
\newtheorem*{rep@theorem}{\rep@title}
\newcommand{\newreptheorem}[2]{%
\newenvironment{rep#1}[1]{%
 \def\rep@title{#2 \ref{##1}}%
 \begin{rep@theorem}}%
 {\end{rep@theorem}}}
\makeatother

\theoremstyle{theorem}

\newreptheorem{theorem}{Theorem}
\newtheorem{thm}{Theorem}[section]
\newtheorem*{thm*}{Theorem}
\theoremstyle{definition}
\newtheorem{prop}[thm]{Proposition}
\newtheorem*{prop*}{Proposition}
\newtheorem{defn}[thm]{Definition}
\newtheorem{lem}[thm]{Lemma}
\newtheorem{cor}[thm]{Corollary}
\newtheorem*{cor*}{Corollary}
\theoremstyle{remark}

\newtheorem{rem}[thm]{Remark}

%% file: title.tex
\title{\vspace*{-0.5cm}Complex embeddings, Toeplitz operators and \\
transitivity of optimal holomorphic extensions.}

\author
{Siarhei Finski
}

\date{}

%% file: paddings.tex
\usepackage[%
    left=1in,%
    right=1in,%
    top=1.1in,%
    bottom=0.8in,%
    paperheight=11in,%
    paperwidth=8.5in%
]{geometry}

%% file: commands.tex
\newcommand{\imun} {\sqrt{-1}}

\newcommand{\res}{{\rm{Res}}}
\newcommand{\ext}{{\rm{E}}}

\newcommand{\comp}{\mathbb{C}}
\newcommand{\real}{\mathbb{R}}

\newcommand{\nat}{\mathbb{N}}

\newcommand{\dist}{{\rm{dist}}}

\newcommand{\enmr}[1]{\text{End}{(#1)}}

\newcommand{\ccal}{\mathscr{C}}

\renewcommand{\Im}{\operatorname{Im}}
\newcommand{\scal}[2]{\big< #1, #2 \big>}

\makeatletter
\DeclareFontFamily{OMX}{MnSymbolE}{}
\DeclareSymbolFont{MnLargeSymbols}{OMX}{MnSymbolE}{m}{n}
\SetSymbolFont{MnLargeSymbols}{bold}{OMX}{MnSymbolE}{b}{n}
\DeclareFontShape{OMX}{MnSymbolE}{m}{n}{
    <-6>  MnSymbolE5
   <6-7>  MnSymbolE6
   <7-8>  MnSymbolE7
   <8-9>  MnSymbolE8
   <9-10> MnSymbolE9
  <10-12> MnSymbolE10
  <12->   MnSymbolE12
}{}
\DeclareFontShape{OMX}{MnSymbolE}{b}{n}{
    <-6>  MnSymbolE-Bold5
   <6-7>  MnSymbolE-Bold6
   <7-8>  MnSymbolE-Bold7
   <8-9>  MnSymbolE-Bold8
   <9-10> MnSymbolE-Bold9
  <10-12> MnSymbolE-Bold10
  <12->   MnSymbolE-Bold12
}{}

\let\llangle\@undefined
\let\rrangle\@undefined
\DeclareMathDelimiter{\llangle}{\mathopen}%
                     {MnLargeSymbols}{'164}{MnLargeSymbols}{'164}
\DeclareMathDelimiter{\rrangle}{\mathclose}%
                     {MnLargeSymbols}{'171}{MnLargeSymbols}{'171}
\makeatother